\documentclass[12pt]{article}

\usepackage{amsbsy,amsfonts,amsmath,amssymb,amsthm,bbm,enumerate,euscript,graphicx}

\usepackage{xcolor}
\newcommand{\blue}{black}

\renewcommand{\emptyset}{\varnothing}

\newtheorem{prop}{Proposition}[section]
\newtheorem{theo}[prop]{Theorem}

\newtheorem{lemm}[prop]{Lemma}

\newtheorem{assu}[prop]{Condition}
\newcommand{\Po}{{\cal P}}
\newcommand{\cR}{{\cal R}}
\newcommand{\cK}{{\cal C}}
\newcommand{\Q}{{\mathbb{Q}}}
\newcommand{\cM}{{\cal M}}

\def\N{\mathbb{N}}
\def\Z{\mathbb{Z}}

\def\R{\mathbb{R}}
\def\E{\mathbb{E}}

\def\bM{{\bf M}}

\def\cI{\mathcal{I}}
\def\cN{\mathcal{N}}

\def\cM{\mathcal{M}}

\def\cF{\mathcal{F}}

\def\Pr{\mathbb{P}}

\def\0{{\bf 0}}

\def\cB{{\cal B}}

\renewcommand{\E}{\mathbb E \,}

\newcommand{\tod}{\stackrel{{\cal D}}{\longrightarrow}}
\newcommand{\eqd}{\stackrel{{\cal D}}{=}}

\newcommand{\dist}{{\rm dist}}

\newcommand{\eqco}{\setcounter{equation}{0}}
\newcommand{\allco}{\eqco   }
\renewcommand{\H}{{\cal H}}

\newcommand{\Cov}{{\rm Cov}}

\newcommand{\Var}{{\rm Var}}

\newcommand{\K}{{\cal K}}

\newcommand{\F}{{\cal F}}

\newcommand{\tN}{{\tilde{N}}}

\newcommand{\oF}{{\overline{F}}}

\newcommand{\txi}{{\tilde{\xi}}}

\newcommand{\tx}{{\tilde{x}}}
\newcommand{\tily}{{\tilde{y}}}

\newcommand{\tbeta}{{\tilde{\beta}}}

\newcommand{\eps}{\varepsilon}
\def\bdm{\begin{displaymath}}
\newcommand{\edm}{\end{displaymath}}
\def\benu{\begin{enumerate}}
\def\eenu{\end{enumerate}}
\def\beqn{\begin{equation}}
\def\eeqn{\end{equation}}
\def\be{\begin{equation}}
\def\ee{\end{equation}}
\def\bea{\begin{eqnarray}}
\def\eea{\end{eqnarray}}
\newcommand{\bean}{\begin{eqnarray*}}
\newcommand{\eean}{\end{eqnarray*}}
\newcommand{\bear}{\begin{eqnarray}}
\newcommand{\eear}{\end{eqnarray}}

\renewcommand{\epsilon}{\varepsilon}

\def\R{\mathbb{R}}

\renewcommand{\P}{{\mathbb P}}

\def\tz{\tilde{z} }
\def\ty{\tilde{y} }
\def\tu{\tilde{u} }

\def\qed{\hfill\hbox{${\vcenter{\vbox{
    \hrule height 0.4pt\hbox{\vrule width 0.4pt height 6pt
    \kern5pt\vrule width 0.4pt}\hrule height 0.4pt}}}$}}

\def\la{{\lambda}}

\begin{document}
\title{\bf Leaves on the line and in the plane}

\author{
%Tobias M\"uller$^{1}$
%and
Mathew D. Penrose$^{1}$ \\
%\author{Mathew D. Penrose , {\normalsize{\em University of Bath}} }
{\normalsize{\em University of Bath}} }

 %\footnotetext{ $~^1$ Tobias contact details }
 \footnotetext{ $~^1$ Department of
Mathematical Sciences, University of Bath, Bath BA2 7AY, United
Kingdom: {\texttt m.d.penrose@bath.ac.uk} }
% \footnotetext{ $~^2$
%}

%Running title: {\bf } \\

%\footnotetext{ AMS classifications: 05C80, 60J85, 92D30 }

%\footnotetext{ Keywords: semi-homogeneous random digraph,
% giant component,branching process}

%\date{}
\maketitle

%\newpage
\begin{abstract} The {\em dead leaves model} (DLM) provides a random tessellation of $d$-space, representing the visible portions of fallen leaves on the ground when $d=2$.  For $d=1$, we establish formulae for the intensity, two-point correlations, and asymptotic covariances for the point process of cell boundaries, along with a functional CLT.  For $d=2$ we establish  analogous results for the random surface measure of cell boundaries, and also \textcolor{\blue}{determine} the  intensity of cells in a more general setting than in earlier work of Cowan and Tsang. We introduce a general notion of {\em dead leaves random measures} and  give formulae for means, asymptotic variances and functional CLTs for these measures; this has applications to various other quantities associated with the DLM. \\



\end{abstract}

%Add more motivation in Sec 1.1?  Criteria for 
%Condition \ref{measassu}?  Avram-B argument in $d=1$? 
%Measurability of $\chi$? Proof of (3.3a)? Better handshaking argument?
%Discuss concepts of exceptional times?
%More on `can be deduced from (6.3)' at end proof of Lemma 10.1?
%Justify cts sample paths for O-U process? See Lemma 10.2.

\section{Introduction}
\label{secintro}
\subsection{Overview}
\label{subsecoverview}
The {\em dead leaves model} (or DLM for short) in $d$-dimensional
space 
($d \in \N$), due originally to Matheron \cite{Matheron},
 is defined as follows \cite{Bordenave,Serra}.
Leaves fall at random onto the ground and the visible parts of
leaves on the ground (i.e., those parts that are not covered by later
arriving leaves) tessellate $\R^d$ (typically with $d=2$; see Figure 1).
Motivation for studying the DLM   
from the modelling of
natural images, and from materials science, is discussed
in \cite{Bordenave}, \cite{Jeulin}, \textcolor{\blue}{
 and \cite{GG},} for example.  
The DLM provides a natural way of generating a stationary random
tessellation of the plane with non-convex cells having
possibly curved boundaries. 

 To define the model more formally,  
let $\Q$ be a probability measure on the space $\cK$ of compact 
sets in $\R^d$ (equipped with the $\sigma$-algebra generated
by the Fell topology, as defined in  Section \ref{subsecnotation} below),
 assigning  strictly positive measure to the collection of
sets in $\cK$ having non-empty interior.
A collection of `leaves'
arrives as 
%a Poisson point process in $\R^d \times \R \times \cK$
%with intensity measure  $\H_d \times \H_1 \times \cK$,
%where $\H_d$ denotes $d$-dimensional Lebesgue measure.
an independently marked homogeneous Poisson point process 
%$\Po = \{(x_i,t_i), i \geq 1 \}$
$\Po = \sum_{i =1}^\infty \delta_{(x_i,t_i)}$
in $\R^d \times \R $ of unit intensity with marks $(S_i)_{i \geq 1}$
taking values in $\cK$ with common mark distribution $\Q$.
Each point $(x_i,t_i,S_i)$ of this
% Poisson process 
marked point process
is said to have {\em arrival time} $t_i$ and the  associated leaf covers
the region $S_i+x_i \subset \R^d$ from that time onwards (where
$S+x:= \{y+x:y \in S\}$).
%(where $(x,t)$ denotes the location in space-time of the point
%and $S$ the associated mark) represents a leaf that occupies
%the region $x+S \subset \R^d$, arriving at time $t$.
% in space-time 
%with intensity Leb$^{d+1} \times \Q$.
At a specified time, say time 0, and at spatial location $x \in \R^d$,
the most recent leaf to arrive before time 0 that covers location
$x$ is said to be {\em visible} (or {\em exposed}) at $x$.
\textcolor{\blue}{For each $i \in \N$,
 the {\em visible portion} of leaf $i$ at time 0
is the set of sites $x \in \R^d$ such that leaf $i$ is visible 
at $x$ at time 0.}
 The
connected components of visible portions
of leaves (at time 0) form a tessellation of $\R^d$, which
we call the {\em DLM tessellation}. 

\begin{figure}[!h]
\label{fig1}
\center
\includegraphics[width=17cm]{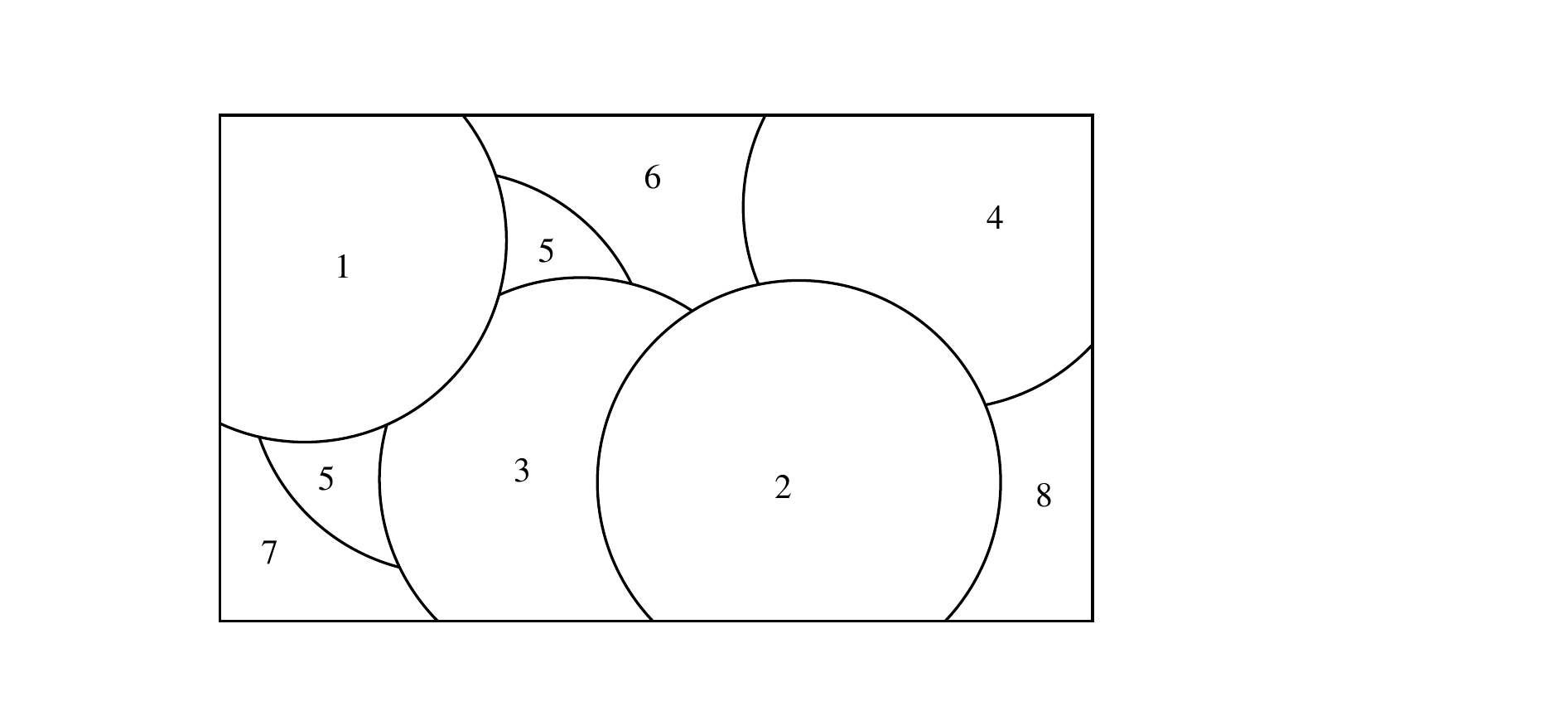}
\caption{A realization of the DLM tessellation, restricted to a window,
 where all the leaves are
unit disks. The numbers indicate the reverse order of arrival
of the leaves visible within the window. In this paper we view the two visible
 components of leaf 5 as being
 separate components of the DLM tessellation}
\end{figure}

Properties of the DLM itself
are discussed in \cite{Bordenave,Cowan,Jeulin,Serra},
 while percolation on the DLM tessellation
has been considered in
\cite{Schramm,Muller}. In some of these
works the authors call the DLM the `confetti' model.
Note that in the present paper, all cells of our DLM tessellation
are connected; the tessellation where cells are taken to be
the visible portions of leaves (rather than their connected 
components) is also of interest, and is considered in
some of the works just mentioned.

In this paper we consider the DLM for $d=1$ and for $d=2$. 
For $d=1$, we develop the second order theory for the point process
of cell boundaries. That is, we determine its second factorial moment
measure, two-point correlation functions, asymptotic variance
and a spatial central limit theorem (CLT).
% for a large window.
Moreover, we can and do consider the point process 
 to  be {\em evolving}  as leaves continue
to rain down; we establish a {\em functional} CLT
showing that the (evolving) number of cells in a large window approximates
to  an Ornstein-Uhlenbeck process.
For $d=2$ we carry out a similar programme (asymptotic variance
and functional CLT)
 for the surface measure of cell
boundaries within a large window.
We state our results for $d=1$ in Section \ref{seconedim},
and for $d=2$ in
Section \ref{ctshidim}.

For general $d$, we also develop (in Section \ref{secDLRM}) 
 an extension of the DLM which we call
the {\em dead leaves random measure} (DLRM). 
Suppose now that each point $(x_i,t_i)$ of $\Po$ is  marked with not only
a random closed set $S_i$ as before, but also a {\em random measure}
  $M_i$, for example the surface measure of $\partial S_i$.   The DLRM
 at time $t$ is  the sum, over those $i$ with $t_i \leq t$,
  of the measures $M_i +x_i := M_i (\cdot +(-x_i) )$, restricted to the
 complement of leaves  arriving between times $t_i$ and $t$. 
We give results on its intensity, limiting covariances and functional CLTs.
  This provides a general framework  from which we may deduce the
results already mentioned as special cases, and is also
 applicable to other DLM functionals and to variants of the DLM
including the {\em colour DLM} and {\em dead leaves random function},
as discussed in Section \ref{secDLRM}. 

As well as the new results already mentioned, we provide some extensions
to known first-order results, giving the intensity
of cell boundaries in
$d=1$, and the intensity of cells in $d=2$.
These were already in the literature
% previously known
in the special cases  when
all of the leaves are connected  
(for $d=1$; see \cite{Matheron}), and when they
all have the same shape (for $d=2$; see \cite{Cowan}).
Finally, for $d=1$ we discuss the distribution of cell sizes; essentially
this was given in \cite{Matheron} but we give a bit more detail here.

%In this paper we consider the DLM for $d=1$ and for $d=2$. 
%We shall derive expressions (in terms of the leaf distribution $\Q$)  for
%the intensity of cells (that is, the average
% number of cells per unit volume) in the DLM tessellation.
%  For $d=1$ we also derive formulae for the
%distribution of the size of the cell containing the origin, for the
%size of a `typical' cell,   and for the second factorial moment measures
%and two-point correlation functions for the  point process of cell boundaries.
%%For $d=2$ 
%We also derive formulae for the asymptotic variance, and central limit
% theorems (CLTs), for the total number of cells in a large window
%(for $d=1$) and for the total length of cell boundaries 
% within a large window (for $d=2$).  In fact, for the {\em evolving} DLM
% (that is, a DLM tessellation that evolves in time as the leaves continue
% to rain down) we establish a {\em functional} CLT showing that 
%  the evolving value of each of these two quantities approximates (for
% a large window) to an Ornstein-Uhlenbeck process.

%Our formulae for the intensity of cells were already known in
%the literature in a more restricted setting, namely  when
%all of the leaves are connected  
%(in the case $d=1$), and when they
%all have the same shape (in the case $d=2$; see \cite{Cowan}). We are not 
%aware of any previous work on the other quantities 
%considered here.

Exact formulae for second moment measures
and for two-point correlation functions are rarely available
for non-trivial point processes in Euclidean space, and
one contribution  of this work is to provide
such formulae in one class of such models.
Our CLTs could be useful for providing  confidence
intervals for parameter estimation in the DLM and related models.
Our general functional CLT shows that the DLRMs are a class of off-lattice
interacting particle systems for which the limiting process
of fluctuations can be identified explicitly as an Ornstein-Uhlenbeck
 process. In earlier works \cite{Surveys,Qi}, functional CLTs
were obtained for certain general classes of particle 
systems but without any characterisation of the limit process.

The proof of some of our results in $d=2$ uses
the  following results. For rectifiable 
%smooth
  curves $\gamma$ and $\gamma'$ in $\R^2$
of respective lengths $|\gamma|$ and $|\gamma'|$, let $N(\gamma,\gamma')$
(respectively $\tN(\gamma,\gamma')$)
denote the number of times they cross each other (respectively, touch each
other). If one integrates $N(\gamma,T(\gamma'))$  
(resp. $\tN(\gamma,T(\gamma'))$ over all rigid motions $T$ of
the plane, one obtains a value of $4|\gamma| \times |\gamma'|$
(resp. zero). 
We discuss these results, which are related to the classic 
 {\em Buffon's needle} problem, 
 in Section \ref{secBuffon}.

Since we prefer to work with positive rather than negative
times, in this paper we often consider a {\em time-reversed}
version of  the DLM
where,
 for each site $x \in \R^d$, the {\em first}
leaf to arrive at $x$ after time 0 is taken to be visible at $x$.
 Imagine leaves falling
onto a glass plate which can be observed from below,
starting from time 0. Clearly this gives a tessellation
with the same distribution as the original DLM. 
This observation dates back at least to \cite{JeulinMik}. In \cite{KT},
it is the basis for a perfect simulation algorithm for the DLM.
The time-reversed DLM is illustrated for $d=1$ in Figure 2.

We shall state our results  in Sections \ref{seconedim}-\ref{secBuffon},
% and prove them in Sections \ref{secproofsDLRM}-\ref{secConvD}.
 and prove them in Sections \ref{secproofsDLRM}-\ref{secpftwodim}.

\begin{figure}[!h]
\label{fig2}
\center
\includegraphics[width=15cm]{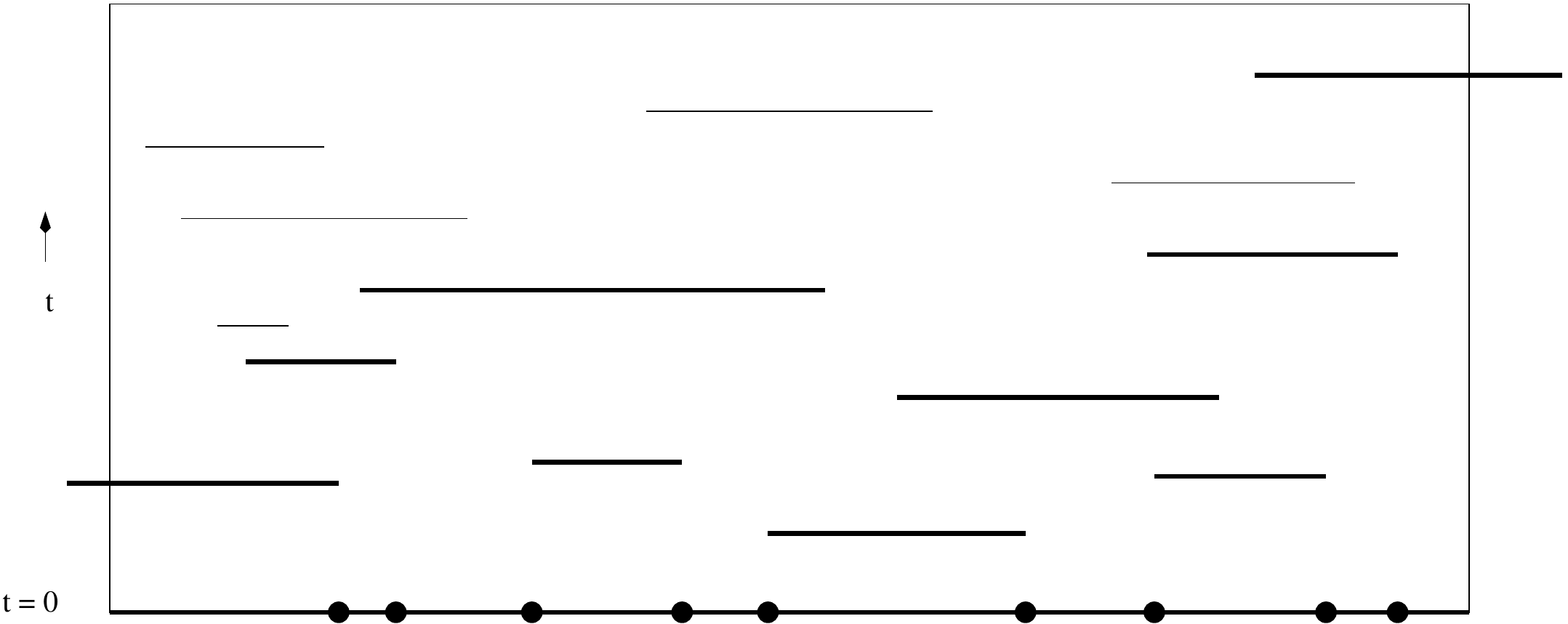}
\caption{A realization of the time-reversed DLM in $d=1$, with
time as the vertical coordinate. The leaves are intervals
of variable length. Those leaves
arriving in a given rectangular region of space-time are shown,
and the boundary points of the induced 
1-dimensional DLM tessellation are shown at the bottom. The 
\textcolor{\blue}{thicker lines}
represent leaves which are at least partially visible}
\end{figure}

\subsection{Motivation}

We now discuss further the motivation for considering the DLM.
Since we consider the case $d=1$ in some detail, we discuss the  
motivation for this at some length.

The phrase `leaves on the line' entered the British folklore in
the early 1990s as a corporate justification for delays on the railways.
%notorious excuse used
%used by rail operators.
%for delays.
 To quote Wikipedia,
this phrase is `a standing joke ... seen by members of the public
 who do not understand the problem as an excuse for poor service.' 
This paper is a mathematical contribution to said public
understanding.
%In other words, if you think the title of this paper is funny, you are
%showing your ignorance; this paper may be viewed as a contribution
%to said public understanding.

%As further motivation, note that 
A one-dimensional
DLM is obtained whenever one takes the restriction of a
 higher-dimensional DLM (in $\R^d$, say, with $d > 1$)
to a specified one-dimensional  subspace of $\R^d$. 
Such restrictions are considered in \cite{Serra} and \cite{Bordenave}.

 Moreover, the one dimensional DLM is quite   natural 
in its own right. For example, to quote
\cite{Serra}, `Standing at the beginning of a forest, one
sees only the first few trees, the others being hidden behind.' 
A less pleasant interpretation,
% suggested to the author by Pieter Trapman in 2016,
is that if  an explosion takes place in a crowded spot,
one might be interested in the number of people directly 
exposed to the blast (rather than shielded by others). 
In these two interpretations, the `time' dimension in fact represents
a second spatial dimension.

In another interpretation of the one-dimensional time-reversed DLM, consider
a rolling news television or radio station. Suppose news stories
arise  as a homogeneous Poisson process in the product space $\R  \times \R_+$,
where the first coordinate represents the time at which the story starts,
 and the second
coordinate represents its `newsworthiness' (a lower score representing
a more newsworthy story), and each story is active for
 a random duration. Suppose
at any given time the news station presents the most
newsworthy currently active story. Then the stories presented
form a sequence of intervals, each story presented continuing until it finishes
or is   superseded by a more newsworthy story. The continuum time series
 of stories
presented forms a DLM in time rather than in space, with `newsworthiness'
taking the role of `time' in the original DLM. One can imagine
a similar situation with, for example, the time series of 
top-ranked tennis or golf players.

In these interpretations, we are taking the trajectory of a news story's
newsworthiness, or a tennis player's standard of play,
 to be flat but of possibly
random duration. It would be interesting in future work to allow for
other shapes of trajectory. If the trajectory is taken to be a fixed
wedge-shape, then the sequence of top-ranked stories/players is the
sequence of {\em maximal points} (actually minimal points in our formulation),
 which has been considered, for example,
in \cite{LSY,Yukich}.

The two-dimensional DLM has received considerable attention in
applications; see \cite{Bordenave} and references therein.
For any two-dimensional image of a three-dimensional particulate
material with opaque particles, the closest particles
obscure those lying behind, and the  DLM models this phenomenon.
See  for example \cite{Jeulin,JeulinMic} for applications to analysis
of images of powders. Jeulin \cite{JeulinSig,Jeulin}  
extends to the DLM to a {\em dead leaves random function} model
for further flexibility in modelling greyscale images, and 
some of our results are applicable to this model. See Section \ref{secDLRM}.

 Another reason to study the DLM, in arbitrary dimensions, is as an analogue
to the {\em car parking} model of random sequential adsorption.
In the one-dimensional and infinite-space version of the latter model,
 unit intervals (`cars') arrive 
 at locations in space-time given
by the points of a homogeneous Poisson process in $\R \times \R_+$.
Each car is  accepted if the position (a unit interval) where it arrives
does not intersect any previously accepted cars.
Ultimately, at time infinity one ends up with a random maximal 
(i.e., saturated) packing  of
$\R$ by unit intervals.
%   By a celebrated  result of R\'enyi
%\cite{Renyi}, the  proportion of $\R$ that is ultimately covered
% is given by an integral which comes to $0.748$ (to 3 decimal places). 
%This can be viewed as the  intensity of the point
%process of midpoints of accepted intervals. 
The higher-dimensional version of the  car parking model has
also been studied, for example in \cite{GP} and references
therein.

The problem of covering can be viewed as in some sense dual to that of
packing (see for example \cite{Rogers,Zong}),
 and in this sense the (time-reversed) DLM is dual to the car parking model;
in each case, objects of finite but positive size (cars/leaves)
arrive sequentially at random in $d$-space, and are accepted in a greedy manner
subject to  a hard-core constraint  (for the packing) or 
a visibility  constraint (for the DLM). 
%In the one-dimensional DLM the set of
% accepted intervals provides a random covering of $\R$.
%By analogy with R\'enyi's result, one might ask about the
%intensity of the (stationary) point process of midpoints of
%accepted intervals. In the case of unit intervals this
%turns  out to be 2.
% which was implicit in the existing literature but not very clearly spelt out.
% See Proposition \ref{th:main}  below.

\subsection{Notation and terminology}
\label{subsecnotation}
Let $\cB^d$ denote the Borel $\sigma$-algebra on $\R^d$.
For $k \in \{0,1,\ldots,d\}$, let
 $\H_k$
denote the $k$-dimensional Hausdorff measure of sets in $\R^d$. 
This is a measure on $(\R^d,\cB^d)$.
%the Borel $\sigma$-algebra.
In particular,
  $\H_0$ is the counting measure and
$\H_d$ is Lebesgue measure.  See \cite{LP}.

Let $\|\cdot\|$ denote the Euclidean norm in $\R^d$.
Given $x \in \R^d$, let $\delta_x$ denote the Dirac measure
at $x$, i.e. $\delta_x(A) = 1 $ if $x \in A$, otherwise $\delta_x(A)=0$.
For $r >0$, let $B(r): = \{y \in \R^d:\|y\| \leq r\}$,
the closed Euclidean ball of radius $r$  centred on the origin.
Set $\pi_d:= \H_d(B(1))$, the Lebesgue measure of the unit ball in
$d$ dimensions.

%We say a set $\gamma \subset \R^2$ is a {\em $C^1$ curve}
%if there exists a continuously differentiable injective function 
%$\Gamma:[0,1] \to \R^2$ with non-zero derivative
%at all points in $[0,1]$  such that $\gamma = \Gamma([0,1])$.
%In this case we call  $\Gamma(0)$ and $\Gamma(1)$
%the {\em endpoints} of $\Gamma$. We say that a set 
%$\gamma \subset \R^2$ is a {\em piecewise $C^1$ curve}
%if  it is the finite union of $C^1$ curves $\gamma_1,\ldots,\gamma_k$
%such that $\gamma_i \cap \gamma_j$  
%
We say a set  $\gamma \subset \R^2$ is a
 {\em rectifiable curve}
%if $\gamma = \Gamma([0,1]$ for some
if there exists a continuous injective function 
$\Gamma:[0,1] \to \R^2$
such that
 $\gamma = \Gamma([0,1])$ and $\H_1(\gamma) < \infty$.
If moreover there exist $k \in \N \cup \{0\} $ and
numbers  $0= x_0 < x_1 < \cdots < x_{k+1} =1$, 
 such that for $1 \leq i \leq k+1$ the restriction
of $\Gamma$ to $[x_{i-1},x_i]$ is continuously
differentiable with derivative that is nowhere zero,
 we say that $\gamma$ is a {\em piecewise $C^1$ curve}.
We then refer to the points $\Gamma(x_1),\ldots,\Gamma(x_k)$
(where  $k$ is assumed to be taken  as small as possible)  
as the {\em corners} of $\gamma$. 
If we can take $k=0$
 (so there are no corners),
then we say $\gamma$ is a $C^1$ curve.
We say that $\Gamma(0)$ and $\Gamma(1)$ are the
{\em endpoints} of $\gamma$.
We define a {\em rectifiable Jordan curve} 
(respectively, a {\em piecewise $C^1$ Jordan curve})
 similarly to a rectifiable curve
(respectively,  piecewise $C^1$  curve)
 except that now  $\Gamma$
%similarly to
%a  rectifiable  curve, except that now  $\Gamma$
must satisfy $\Gamma(1)= \Gamma(0)$ but  be otherwise injective.

%$E \subset (0,1)$ and all $x \in [0,1] \E$
%the function is continuously 
% with non-zero derivative at all points
%where the $\Gamma$ is differentiable

%In this case we call  $\Gamma(0)$ and $\Gamma(1)$
%the {\em endpoints} of $\Gamma$. We say that a set 
%$\gamma \subset \R^2$ is a {\em piecewise $C^1$ curve}
%if  it is the finite union of $C^1$ curves $\gamma_1,\ldots,\gamma_k$
%such that $\gamma_i \cap \gamma_j$  

%We say a set $\gamma \subset \R^2$ is a
%{\em piecewise $C^1$ curve} if (i) it is connected
%and (ii)  there is a finite set $E \subset \gamma$ such that
%for $x \in \gamma \setminus E $, $x$  has a neighbourhood $U$ such that
%the set $\gamma \cap U$, after a suitable rotation, is
%the graph of a continuously differentiable function defined on an interval. 
%Assuming  we take the exceptional set $E$ to be minimal,
%we refer to the elements of $E$ as {\em corners} of $\gamma$.
%If $\gamma$ is a piecewise $C^1$ curve with no corners,
%%If we can take the exceptional set $E$ to be empty, 
%then we say $\gamma$ is a $C^1$ curve.
%

For $\sigma \geq 0$,
let  ${\cal N}(0,\sigma^2)$ denote a  normally distributed
random variable
having mean zero and variance $\sigma^2$ if $\sigma >0$, and
denote a random variable taking the value 0 almost surely
% the (one-dimensional) Dirac probability measure $\delta_0$ 
if $\sigma =0$.

We now review some concepts from the theory of point processes and
random measures that we shall
be using. See for example \cite{LP} or \cite{SW} for more details.

Let ${\bf M}$ be the space of \textcolor{\blue}{locally}
 finite measures on $(\R^d,\cB^d)$,
equipped with the smallest $\sigma$-algebra  that makes measurable all
of the functions from $\bM$ to $\R$ of the form $\mu \mapsto \mu(A)$, 
with $A \in \cB^d$. For $\mu \in \bM$ we shall often write
$|\mu|$ for $\mu(\R^d)$.
A {\em random measure} on $\R^d$ is a measurable function
from an underlying probability space 
%(denoted $(\Omega,\F,\P)$)
to  $\bM$, or equivalently, a measurable kernel from the
probability space to $(\R^d,\cB^d)$.
A random measure on $\R^d$ is said to be {\em stationary}
if its distribution is shift invariant, in which case 
 the expected value of the measure it assigns
to a Borel set $B \subset \R^d$ is proportional to the Lebesgue measure
of $B$; the constant of proportionality is called 
 the {\em intensity}
of the random measure.
%is defined to be the expected value of the measure it assigns
%to the unit cube.  

A random measure on $\R^d$ taking integer values is
called a {\em point process} on $\R^d$, and the notions
of intensity and stationarity for random measures carry through to
point processes. A point process is
%A point process in $\R^d$ is said to be {\em stationary} if its distribution
%is shift-invariant, in which case its {\em intensity} is defined to be
%the mean number of points  per unit volume. The point process is
said to be {\em simple} if it has no multiple points.  

The {\em second factorial moment measure} of a point process $\eta$
in $\R^d$
is a Borel measure $\alpha_2$ 
 on $\R^d \times \R^d$,
 defined at \cite[eqn (4.22)]{LP}. 
  For  disjoint Borel sets  $A,B \subset \R^d$,
we have
$
\alpha_2 (A \times B)= \E[ \eta(A) \eta(B)]. 
$
If $\eta$ is simple then
for $x,y \in \R^d$ with $x \neq y$, loosely speaking $\alpha_2(d(x,y))$ is
the probability of seeing a point of $\eta$ in $dx$
and another one in $dy$.
If $\alpha_2(d(x,y)) = \rho(y-x)dxdy$ for some Borel function
$\rho: \R^d \to \R_+$, then
the {\em pair correlation function} $\rho_2(\cdot)$ of the   point
process $\eta$ is defined by $\rho_2(z) := \rho(z) / \gamma^2$,
where $\gamma$ is the intensity of $\eta$. 
%at  \cite[p. 73]{LP}.

The {\em Fell topology} on $\cK$ is the topology 
generated by all sets of the form $\{F \in \cK: F  \cap G  \neq 
\emptyset\}$ with $G \subset \R^d$ open,  or of the form 
 $\{F \in \cK: F  \cap K  = \emptyset\}$ with $K \in \cK$.

A {\em random closed set} in $\R^d$ is a measurable map from a probability space
 to the space of closed sets in $\R^d$ equipped with
 the sigma-algebra generated by the Fell topology. 
See \cite[Definition 2.1.2]{SW} 
\textcolor{\blue}{or \cite[Definition 1.1.$1''$]{Molchanov}.}
\textcolor{\blue}{
A simple point process on
 $\R^d$ can also be interpreted as a random, locally finite
set of points in $\R^d$.
}

We now elaborate on the definition of  the DLM given already.
Let $\Q$ be a probability measure on $\cK$, which we call the
{\em grain distribution} of the model. 
 Assume $\Q$ assigns  strictly positive measure to the collection of
sets in $\cK$ having non-empty interior.
Let $\Po $ be a homogeneous Poisson process 
in $\R^d \times \R $ of unit intensity.
Write $ \Po   = \sum_{i=1}^\infty \delta_{(x_i,t_i)}$ with 
$(x_i,t_i)_{i \in \N}$ a sequence of random elements of 
$\R^d \times \R$. This can be done: see \cite[Corollary 6.5]{LP}.
Independently of $\Po$, let
$(S_i)_{i \geq 1}$ be a sequence of independent random elements of 
$\cK$ with common distribution $\Q$.
By the Marking theorem 
(see 
%\cite{LP}) 
 \textcolor{\blue}{(see \cite[Theorem 5.6]{LP})} 
the point process
$\sum_{i=1}^\infty \delta_{(x_i,t_i,S_i)}$ is
a Poisson process in $\R^d \times \R \times \cK$ with intensity
measure $\H_d \otimes \H_{1} \otimes \Q$.  

%We now define  our DLM tessellation more formally via its boundaries.
\textcolor{\blue}{For $A \subset \R^d$,  let 
$\overline{A}$ denote its closure and
$A^o$ its interior;  let
$\partial A := \overline{A} \setminus A^o$, the topological
boundary of $A$.}
The boundary of the DLM tessellation at time $t$, 
which we denote by $\Phi_t$, is
given by
\bea
\Phi_t :=
 \cup_{i: t_i \leq t} [ (\partial S_i + x_i) \setminus
\cup_{j: t_i < t_j  \leq t} (S_j^o + x_j)].
\label{Phindef}
\eea
The boundary of the the {\em time-reversed} DLM tessellation is denoted by
$\Phi$, and given by
\bea
\Phi := 
 \cup_{i: t_i \geq 0 } [ (\partial S_i + x_i) \setminus
\cup_{j: 0 \leq t_j  < t_i} (S_j^o + x_j)].
\label{Phidef}
\eea
The cells of our (time-reversed) DLM tessellation are then defined to be
the closures of the connected components of $\R^d \setminus \Phi$.
Clearly $\Phi_t$ has the same distribution as $\Phi$ for all $t$.
% In fact one should
%take the closures of these components since the convention is
%to take the cells of a random
%tessellation to be closed sets (see e.g. \cite{Bordenave}). 

%The probability measure $\Q$ on $\cK$ was already
%specified.
 Throughout, we let $S$ denote a random element
of $\cK$ with distribution $\Q$; that is, a measurable
function from an underlying probability space (denoted
$(\Omega,\F,\P)$) to $\cK$. For $x \in \R^d$ we set
\bea
\lambda := \E[ \H_d( S)]; ~~~~~~~
\lambda_x := \E[ \H_d( S \cup (S+x))].
\label{lambdadef}
\eea
\textcolor{\blue}{Clearly $0 < \lambda \leq \lambda_x \leq 2 \lambda$. We set}
\bea
R := \sup \{\|x\|: x \in S\},
\label{Rdef}
\eea
taking $R=0$ if $S$ is empty.
Observe that $2 \lambda - \lambda_x$ equals the expected value
of $\H_d(S \cap (S+x))$, sometimes called the {\em covariogram}
of $S$.  \textcolor{\blue}{ The function $\lambda_x$ will feature in certain
formulae for limiting covariances below. It also features in
certain formulae for limiting variances arising from the 
{\em Boolean model} (see for
example \cite{HM}). The Boolean model, that is
the random set $\cup_{i: 0 \leq t_i \leq \lambda} (S_i+ x_i)$ for some fixed $\lambda$, 
 is another fundamental model in stochastic geometry.
 }

Some of our \textcolor{\blue}{results require} the following
measurability condition.
%to hold.
\begin{assu}
\label{measassu}
 $\Q$ is such that $\H_{d-1}(\cdot \cap \partial S)$
is a random measure on $\R^d$. 
\end{assu}
For $d=1$, we shall show in  Lemma \ref{lemmeas} that
Condition \ref{measassu} can actually be
deduced from our earlier assumption that $S$ is a random element of $\cK$,
that is, a measurable map from a probability space to $\cK$.
However, we do not know whether this is also the case 
for $d \geq 2$.

Given $d$,
for $n  > 0$ let $W_n := [0,n^{1/d}]^d $, a cube of volume $n$ in $d$-space.
Let $\cR_0$ denote the  class of bounded measurable
 real-valued functions on $\R^d$ that are Lebesgue-almost everywhere
continuous and have compact support (the $\cR$
stands for `Riemann integrable'). For $f \in \cR_0$, and $n >0$,
define the rescaled function $ T_n f$ by
$T_n f(x) := f(n^{-1/d}x)$, for $x \in \R^d$.
  Given a measure $\mu$ on $\R^d$ we shall
often write $\mu(f)$ for $\int_{\R^d} f d \mu$.
Also, we write $\|f\|_2$ for $(\int_{\R^d} |f(x)|^2 dx)^{1/2}$,
and for $f,g \in \cR_0$ we write $\langle f ,g \rangle $ for
$\int_{\R^d} f(x) g(x)dx$.

For $A, B \subset \R^d$ we set 
$A \oplus B := \{a+b:a \in A, b \in B\}$. 

\section{Leaves on the line}
\label{seconedim}
\allco

%\subsection{Motivation}

%\subsection{Leaves on the line}

In this section we take $d=1$ and state our results for  the 
$1$-dimensional DLM.
  We shall prove them in Section 
\ref{secpfoned}.

We define a {\em visible interval} to be a cell of the DLM tessellation.
In the special case where
all of the leaves are single intervals of fixed  length,
 a visible interval is simply
the visible part of a leaf, because this visible part cannot be disconnected.

The endpoints of the visible intervals
form a stationary point process in $\R$. In terms  of earlier notation
this point process is simply the random set $\Phi$ (if viewed as a random
subset of $\R$) or the measure $\H_0(\Phi \cap \cdot)$ (if viewed as a random
point measure).  
We denote this point process (viewed as a random measure)
 by $\eta$, as illustrated in Figure 2.
 The point process $\eta$ is simple, even
if some of the leaves include constituent intervals of length zero
(recall that we assume $\Q$ is such that some of the
leaves have non-empty interior).
% then some of the points of  $\eta$ are  left endpoints
%to two visible intervals (one of strictly positive length and
%one of length zero).
%We refer to the intensity of the point process $\eta$
%as the {\em density of visible intervals.}

%This section is concerned with the point process of (endpoints of)
%visible intervals.

%We shall then present new results concerning
%give formulae for 
Our main results for $d=1$ concern
the second factorial moment measure and the
pair correlation function of $\eta$ (Theorem \ref{th:fmm}), 
% this point process,
 asymptotic covariances
%and limiting variance. We ask if it satisfies a CLT.
and a CLT for the total number of points of $\eta$ in
a large interval (Theorems \ref{th2mom1d}, \ref{limcovthm}, and \ref{CLTb}), 
and
% more generally 
a functional
CLT for this quantity as the DLM evolves in time
(Theorem \ref{FCLT1}).

 We shall also give formulae for the intensity of $\eta$,
%We shall also give %and
 the distribution of the length of 
the visible interval containing the origin
%and deduce the distribution of
and the length of a typical visible interval
 (Propositions \ref{th:main},
% below), (Proposition 
\ref{thmXdist} and
%Corollary
 \ref{typintcoro}).
These propositions are to some extent already known, as we shall discuss.
%below.
% when all the
%leaves are intervals of equal length, these distributions become quite simple.
%The result is in some sense known (at least when the leaves are all connected)
% but the (very simple) proof presented here
%may be new: we explain how
%this formula matches up with the literature.  
%
%and  for 

Recall that $\lambda$ and $R$  are defined by (\ref{lambdadef}) and
(\ref{Rdef}), respectively.

\begin{prop}[Intensity of $\eta$]
\label{th:main}
Assume that $\E[R]< \infty$
% $0< \la < \infty$
 and $\E[\H_0(\partial S)]< \infty$.
Then $\eta$ is a stationary point process with intensity 
%of $\eta$ is equal to
$ \la^{-1} \E[\H_0(\partial S)]$.
\end{prop}
{\bf Remarks.}
 In the special case where all the leaves are intervals
of strictly positive length, the  intensity of $\eta$
 simplifies to just $2/\lambda$.
This special case 
 of Proposition \ref{th:main} is already documented;
% implicit
% (but somewhat hidden)
%in the literature;
 see \cite[page 4]{Matheron},
or \cite[(XIII.41)]{Serra}.
% or the final display of \cite{Bordenave}.

Our more general statement (and proof) of Proposition \ref{th:main}
allows for disconnected leaves.
If $\H_0(\partial S)$ is finite, then $S$ consists of finitely many
disjoint intervals,
and $\H_0(\partial S)$ equals twice the number of constituent intervals
of $S$, minus the number of these intervals having length zero.
It is quite natural to allow a leaf to have several components;
 for example if the one-dimensional
DLM is obtained as the restriction of a higher-dimensional DLM
(in $\R^d$, say, for some $d > 1$)
to a specified one-dimensional subspace of $\R^d$. 
If the leaves in the parent DLM in $\R^d$ are not restricted
to be convex, but have `nice' boundaries (for example
the polygonal boundaries considered in \cite{Cowan}),
then the one-dimensional DLM  induced in this manner
will typically include leaves with more than one component.
The proof given here,  based
on  a general result for DLRMs, and ultimately on the Mecke formula
 from the theory of Poisson processes \cite{LP},
 is quite simple and may be new.

In the next two results we use the following notation.
Let $\nu$ denote the distribution of the Lebesgue measure 
 of a leaf under the measure $\Q$.
Let $H$ denote a random variable with distribution $\nu$.
For $x >0$, write $F(x) = \Pr[H \leq x]$ and $\oF(x):=1-F(x)$, and
 set $\lambda_x := \E[H+ \min(x,H)]$ 
(this is consistent with (\ref{lambdadef})).

\begin{theo}[Second moment measure of $\eta$]
\label{th:fmm}
 Suppose that $\Q$ is concentrated on \textcolor{\blue}{connected sets (i.e., on
 intervals)},  and that $F(0)=0$
and $\lambda < \infty$.
Then the  second factorial moment measure $\alpha_2$ of the point process $\eta$
is given for $x < y $ by
\bea
\alpha_2(d(x,y)) =  4 (1 + F(y- x)) (\lambda \lambda_{y-x})^{-1} dy dx 
+  \la_{y-x}^{- 1}  \Pr[H + x \in dy] dx. 
\label{0615a}
\eea
If $\nu$ has a probability density function $f$, then
 the pair correlation function $\rho_2$ of $\eta$
% (see \cite[p. 73]{LP}) 
is given by
\bea
\rho_2 (z) = \frac{\lambda (1+F(z))}{\lambda_z} + 
\frac{\lambda^2 f(z)}{ 4 \lambda_z}, ~~~ z >0.   
\label{pcf}
\eea
In the particular case where $\nu = \delta_\la$ for some $\lambda >0$,
we have for $x < y$ that
\bea
\alpha_2(d(x,y)) =  \frac{4  
{\bf 1}_{(0,\la)}(y-x)
}{ \lambda (\la + y-x)} 
  dy dx
+ 8 \la^{-2} {\bf 1}_{(\la,\infty)}(y-x) dy dx
+  (2 \la)^{- 1}   \delta_{\la +x} (dy) dx.
%~
\nonumber \\
\label{delta2mom}
\eea
\end{theo}

We now give some limit theorems for
% results on asymptotic variance and CLT for
$\eta([0,n])$, as $n \to \infty$. In these results,
$n$ does not need to be integer-valued.

\begin{theo}[Asymptotic variance for \textcolor{\blue}{$\eta([0,n])$}]
\label{th2mom1d}
Suppose that $\E[(\H_0(\partial S))^2] < \infty$ and
$\E[R^2] < \infty$.
 Then the limit $\sigma_1^2 :=
  \lim_{n \to \infty} n^{-1} \Var[\eta([0,n])]  $ 
exists.  If also $\Q$ is concentrated \textcolor{\blue}{on intervals},  
 and $F(0)=0$, then
\bea
\sigma_1^2 =
 \frac{2}{\la} + 2 \int_{(0,\infty)} \lambda_u^{-1} \Pr[H \in du] 
+ 8 \int_0^\infty \left( \frac{1+F(u)}{\la \la_u} - \frac{1}{\la^2} \right)du.
%\nonumber \\
\label{0628a}
\eea
%and moreover $\sigma_1^2 < \infty$.
In (\ref{0628a}) the last integrand on the right hand side is equal
to
\bea
 \frac{\int_u^\infty \oF(t)dt - \la \oF(u)}{ 
\la^2 (2 \la - \int_u^\infty \oF(t)dt)}. 
\label{0628b}
\eea
In the special case with $\nu = \delta_1$,
 the right hand side of (\ref{0628a}) comes to
$ 8 \log 2 -5 \approx 0.545 $
\end{theo}
It is interesting to consider the  evolving point process 
of visible leaf boundaries.
 For $t \in \R$   let $\eta_t:= \H_0(\Phi_t \cap \cdot)$,
 the point process
of endpoints of visible intervals at time $t$, for the
DLM run {\em forward} in time.
 See Figure 3.
 We use the abbreviation $\eta_t(n)$ for $\eta_t([0,n])$.
%and  $a \wedge b$ for $\min(a,b)$, $a, b \in \R$.
\begin{theo}[Asymptotic covariance for $\eta_t(n)$]
\label{limcovthm}
Suppose that $\E[(\H_0(\partial S))^2] < \infty$ and
$\E[R^2] < \infty$.
Let $t,u \in \R$ and $r,s \in \R_+:= [0,\infty)$.  Then with $\sigma_1^2$ as given
in the preceding theorem, 
\bea
\lim_{n \to \infty} n^{-1} \Cov (\eta_{t}(nr),\eta_{u}(ns) ) = \sigma_1^2
%(r \wedge s)
(\min(r , s))
\exp(- \lambda (|u-t|) )
\nonumber
\\
 =: \kappa_1((r,t),(s,u)).
\label{kappadef}
\eea
\end{theo}

\begin{figure}[!h]
\label{fig5}
\center
\includegraphics[width=1.0\textwidth]{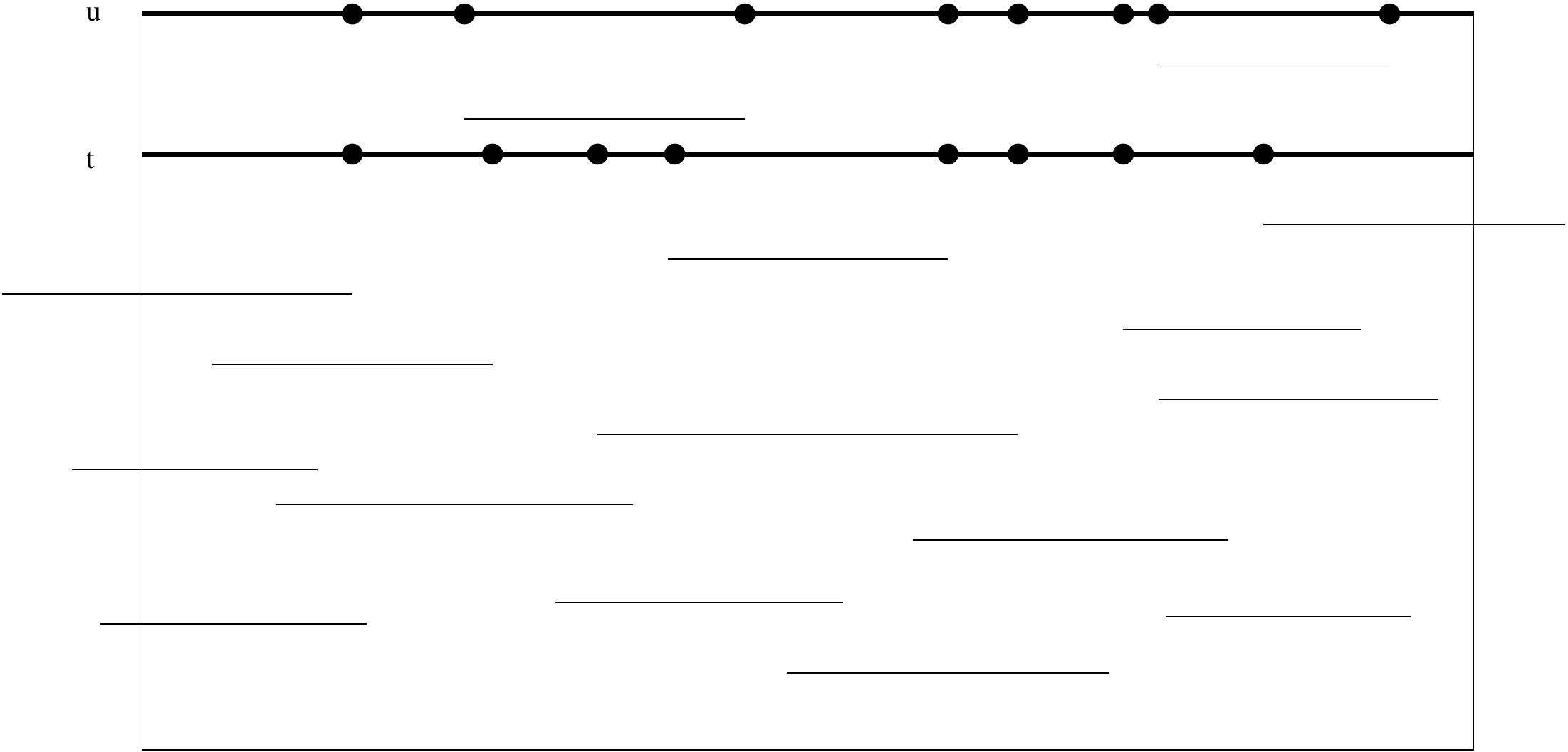}
\caption{Illustration of the evolving DLM tessellation in $d=1$ 
with the point processes $\eta_t$ and $\eta_u$ shown. Here $t < u$.} 
\end{figure}

%\subsection{Central limit theorem}
%\label{sec:clt1d}

%Now we state a central limit theorem for $\eta([0,n])$.

\begin{theo}[CLT for $\eta_t(n)$]
\label{CLTb}
Suppose that there exist constants $\eps, K \in (0,\infty)$ 
such that
 $\E[(\H_0(\partial S))^{2+\eps}] < \infty$
%that there exists a constant $K$ such that
and $\Pr[R \leq K] = 1$.
Then as $n \to \infty$,
\bea
n^{-1/2} (\eta([0,n] ) - \E \eta([0,n])) \tod {\cal N}(0,\sigma^2_1),
\label{CLTeq3}
\eea
where $\sigma^2_1$ is given in Theorem \ref{th2mom1d}.
More generally, the finite-dimensional distributions
of the random field 
%stochastic process 
$ (n^{-1/2} (\eta_t(ns) - \E[\eta_t(ns)]), (s,t) \in \R_+ \times \R )$ 
converge to those of a centred Gaussian random field with covariance function
$\kappa_1((r,t),(s,u))$ as defined in (\ref{kappadef}).
%$$
%\kappa(t,u) = \sigma_1^2 \exp( - \lambda|u-t|), ~~t, u \in \R.
%$$
\end{theo}
By the case $s=1$ of Theorem \ref{CLTb},
 the finite dimensional distributions of the process 
$ (n^{-1/2} (\eta_t(n) - \E[\eta_t(n)]), t  \in \R)$ 
converge to those of  a Gaussian process $(X_t)_{t \in \R}$ with covariance
 function $\sigma_1^2 \exp( - \lambda|u-t|)$.
This limiting process is a stationary  
Ornstein-Uhlenbeck process; see \cite[page 358]{KS}. That is, it is
the solution to the stochastic differential equation
$$
dX_t = - \lambda X_t dt + (2 \lambda)^{1/2} \sigma_1 dB_t,
$$
where $(B_t)$ is a standard Brownian motion.  
Under a stronger moment condition, we can improve this  finite dimensional
convergence
% to $(X_t)$ 
%of Theorem \ref{CLTb}
 to a {\em functional}
CLT;
that is, to convergence in 
$D(-\infty,\infty)$ of right-continuous functions on
$\R$ with left limits. We give this space the
Skorohod topology, as described in \cite{Bill} and
extended to non-compact time intervals in \cite{Whitt}.

\begin{theo}[Functional CLT for $\eta_t(n)$]
\label{FCLT1}
Suppose \textcolor{\blue}{there exist $\eps, K \in (0,\infty)$ such that  
 $\E[(\H_0(\partial S))^{4+\eps}] < \infty$
and $\Pr[R \leq K] = 1$.}
%for some $K \in (0,\infty)$. 
Then as $n \to \infty$, the
 stochastic process $ (n^{-1/2} (\eta_t(n) - \E[\eta_t(n)]), t \in \R)$ 
converges in distribution, in the space
$D(-\infty,\infty)$,
 to the stationary Ornstein-Uhlenbeck Gaussian process with covariance
 %function $\kappa_1((1,t),(1,u))$.
 function $\sigma_1^2 \exp( - \lambda|u-t|),$ $t,u \in \R$.
%as defined in (\ref{kappadef}).
\end{theo}

%It would also be of interest 
%to show full convergence in distribution of 
%the sequence of processes $(n^{-1/2} ( \eta_t(n) - \E \eta_t(n) ),
% t \in \R)_{n \geq 1}$ 
%in the space $D(-\infty,\infty)$ 
%to the limiting Ornstein-Uhlenbeck process, rather than just the convergence
%of finite-dimensional distributions considered here.  
%See \cite{Bill} for a general discussion of weak convergence in $D$.

The limiting random field in Theorem \ref{CLTb} 
%Finally, one could consider the re-scaled point process
%as a two-parameter process in space and time:
%$$
%(s,t) \mapsto n^{-1/2} (\eta_t(ns) - \E[\eta_t(ns) ]), ~~~~ s \geq 0, t \in \R
%$$
%The preceding result tells us that for $s$ fixed this converges
%(at least in the sense of finite-dimensional distributions)
%is to  an Ornstein-Uhlenbeck
%process in $t$. We would expect that as a two-parameter process in $(s,t)$
%it converges to 
\textcolor{\blue}{is an} Ornstein-Uhlenbeck process in Wiener space. See
for example  \cite{quasi} for a definition, or \cite{Meyer} for a more
detailed discussion of this infinite-dimensional Ornstein-Uhlenbeck process. 
It would be interesting, in future work, to try to extend the
finite dimensional  convergence of Theorem \ref{CLTb} to convergence
in an appropriate two-parameter function space.
\textcolor{\blue}{
This is beyond the scope of the methods used here, since our proof is based
ultimately on the classical approach of \cite{Bill}
for showing convergence of a sequence of processes with a single 
time parameter, namely finite-dimensional
convergence plus tightness via moment bounds.}

%Another interesting extension would be 
It would also be of interest
 to extend these CLTs to cases
where the leaves are intervals of unbounded length but satisfy
a moment condition.

% For example, the case of exponentially distributed
%leaves is of interest since in that case, viewing `space' as `time'
%and `time' as `height' as in the `news story' interpretation of the 
%one-dimensional
%DLM, the `height' of the currently lowest leaf follows a Markov process
%in $\R_+$.

We conclude this section with results on
the length of the interval of the DLM tessellation containing the origin,
and the length of a `typical interval'. 
% Propositions \ref{thmXdist} and \ref{typintcoro}
These follow from results in \cite{Matheron}, as we discuss   
in Section  \ref{secpfoned}.
 For an alternative proof,
see the earlier version \cite{v1} of \textcolor{\blue}{the present} paper.

\begin{prop}[Exposed interval length distribution]
\label{thmXdist}
Assume that $\Q$ is concentrated 
\textcolor{\blue}{
on 
%connected
 intervals} 
 of strictly positive length, and 
that $\lambda < \infty$.
Let $\nu$ denote the distribution of the length of a leaf under
the measure $\Q$, and
 $X$ the length of the  visible interval covering the origin.
The distribution of $X$ is given
 by 
\bea
\Pr[X \in dx] =  \left(  \int_{(x,\infty)} 
\left(
 \frac{2 x (\la + u) }{ (\la +x)^3} \right) \nu(du) 
\right)
dx + \frac{x \nu(dx)}{x + \la} .
\label{Xdist}
\eea
In the special case where
 the measure $\nu$ is the Dirac measure $\delta_\la$
for some $\la >0$,
\bea
\Pr[X \in dx] =   \left(
 \frac{4 \la x}{ (\la +x)^3} \right) {\bf 1}_{\{x < \la\}} 
dx + (1/2) \delta_{\la}(dx).
\label{0323b}
\eea
\end{prop}
If $\nu$ is a Dirac measure, then $X$ is the length of
the visible part of the leaf visible at the origin.
In general, 
 $X$ counts the length only of the connected component
of the visible part of this
 leaf that includes $0$, ignoring any other components.

A {\em typical visible interval}, loosely speaking, is
obtained by fixing a very large region of $\R$, and
 choosing at random one of the inter-point intervals of $\eta$  
that lie in that region.
  The distribution of the length of a typical visible interval is 
the inverse-size-biased distribution of $X$ (see \cite[Proposition 9.7]{LP});
that is,
 if $Y$ is the length of a typical visible interval
we have $\Pr[Y \in dy] = y^{-1} \Pr[X \in dy]/ \E[X^{-1}]$.
Now $\E[X^{-1}]=2/\lambda$, which can be deduced either from
(\ref{Xdist}) or from Proposition \ref{th:main} using \cite[eqn (9.22)]{LP}. 
Hence  from (\ref{Xdist}) 
%and (\ref{0323c})    
we have the following.
%$Y$ has a mixed distribution with density
\begin{prop}
\label{typintcoro}
 Let $Y$ denote the length of a typical visible interval. 
Under the assumptions of Proposition \ref{thmXdist},
the distribution
of $Y$ is given by
\bea
\Pr[Y \in dy]
 =  
%(\la/2) \left(
 \left(
 \int_{(y,\infty)} \left(\frac{\la  (\la +u)}{ (\la +y)^3 } 
\right)
\nu(du) \right) dy
+  \frac{\la \nu(dy) }{2 (y+ \la)} 
.
\label{0323c}
\eea
If all the leaves are  intervals of length $\lambda$, i.e.
 $\nu = \delta_\lambda$ for some $\lambda >0$,
then $Y$ has a mixed distribution with
$$
\Pr[Y \in dy] =   \left(
\frac{2 \lambda^2 {\bf 1}_{(0,\lambda)}(y) }{(\lambda+y)^3}
\right)
dy + (1/4) \delta_\lambda(dy).
$$
\end{prop}

\section{Leaves in the plane}
\label{ctshidim}
\allco

In this section we take $d=2$, and state our results for the 
two-dimensional DLM.
  We shall prove them in Section 
\ref{secpftwodim}.

We shall say that our grain distribution $\Q$
% on $\cK$ (the space of compact sets in $\R^2$) 
 has
the {\em rectifiable Jordan property} if it is concentrated on nonempty regular
compact sets 
having a rectifiable Jordan curve as their boundary. 
Here, we say a compact set in $\R^2$ is {\em regular} if 
it is the closure of its interior.
We say $\Q$ has
the {\em piecewise $C^1$ Jordan property} if it is concentrated on sets 
having a piecewise $C^1$ Jordan curve as their boundary. 
%We say that $\Q$ is {\em rotation invariant}, if the distribution
%of $S$ is invariant under any   rotation about
%the origin.

Recalling the definitions (\ref{Phindef}) and (\ref{Phidef}),
define the measures $\phi := \H_1(\Phi \cap  \cdot)$,
 the restriction of the one-dimensional Hausdorff measure 
to the  boundaries of the DLM tessellation, and
 $\phi_t := \H_1(\Phi_t \cap  \cdot)$ for $t \in \R$.
As in (\ref{lambdadef}),
%we set $\lambda:= \int \H_2(\sigma) \Q(d\sigma)$, and
% $\lambda_x := \int \H_2 (\sigma \cup (\sigma+x) ) \Q(d\sigma)$  for
we set $\lambda:= \E [ \H_2(S) ]$, and
 $\lambda_x := \E [\H_2 (S \cup (S+x) ) ]$ for
 $x \in \R^2$, where $S$ is a random element of $\cK$ with
distribution $\Q$ as per usual. Define $R$ by (\ref{Rdef}).

\begin{theo}[Intensity of cell boundaries]
\label{meantheo}
Suppose that Condition  \ref{measassu} holds and 
%$\E[R^2] < \infty$.
$\E[\H_2(S \oplus B(1)) ] < \infty$.
 Then
%Suppose $\H_1(\partial S \cap \cdot)$ is a random measure.  Then
 $\phi$ is a stationary random measure and its intensity is
$
\lambda^{-1}
\E [ \H_1(\partial S) ].
$
%[MAYBE TAKE $\E[\H_2(S \oplus B(1)) ] < \infty$ instead of $\E[R^2] < \infty$.]
\end{theo}

As mentioned earlier, the cells of the (time-reversed) DLM tessellation
are the closures of the connected components of the set $\R^2 \setminus \Phi$.
We now define $\Xi$ to be the set of points in $\R^2$ which lie in
three or more cells of this tessellation.
Later we shall view $\Phi$ as a planar graph with   the points
of $\Xi$ as  the nodes, which we call {\em branch points},
in this graph.
We define the measure $\chi := \H_0(\Xi \cap \cdot) $.

%Let $\chi$ be the point process of
% points at which three cells of
%the DLM tessellation meet; 
%Denote the intensity of the point processes $\chi$ and $\psi$
%by $\beta_3$ and $\beta_1$ respectively, assuming that they
%are indeed point processes (i.e., that they are measurable).

%For the next result we need to assume $\Q$ is such
% that $\H_1(\partial S \cap \cdot)$
%is a random measure, that is,  a measurable map from the underlying
%probability space to the space of measures on $\R^d$. We do not know
%if this follows automatically
% from the earlier assumption that $S$ is a random closed set.

For $A \subset \R^2$ 
and $\theta \in (-\pi, \pi]$, let $\rho_\theta(A)$ denote the image
of $A$ under an anticlockwise rotation
through an angle $\theta$ about the origin (elsewhere we are using $\rho_2$
to denote pair correlation, but this clash of notation should not be confusing).
We say $\Q$ is {\em rotation invariant}
if $\rho_\theta(S) \eqd S$ for all $\theta \in (-\pi,\pi]$.

\begin{theo}[Intensity of branch points]
\label{thbranch}
Assume either that $\Q$ has the piecewise $C^1$ Jordan property, or that 
$\Q$ has the rectifiable Jordan property  and is rotation invariant.
Assume also that $\lambda < \infty$, and $\E[R^2] < \infty$, 
\textcolor{\blue}{and 
%where we
 set
\bea
 \beta_3 := \lambda^{-2} \int_{\cK}  \int_{\cK} 
 \int_{\R^2} \H_0(\partial \sigma \cap (\partial \sigma'+x))
dx \Q(d \sigma) \Q(d \sigma') . 
\label{0628c}
\eea
Then:
(a)
If $\beta_3 < \infty$, then 
 $\chi$  is a stationary point process,
%(b)
%If  $\Q$ has the Jordan property, then 
with intensity
% of
%$\chi$ is equal to
 $\beta_3$.  
}

(b)  If $\Q$ 
 is rotation invariant, then
\bea
\beta_3 = \frac{2}{\pi \lambda^{2}} \left(
% \int_{\cK} \H_1(\partial S) \Q(dS)
\E[ \H_1 (\partial S) ] 
 \right)^2. 
\label{0629a}
\eea
\end{theo}

The \textcolor{\blue}{next two results require} $\Q$ to have a further property.
We say $\Q$ has the {\em non-containment property} if
 for $(\Q \otimes \Q)$-almost all pairs $(\sigma,\sigma') \in \cK \times \cK$,
 the  set of  $x \in \R^2$ such that  
$\sigma +x \subset \sigma'$ is Lebesgue-null. 
One way to guarantee the non-containment property is
to have $\Q$ be such that under $\Q$, all of the 
sets $S_i$ have the same area.

\begin{theo}[Connectivity of $\Phi$]
\label{thmconnect}
\textcolor{\blue}{Suppose $\Q$ has the rectifiable Jordan and non-containment
properties, and that $\E[R^2] < \infty$. Then $\Phi$ is almost surely  a
connected set.}
\end{theo}

Let $\Psi$ be the set of centroids of cells of the DLM tessellation,
and define the measure $\psi := \H_0(\Psi \cap \cdot)$. 
While we would expect that $\psi $ is 
a point process (i.e., that it is measurable),
we have not proved this in general
(unlike in the case of $\chi$), so we leave
this as an open problem and include
the measurability  as an assumption in the next result.

\begin{theo}[Intensity of cells]
\label{meantheo2}
Suppose $\Q$ has  the rectifiable Jordan and non-containment properties,
and either has the piecewise $C^1$ Jordan property or is rotation
invariant. 
Assume that $\beta_3 < \infty$, and $\E[R^2]< \infty$,
 and that $\psi $ is a point process.
Then $\psi$ is a stationary point process, and 
its intensity, denoted  $\beta_1$, is given by
$\beta_1 = \beta_3/2$.  In particular, if 
$\Q$ is %also
 rotation invariant, 
 then
\bea
\beta_1 =
  (\pi \lambda^2)^{-1} 
(\E[\H_1(\partial S) ])^2.
\label{0320a}
\eea
\end{theo}
{\bf Remarks.}
%Eqn (\ref{meanbdy})
Our formula for the intensity of $\phi$ in Theorem \ref{meantheo} 
 agrees with that of \cite[p. 57]{Cowan} but
is considerably more general. In \cite{Cowan} it is assumed
that $\Q$ is such that a random
 set $S$ with distribution $\Q$ is a uniform random
rotation of a fixed polygon $S_0$.
In \cite[Sec. 7]{Cowan} there is some discussion on
generalising to the case where $S_0$ is  non-polygonal,
 but it is still taken to
be a fixed set.
Similarly,
Equations (\ref{0629a}) and
(\ref{0320a}) also  generalize  formulae in \cite{Cowan}.

\textcolor{\blue}{
Theorem \ref{thmconnect} is perhaps intuitively obvious, but a careful
proof seems to require some effort. As well as being of interest in itself,
 the connectivity of $\Phi$ is required
for the proof of Theorem \ref{meantheo2}.
}

The reason we require  $\Q$ to have the
Jordan and non-containment properties in Theorem 
\ref{meantheo2}, is because the proof relies on a 
topological argument based on the cell boundaries of the DLM 
tessellation forming a connected planar graph with all
vertices of degree 3.  The Jordan property 
(requiring all leaves to be connected with a
% piecewise $C^1$ 
Jordan curve boundary)
could be relaxed to  a requirement that every leaf
has a finite (and uniformly bounded)  number of components, each with
a Jordan curve boundary;
%a piecewise $C^1$ Jordan curve boundary;
%are connected
% boundary of $\Q$-distributed
%random set being almost surely a union of a 
%uniformly bounded finite number of
%Jordan curves;
 the key requirement here is to avoid 
having leaves which are one-dimensional sticks or have
boundary shaped like a figure 8 or letter $b$, for example, since
then there would be vertices of degree other than 3.
The non-containment 
condition is needed to ensure that the planar graph of
boundaries is connected.
If it fails  (but the Jordan 
condition holds) then one may still deduce in a more
general version of Theorem \ref{meantheo2} that $\beta_3/2$
is the density of   faces minus the density of `holes', 
where by a `hole' we mean a bounded 
component of the union of cell boundaries.

{\bf Examples.}
If the leaves  are all a fixed  convex set $S_0$
with $0 < \H_2(S_0) < \infty$, then
using \textcolor{\blue}{Theorem \ref{meantheo2} and (\ref{0628c}) we have}
$$
\beta_1 = 
\textcolor{\blue}{ \beta_3/2 }
 = (\H_2(S_0))^{-2}   \H_2( S_0 \oplus \check{S_0}),
$$
where $\check{S_0} := \{-x: x \in S_0\}$.  
Therefore,  if moreover
% the leaves are a fixed convex set
$S_0$ 
% which
 is symmetric (i.e. $S_0 = \check{S}_0$;  for example if $S_0$ is
 a fixed rectangle or circle centred on the origin), then
$
\beta_1 = 4/\H_2(S_0).
$
%which is consistent with (\ref{180409a}). 

On the other hand,
if each leaf is a uniformly distributed random rotation of
a unit square, then (\ref{0320a}) gives us $\beta_1 =  16/\pi$.

Recall the definitions of $\cR_0$, $W_n$, 
 $T_n$ and $R$ from
Section \ref{subsecnotation}, \textcolor{\blue}{and of $\phi$, $\phi_t$
from the start of this section}.

\begin{theo}[Asymptotic covariance for edge length]
\label{thlengthcov}
Suppose $\E[(\H_1(\partial S))^2] < \infty$ and $\E[R^4] < \infty$.
Let $f \in \cR_0$.
Then $n^{-1} \Var[\phi(T_n(f))] \to \sigma_2^2 \|f\|_2^2$
as $n \to \infty$, with $\sigma_2^2 := v_1 + v_2- v_3$,
where we set
\bea
 v_1 :=
  \E \int_{\partial S} \int_{\partial S} \lambda_{y-x}^{-1} \H_1(dx) \H_1(dy),
\label{A5def}
\eea
\bea
v_2 : = \lambda^{-2} (\E \H_1(\partial S) )^2   \int_{\R^2}  
((2\lambda/\lambda_x) -1) dx,
\label{A6def}
\eea
and
\bea
v_3 :=  \lambda^{-1} \left( \E [ \H_1(\partial S) ] \right) 
 \E \int_{\partial S} \left( \int_{S}  (2/\lambda_{y-x}) dy \right) \H_1(dx) , 
\label{A7def}
\eea
and the quantities $v_1,v_2,v_3$ are all finite.
More generally, for $t ,u \in \R$ and $f,g \in \cR_0$,
\bea
\lim_{n \to \infty} n^{-1} \Cov(\phi_t(T_n(f)),\phi_u(T_n(g))) = 
\sigma_2^2 \langle f, g \rangle 
\exp(- \lambda|u-t| ) 
\nonumber \\
=: \kappa_2((f,t),(g,u)).
\label{0509a}
\eea
\end{theo}

%%\allco

We now provide a central limit theorem for  $\phi(W_n)$,
under the assumption that the leaves are {\em uniformly bounded} together
with a moment condition on $\H_1(\partial S)$.
% that is,
%there exists a constant $r_0$ such that $\Pr[R \leq r_0]=1$,
%where $R:= \max\{\|x\|:x \in S\}$.
%and a measurable  $A \subset {\cal K}$
%such that $\Q(A)=1$ and  
% such that 
%$\Pr[ S < r_0 ] =1$. 
%$ S \subset B(0,r_0)$ for all $S \in A$.

\begin{theo}[CLT for the length of tessellation boundaries]
\label{CLTa}
Suppose for some $\eps, r_0 \in (0,\infty)$ 
that $\E[(\H_1(\partial S))^{2 + \eps}] < \infty$ and
 $\Pr[R \leq r_0]=1$.
%where $R:= \max\{\|x\|:x \in S\}$ and $\|\cdot\|$ is the Euclidean norm.
% the leaves are uniformly bounded.
Then
\bea
n^{-1/2} (\phi(W_n ) - \E \phi(W_n)) \tod {\cal N}(0,\sigma_2^2),
\label{CLTeq2}
\eea
where $\sigma_2$ is as given in Theorem \ref{thlengthcov}. 
More generally, the finite-dimensional distributions
of the random field
 $ (n^{-1/2}( \phi_t(T_n(f)) - \E[\phi_t(T_n(f))]) , f \in \cR_0, t \in \R)$ 
converge to those of a centred Gaussian random field with covariance function
$\kappa_2((f,t),(g,u)) $ given by (\ref{0509a}).
%$$
%\kappa(t,u) = \sigma_2^2 \exp( - \lambda|u-t|), ~~t, u \in \R.
%$$
\end{theo}
\begin{theo}[Functional CLT for the length of tessellation boundaries]
\label{FCLT2}
Suppose \textcolor{\blue}{for some $ \eps, r_0 \in (0,\infty)$ 
that $\E[(\H_1(\partial S))^{4+\eps}] < \infty$ and
 $\Pr[R \leq r_0]=1$.}
Let $f \in \cR_0$.
Then
% the stochastic process 
 $ (n^{-1/2}( \phi_t(T_n(f)) - \E[\phi_t(T_n(f))]) , t \in \R)$ 
converges in distribution as $n \to \infty$, in the space $D(-\infty,\infty)$,
to the stationary Ornstein-Uhlenbeck process with
covariance function $\kappa_2((f,t),(f,u))$.
\end{theo}
{\bf Remarks.}
The limiting Gaussian process in the preceding theorem is a stationary
  Ornstein-Uhlenbeck process.
Similar remarks to those made after Theorem \ref{CLTb}, regarding
possible extensions to the result above, apply here.

It should be possible to adapt 
the conditional variance argument  of Avram and Bertsimas
\cite{AvB}
to show the proportionate variance of $\phi(W_n)$ is bounded
away from zero, so that  $\sigma_2^2 $
 %the variance of the limiting normal distribution in (\ref{CLTeq2})
 is strictly positive.

\section{Dead leaves random measures}
\label{secDLRM}
\allco

In this section we present some results for the DLM in
arbitrary dimension $d \in \N$ (which we shall prove in
 Section \ref{secproofsDLRM}), which enable us to consider some
of the results already stated in a unified framework and also
to indicate further results on dead-leaves type models that can
be derived similarly.

 It is convenient here to consider a slightly more general
setting than before.
  We augment our mark space (previously taken to be $\cK$)
to now be the space $\cK \times \bM$.
  Let $\Q'$ denote a probability measure on $\cK \times \bM$ 
with first  marginal $\Q$. Assume that our
%Poisson process $\Po = \{(x_i,t_i): i \in \N\}$  
Poisson process $\Po = \sum_{i=1}^\infty \delta_{(x_i,t_i)}$  
in $\R^d \times \R $
is now independently marked using a sequence 
$(S_i,M_i)_{i \geq 1}$
of independent random elements of $\cK \times \bM$
with common distribution $\Q'$.
With each point $(x_i,t_i)$ of $\Po$
 we
associate a `leaf' $S_i+x_i$ and also a {\em measure} $M_i+x_i$, where
$(\mu +x)(A) := \mu(A+ (-x))$ for any $(\mu,x) \in \bM \times
\R^d$ and  $A \in \cB^d$.
For each $i$ the measure $M_i +x_i$ is added at time $t_i$
but is then restricted to the complement of
 regions covered  by later arriving leaves
$S_j + x_j$, as they arrive.
%We superimpose the measures $M_i + x_i$  
%on each other but remove the regions covered  by later arriving leaves
%$S_j + x_j$.
Thus, at time $t \in \R$ we end up with a measure
\bea
\xi_t := \sum_{\{i: t_i \leq t \}} (M_i +x_i) (\cdot \cap 
\R^d \setminus \cup_{\{j: t_i < t_j \leq t 
 \}} (S_j + x_j)), 
\label{xitdef}
\eea 
which we call the {\em dead leaves random measure} (DLRM) at time $t$.
We also define the {\em  time-reversed DLRM} (at time zero) by
\bea
\xi := \sum_{\{i: t_i \geq 0  \}} (M_i +x_i) (\cdot \cap 
\R^d \setminus \cup_{\{j: 0 \leq t_j < t_i 
 \}} (S_j + x_j)), 
\label{xidef}
\eea

Here are some examples of how to specify 
%the distribution $\Q'$ to obtain
\textcolor{\blue}{
a type of distribution  $\Q'$ that yields a}
  resulting DLRM of interest.
In these examples,
%we specify $\Q'$ by describing the distribution
%under $\Q'$ of a random element $(S,M)$ of $\cK \times \bM$.
\textcolor{\blue}{
to describe $\Q'$ 
we let $(S,M)$ 
denote a random element
of $\cK \times \bM$ having
the distribution
 $\Q'$, and we describe the interpretation of the resulting DLRM.
}
 Often we take   $M$ to be supported by $S$ but this is not essential. 
\begin{itemize}
\item
Let $M(\cdot) := \H_{d-1}(\partial S \cap \cdot)$.
%Assume that Condition  \ref{measassu} holds, and that
 %$\H_{d-1}(\partial S) <\infty$ almost surely. 
Then (see Proposition \ref{prop:bdy} below),
 the resulting DLRM is $\xi_t = \H_{d-1}(\Phi_t \cap \cdot) $,
where $\Phi_t$ is the set of points in the union of
all cell boundaries of the DLM tessellation
at time $t$. Similarly, $\xi = \H_{d-1}(\Phi \cap \cdot) $.
For $d=1$,  this $\xi$ is the same as the  measure $\eta$ considered
earlier.  For $d=2$, this $\xi$ is the same as the measure $\phi$
considered earlier.
%For $d=2$, the measure $\xi_t$ is the measure $this  yields the 
\item
Take $d=2$ and let $M$ be
\textcolor{\blue}{the counting measure supported by} 
%restricted to
 the set of corners of $S$ 
\textcolor{\blue}{(counting measures are defined in
e.g. \cite{LP})}. Here we could be assuming
that the shape $S$ is almost surely polygonal, or more
generally,  that its boundary  is almost surely a piecewise $C^1$
Jordan curve.
We defined a `corner' of such a curve in Section \ref{subsecnotation}. 
%with a `corner' defined to be a point in the boundary where it
%is not a $C^1$ curve.
  The resulting measure $\xi$ is
% counting measure for 
\textcolor{\blue}{the counting measure supported by} 
the set of corners of the boundaries of the DLM tessellation,
which has been considered in \cite{Cowan}.
\item
{\em Colour Dead Leaves Model (CDLM)}. Let each leaf have a `colour'  (either 1 or 0)
and let $M$ be Lebesgue measure restricted to $S$ (if the colour is 1)
or the zero measure (if the colour is 0).  Then $\xi$ is Lebesgue
measure restricted to those visible leaves which are coloured 1.  
The CDLM was introduced by Jeulin in \cite{JeulinMik} (see
also \cite{Jeulin}), and is the basis of
the percolation problems considered in \cite{Schramm,Muller}.
%In the usual version of this, the decisions on how to colour a leaf,
%and its  shape/size, are independent, but they could also be made dependent.
\item
{\em Dead Leaves Random function (DLRF).} Let \textcolor{\blue}{$M$
%$\mu$
 have a density} 
given by a random function $f:\R^d \to \R_+$ with support 
%contained in 
$S$ (representing for example the level of `greyscale' on 
the leaf $S$). Then $\xi$ is a measure with density at each site
$x \in \R^d$ given by the
level of greyscale on the leaf visible at $x$.
The DLRF  has been proposed by Jeulin \cite{Jeulin} 
for modelling microscopic images.

\item  {\em Seeds and leaves model.}
%We introduce the following model to illustrate that there is no need for
%the measure $\mu$ to be supported by $S$.
 Imagine that at each `event' of
our Poisson process the arriving object is
 either a random finite set of points
(seeds)  or a leaf. Thus for the random 
$\Q'$-distributed pair $(S,M)$,
   either $M$ is  a finite sum of
Dirac measures and $S$ is the empty set, or $M$ is the zero measure
and $S$ is a non-empty set in $\cK$ (a  leaf).
 The point process $\xi_t$ will then represent
the set of locations of seeds on the ground that are
 visible (i.e., not covered by leaves) at time $t$. It might
be that these are the seeds which have potential to grow into new trees,
or that they are the seeds which get eaten.
\end{itemize}

We now give some general results on the DLRM. In applying these results
elsewhere in this paper, we concentrate on the first of the examples
 just listed.  However, the general results could similarly be applied
to the other examples.
In the following results, $(S,M)$ denotes a random element
of $\cK \times \bM$ with distribution $\Q'$, and
we write $|M|$ for $M(\R^d)$.  
 We define $\lambda$, $\lambda_x$, $R$  and $W_n$
 as in Section \ref{subsecnotation}.
\begin{theo}[Intensity of $\xi$]
\label{th:momgen}
Assume $\Q'$ is such that $\E[|M|] < \infty$, and also 
%$\E[R^d] < \infty $.  [OR ALTERNATIVELY 
$\E[\H_d(S \oplus B(1))]< \infty$.
% where $R$ was defined at (\ref{Rdef}).
Then $\xi$ defined at (\ref{xidef}) is a stationary random measure and its
 intensity, denoted  $\alpha $, is given by
% of $\xi$ is
\bea
 \alpha = \lambda^{-1} \E[|M|].
\label{genmean}
\eea
\end{theo}
\begin{theo}[Asymptotic covariance for the DLRM]
\label{thcovgen}
Suppose $\E[|M|^2] < \infty$ and $\E[R^{2d}] < \infty$.
Let $f \in \cR_0$. 
Then $n^{-1} \Var[\xi(T_n(f))] \to \sigma_0^2 \|f\|_2^2$
as $n \to \infty$,
where we set
$\sigma_0^2 := v_4 + v_5- v_6$
with
\bea
 v_4 :=
  \E \int_{\R^d} \int_{\R^d} \lambda_{y-x}^{-1} M(dy) M(dx) ,
\label{B5def}
\eea
\bea
v_5 : = \lambda^{-2} (\E [|M|] )^2   \int_{\R^d}
((2\lambda/\lambda_x) -1 ) dx ,
\label{B6def}
\eea
and
\bea
v_6 :=  \lambda^{-1} \E [ |M| ]
 \E \left[ \int_{\R^d}  \left( \int_{S} (2/\lambda_{y-x}) dy 
\right)
M(dx)
\right], 
\label{B7def}
\eea
and the quantities $v_4,v_5,v_6$ are all finite.
More generally, for $t ,u \in \R$, $f,g \in \cR_0$,
\bea
\lim_{n \to \infty} n^{-1} \Cov(\xi_t(T_n(f)),\xi_u(T_n(g))) = 
\sigma_0^2 \langle f,g \rangle 
\exp(- \lambda|u-t| ) 
\label{0509a2}
%\nonumber
 \\
=: \kappa_0((f,t),(g,u)).
\label{kappa0def}
\eea
\end{theo}

\textcolor{\blue}{
 Theorem \ref{thcovgen} does not rule out the possibility
that $\sigma_0$ could be zero. Our formula for $\sigma_0^2$ has some 
resemblance for the formula for the asymptotic variances of
certain measures associated with the Boolean model in \cite[eqn (7.3)]{HM}.  
There is a certain similarity between the manner in which these measures
are defined in \cite{HM}, and the DLRMs considered here. However there
is no time-parameter in the definition of the Boolean model.
}
%In the next result, $\cN(0,s^2)$ denotes a normal random variable with
%mean 0 and variance $s^2$ (if $s >0$) or a degenerate random variable
%taking value 0 almost surely (if $s =0$).

\begin{theo}[CLT for the DLRM]
\label{th:CLTgen}
(a)
Suppose for some $\eps, r_0 \in (0,\infty)$ that
 $\E[|M|^{2+\eps} ] < \infty$,
%and that there exists $r_0 < \infty$ such that 
$M$ is supported by
the ball $B(r_0)$ almost surely, \textcolor{\blue}{and that} $R \leq r_0$ almost surely.
Then with $\sigma_0$ given in Theorem \ref{thcovgen},
%the preceding theorem,
 the finite-dimensional distributions of the 
random field 
$n^{-1/2}(\xi_t(T_n(f)) - \E [\xi_t(T_n(f))] )_{f \in \cR_0,t \in \R} $
 converge to those of a centred Gaussian random field with covariance
function $\kappa_0((f,t),(g,u))$ given by (\ref{kappa0def}).
%$$
%\kappa(t,u) = \sigma_0^2 \exp(- \lambda |t-u| ).
%$$
In particular,
\bea
n^{-1/2} (\xi(W_n) - \E [\xi(W_n)] ) \tod \cN(0,\sigma_0^2).
\label{0529a}
\eea
(b) \textcolor{\blue}{Suppose in addition that $\E[|M|^{4+ \eps}]< \infty$ for
some $\eps >0$.} Let 
 $f \in \cR_0$.
Then as $n \to \infty$ the  process 
%\linebreak 
$n^{-1/2}(\xi_t(T_n(f)) - \E [\xi_t(T_n(f))] )_{t \in \R} $
 converges in distribution (in $D(-\infty,\infty)$)  to 
the stationary Ornstein-Uhlenbeck process 
with covariance function $\kappa_0((f,t),(f,u))$.
\end{theo}

Our proof of (\ref{0529a})
 provides a rate of convergence (using the Kolmogorov distance)
to the normal in  (\ref{0529a}), and hence also in
 Theorems \ref{CLTb} and \ref{CLTa}.
Under a stronger moment condition, namely
 $\E [ |M|^{3+\eps} ] < \infty $, 
one can adapt the proof (which is based on the Chen-Stein method)
 to make the rate of convergence presumably optimal. 

It would be of interest to derive a 
 functional CLT for the DLRM starting from the
zero measure at time 0 (rather than starting from
equilibrium as we have taken here). It may be possible
to do this using  \cite[Theorem 3.3]{Surveys};
the evolving DLRM fits into the general framework of the spatial
birth, death, migration and displacement process
in \cite[Section 4.1]{Surveys}. It is not so clear
whether results from \cite{Surveys} can be used
directly in the present setting where the DLM starts in
equilibrium, although the argument used here
is related to that in 
%the  proof of
 \cite{Surveys}.
%o prove Theorem \ref{th:CLTgen}(b).

It would also be of interest
 to extend these CLTs to cases
where there is no \textcolor{\blue}{uniform bound} $r_0$ on
the range of the support of $M$ and  the value of $R$.
We would expect that the uniform boundedness  condition could
be replaced by appropriate  moment conditions, but we leave this
for future work. \textcolor{\blue}{There
 are several approaches to proving central
 limit theorems for Boolean models
(see for example in \cite{HM,PEJP,HLS}),
 which allow for unbounded grains 
and might be adaptable to the dead leaves setting}.

Our last result in this section confirms that the surface measure
$\H_{d-1}(\Phi \cap \cdot)$ of the DLM can be obtained 
as a special case of the DLRM.

\begin{prop}
\label{prop:bdy}
Let $\Q'$ be such that
 $M(\cdot) := \H_{d-1}(\partial S \cap \cdot)$.
Assume that Condition  \ref{measassu} holds, and that
 $\H_{d-1}(\partial S) <\infty$ almost surely. 
Then the resulting DLRM is $\xi_t = \H_{d-1}(\Phi_t \cap \cdot) $,
where $\Phi_t$ is the set of points in the union of
all cell boundaries of the DLM tessellation
at time $t$.  Similarly, $\xi = \H_{d-1}(\Phi \cap \cdot)$.
\end{prop}

\section{Buffon's noodle and Poincar\'e's formula}
\label{secBuffon}
\allco

The classical Buffon's needle problem
may be phrased as follows.
If one throws a stick (straight, and of zero thickness)
at random onto a wooden floor (so both its 
location and its orientation are random and uniform), then
how often does one expect to see it cross the cracks
between floorboards?  The generalization 
to a possibly {\em curved} stick has been wittily christened
{\em Buffon's noodle}.
%\cite{Ramaley}.  
What we require here is a further variant, concerned
with the expected number of crossings for
 {\em two} curved sticks thrown at random onto
a carpeted floor. This  has been referred to as
{\em Poincar\'e's formula} \cite{Maak}, although in an earlier
version of the present paper \cite{v1} we called
it the {\em two noodle} formula.
% and wishes to compute
%The answer is $(2/\pi)$ times the product of the lengths of
%the sticks, divided by the size of the floor.
%  The answer is well known to be $2 \ell /(\pi  a)$,
%where $a$ is the width of the floorboards (assumed infinitely long),
 %and $\ell$ the length of the stick.

%\begin{lemm}[Buffon's two noodles]
\begin{lemm}[Poincare's `two noodle' formula]
\label{lemBuffon}
Let  $\gamma $ and $\gamma'$ be
% compact,
 rectifiable
% piecewise $C^1$
 curves in $\R^2$.  Then
\bea
\int_{-\pi}^\pi \int_{\R^2} \H_0 \left( \gamma \cap 
\left( \rho_\theta(\gamma') +x \right) \right) dx d \theta
= 4 \H_1(\gamma) \H_1(\gamma').
\label{eqBuffon}
\eea
\end{lemm}
In other words, if $\gamma'$ is rotated uniformly at random,
and translated by a random amount uniformly distributed over a large
region $A$, then the expected number of times it  intersects
$\gamma $ is  equal to $(2/\pi)$ times the product of the lengths
of $\gamma$ and $\gamma'$, divided by the area of $A$.
%A result along these lines was given by Morton \cite{Morton}.
%In our opinion, Morton's proof is incomplete.  Just after
%\cite[eqn (5)]{Morton}, he does not explain why the expected number of
%crossings (for approximating polygons) converges to the expected
%number of crossings (for the limiting pair of curves).  
Lemma \ref{lemBuffon} follows from
 \cite[Theorem 1.5.1]{Zahle}.
In the special case where $\gamma$ and $\gamma'$ are piecewise $C^1$, 
an elementary proof of (\ref{eqBuffon})
 may be found in \cite{v1}.

Given rectifiable curves $\gamma$ and $\gamma'$ in $\R^2$,
we say that $\gamma$ and $\gamma'$ {\em cross}
at a point $x \in \gamma \cap \gamma'$ if $x$ is not
an endpoint of $\gamma$ or $\gamma'$, and 
$\gamma$ passes from one side of $\gamma'$ to the other
at $x$,
where the `sides' of $\gamma'$
in a neighbourhood of $x$ 
can be defined by extending $\gamma$ to a Jordan curve and
taking the two components of its complement. 
We say that $\gamma$ and $\gamma'$ {\em touch} at $x$
if $x \in \gamma \cap \gamma'$ but $\gamma$ and $\gamma'$
do not cross at $x$.

We say that $\gamma$ and $\gamma'$ {\em touch} if there
exists $z \in \R^2$ such that they touch at $z$.
As well as Lemma \ref{lemBuffon}, 
in the proof of (\ref{0629a}) and (\ref{0320a})
we require the following.

\begin{lemm}
\label{touchlemrect}
Suppose $\gamma$ and $\gamma'$ are rectifiable curves in $\R^2$.
For Lebesgue-almost every $(z,\theta) \in \R^2 \times (-\pi,\pi]$,
the set $\rho_\theta(\gamma') +z$ does not touch $\gamma$.
\end{lemm}

\begin{lemm}
\label{touchlemC1}
Suppose $\gamma$ and $\gamma'$ 
are \textcolor{\blue}{piecewise} 
$C^1$ curves in $\R^2$.
For Lebesgue-almost every $z \in \R^2$,
the set $\gamma' +z$ does not touch $\gamma$.
\end{lemm}

We  shall prove Lemmas \ref{touchlemrect} and \ref{touchlemC1} 
\textcolor{\blue}{later in} this section. The proof of Lemma \ref{touchlemrect}
is very short, but heavily reliant  on results in
 \cite{Maak,Zahle}.
In the case where $\gamma$ and $\gamma'$
are piecewise $C^1$,  the conclusion of Lemma \ref{touchlemrect}
can alternatively be derived from Lemma \ref{touchlemC1}; we shall provide
an elementary proof of the latter result. 

As a slight digression, we also state the Buffon's noodle
result mentioned above.

  \begin{theo}[Buffon's noodle]
\label{th:noodle}
Let  $\gamma $  be a 
\textcolor{\blue}{rectifiable
% compact, piecewise $C^1$
 curve} in $\R^2$.
\textcolor{\blue}{Let $a >0$.}
 For $i \in \Z$, let $L_i$ denote the horizontal line
$\{(x,y):y = i\}$. 
 If $Y$ and $\Theta$ denote independent random variables, uniformly
 distributed over $[0,a)$ and $(-\pi,\pi)$ respectively, then
\bea
\E \left[ \sum_{i \in \Z}  \H_0(L_i \cap (\rho_\Theta(\gamma) +(0,Y) ) ) 
\right] = \textcolor{\blue}{ (2/(\pi a)) } \H_1(\gamma).
\label{0501a}
\eea
\end{theo}
This result is well known, though not all of the proofs
 in the literature are complete.   It can be deduced
from
% Lemma \ref{th:noodle},
 \textcolor{\blue}{Lemma \ref{lemBuffon},}
 but we do not give the details here.
For further discussion and a proof of Theorem \ref{th:noodle} in
the piecewise $C^1$ case, see
\cite{v1}.

\begin{proof}[Proof of Lemma \ref{touchlemrect}]
Given rectifiable curves $\gamma$ and $\gamma'$ in $\R^2$, let
$\gamma \; \square \; \gamma'$ be the set of points at which
$\gamma$ crosses $\gamma'$. It is proved in 
\cite{Maak} that 
$$
\int_{-\pi}^\pi \int_{\R^2}
 \H_0 (\gamma \; \square \; (\rho_\theta(\gamma') +x)) 
dx d \theta = 4 \H_1(\gamma) \H_1(\gamma'),
$$
and combined with (\ref{eqBuffon}), this gives us  the result.
\end{proof}

In the rest of this section we prove
 Lemma
\ref{touchlemC1}.

Given $C^1$ curves $\gamma$ and $\gamma'$ in $\R^2$,
%a necessary condition for them to touch
we shall say that they {\em graze}
 at a point $z \in \R^2$, if
 $z \in \gamma \cap \gamma'$ 
but $z$  is not one of the endpoints of
$\gamma$ or $\gamma'$, 
 and
 $\gamma$, $\gamma'$
have a common tangent line at $z$ (in \cite{v1} we used the
term `touch' for this notion). We say that $\gamma$ and $\gamma'$
{\em graze} if they graze at $z$ for some $z \in \R^2$.
%We say that $\gamma$ and $\gamma'$ {\em touch} if there
%exists $z \in \R^2$ such that they touch at $z$.

We say that a $C^1$ curve
 $\gamma$ in $\R^2$ is
 {\em almost straight}, if all
 lines tangent to $\gamma$ are  an
angle of at most $\pi/99$ to each other.
Observe that if $\gamma$ is almost straight, then there exists
$\theta \in [-\pi,\pi)$ such that $\rho_\theta(\gamma)$
is the graph of a \textcolor{\blue}{ $C^1$  function } defined on an interval. 
%We shall use this observation repeatedly.

%This is because if $\gamma$ is not itself almost straight, it can  be
%broken into finitely many almost straight pieces.

%It is clear that (\ref{0416a}) holds unless the 
%curve $\rho_\theta(\gamma) + (0,y)$
%touches the $x$-axis $L$ at some point. 

\begin{lemm}
\label{lemtouch}
Suppose $\gamma$ and $\gamma'$ are $C^1$ curves in $\R^2$,
and $\gamma$ is almost straight. Assume
there exist an interval $[a,b] $ and a function $f \in C^1([a,b])$ such that 
\bea
\gamma = \{(x,y): a \leq x \leq b, y = f(x) \}
\label{polar1}
\eea
with $f'(x) =  0$ for some $x \in (a,b)$.
Then 
\bea
\textcolor{\blue}{
\int_{-\infty}^{\infty} {\bf 1} \{ \gamma' +(0,v) 
\mbox{\rm ~ grazes  } \gamma \} dv} =0.
\label{0418a}
\eea 
\end{lemm}

\begin{proof}
Without loss of generality we may assume $\gamma'$ is also almost straight,
since if not, we may break $\gamma'$ into finitely many almost straight pieces.
We claim that we may also assume that the locus  of $\gamma'$ 
takes a similar form to that of $\gamma$, namely
\bea
\gamma' = \{(x,y): a' \leq x \leq b', y = g(x) \}
\label{polar2}
\eea
for some $a' < b'$  and
and for some  $C^1$ function $g: [a',b'] \to \R$.
Indeed, if $\gamma'$ cannot be expressed in this
form then it must %at some point
 have a vertical tangent line somewhere, 
but in this case, since both $\gamma$ and $\gamma'$
are assumed almost straight and $\gamma$ has at least one horizontal
tangent line,
%is almost horizontal, 
it is 
impossible for  any translate of $\gamma'$ to be tangent to $\gamma$.

By (\ref{polar2}) we have for  all $v \in \R$ that 
$$
   \gamma' + (0,v) = \{(x,y): a' \leq x \leq b', y = 
g(x) + v \}
$$
and therefore comparing with (\ref{polar1}) we see that
if $\gamma' +(0,v)$ grazes $\gamma$, then  for some $x \in
[a,b] \cap [a',b']$, we have both
%\begin{itemize}
%\item 
$g(x)+ v = f(x)$, and
%\item
$g'(x) = f'(x)$.
%\end{itemize}
In other words, setting $h = f-g$, we have that 
\bea
\{v: \gamma' + (0,v)
\mbox{ grazes } 
\gamma \} \subset 
%\cup_{x \in [a,b]} 
%\cup
 \{h(x) : x \in [a,b] \cap [a',b']   \mbox{ and } h'(x) =0 \}
\nonumber
\\
= h(A),
\label{0418b}
\eea
where we set 
$A := \{x\in [a,b ] \cap [a',b']: h'(x) =0\}$.
Given $n \in \N$, divide $[a,b] \cap [a',b']$ into
$n$ intervals of equal length, denoted $I_{n,1},\ldots,I_{n,n}$.
Let $A_n$ be the union of those intervals $I_{n,i}$
 having non-empty intersection with $A$. 
That is, set  $A_n  := \cup_{i \in {\cal I}_n} I_{n,i}$ with 
$$
{\cal I}_n := \{i \in \{1,\ldots,n\}: I_{n,i} \cap A \neq \emptyset\}.
$$ 
Since $h \in C^1$, the derivative $h'$ is uniformly
continuous on $[a,b] \cap [a',b']$. 
Therefore, given $\eps >0$, we can choose $n$ large enough so that
for all $i \in {\cal I}_n$ we have $|h'(x)|  \leq \eps$ for all $x \in
I_{n,i}$. Hence, for all such $i$,
by the mean value theorem, 
with $\H_1$ denoting Lebesgue measure we have
$\H_1(h(I_{n,i})) \leq  \eps \H_1(I_{n,i})$. Thus
$$
\H_1(h(A) ) \leq \H_1(h(A_n)) 
\leq \sum_{i \in {\cal I}_n} \H_1(h(I_{n,i}))
\leq \eps \sum_{i \in {\cal I}_n} \H_1(I_{n,i}) \leq \eps 
(b - a),
$$
and hence, since $\eps$ is arbitrarily small,
$\H_1(h(A)) =0$. Therefore by (\ref{0418b}) we have
(\ref{0418a}), as required.
\end{proof}

\begin{proof}[Proof of Lemma \ref{touchlemC1}]
\textcolor{\blue}{
%It suffices to prove the result in the case where $\gamma$ and
%$\gamma'$ are $C^1$ (not just piecewise $C^1$).
%In fact, 
We can and do
assume without loss of generality that both $\gamma$ and $\gamma'$ are 
$C^1$ (not just piecewise $C^1$), and  moreover that they are 
 almost straight.}
It is enough to prove the result for
 $\rho_\theta (\gamma)$ and
 $\rho_\theta (\gamma')$ for some $\theta \in (-\pi,\pi]$ (rather than
for the original $\gamma$, $\gamma'$) and
therefore we can (and do) also assume
there exists an interval $[a,b] $ and function $f \in C^1([a,b])$ such that 
(\ref{polar1}) holds.

Under these assumptions, applying  (\ref{0418a}) to the
curve $\gamma'_x:= \gamma' + (x,0)$ instead of $\gamma'$ shows that
\bean
\int_{-\infty}^\infty \int_{-\infty}^\infty 
{\bf 1} \{ \gamma' +(x,y) 
\mbox{ grazes } \gamma \} 
dy dx
= 
\int_{-\infty}^\infty \int_{-\infty}^\infty 
 {\bf 1} \{ \gamma'_x +(0,y) \mbox{ grazes } \gamma  \} dy dx
=0,
\eean
and hence the set of $z = (x,y)$ such that $\gamma'+z $ grazes $\gamma$
is Lebesgue null.

Let $\gamma_0,\gamma_1$ denote the endpoints of $\gamma$ and
$\gamma'_0,\gamma'_1$ the endpoints of $\gamma'$. If $\gamma$ and
$\gamma' +z$ touch but do not graze for some 
$z \in \R^2$, then either $\gamma_i \in \gamma'+z$ 
or $\gamma'_i + z \in \gamma$ for some $i \in \{0,1\}$. Since
$\gamma'$ is rectifiable we have  for $i =0,1$ that
%$$
%\int_{\R^2} {\bf 1}_{\gamma} (\gamma'_i + z )  dz = \H_2(\gamma) =0;
%% ~~~i=0,1 
%$$
%and
$$
\int_{\R^2} {\bf 1}_{\gamma'+z} (\gamma_i )  dz = 
\int_{\R^2} {\bf 1}_{\gamma'} (\gamma_i + ( -z )  )  dz = 
\H_2(\gamma') =0, 
%~~~i=0,1 
$$
and similarly
$\int_{\R^2} {\bf 1}_{\gamma} (\gamma'_i + z )  dz = 0$.
Hence the set of $z \in \R^2$ such that $\gamma$ and $\gamma'+z$ touch
 but do not graze is also Lebesgue null.
\end{proof}

\section{Proof of results for the DLRM}
\label{secproofsDLRM}
\allco
Throughout this section
 $(S,M)$ and  $(S',M')$ denote independent random elements of
$\cK \times \bM$ \textcolor{\blue}{with common distribution $\Q'$.}
% and $R$ is as in Section \ref{subsecnotation}.

\begin{lemm}
\label{Boolem}
Suppose $\Q$ is such that $\E[\H_d(S \oplus B(1)) ]< \infty$.
Then for any $K \in (0,\infty)$, with probability 1 only
finitely many of the sets $S_j+x_j$ with $-K \leq t_j \leq K $
have non-empty intersection with $B(K)$.
\end{lemm}
\begin{proof}
The number of  sets $S_j + x_j$ that intersect $B(1)$
and have $|t_j| \leq K$ is Poisson with mean
\bean
2K \int_{\R^d} dx \int_{\cK} \Q(d \sigma) {\bf 1}\{ 
(\sigma +x) \cap B(1) \neq \emptyset\} 
= 2K  \int_{\cK} \Q(d\sigma) \int_{ \R^d} 
{\bf 1}_{\sigma  \oplus B(1) } (-x) dx,
\eean
 which is finite by the assumption 
 $\E[\H_d(S \oplus B(1))] < \infty$.
Hence, almost surely, $(S_j+x_j) \cap B(1) \neq \emptyset$
 for only finitely many $j$
 with $-K \leq t_j \leq K $.
 Since
we can cover $B(K) $ with finitely many translates of $B(1)$, the
result follows.
\end{proof}

\begin{lemm}
Suppose $\E[\H_d(S\oplus B(1)) ] < \infty$. Then 
the time-reversed DLRM  $\xi$ is indeed a random measure, and so is
the DLRM  $\xi_t$ 
for all $t \in \R$.
\label{lemmeas}
\end{lemm}
\begin{proof}
We prove just the first assertion (the proof of the second assertion
is similar).
It suffices to show,
for arbitrary bounded Borel $A \subset \R^d$,
 that $\xi(A)$ is a random variable.
By the definition (\ref{xidef}),
$$
\xi(A) = \sum_{i=1}^\infty (M_i+ x_i)(A
 \setminus \cup_{\{j: 0 \leq t_j < t_i\}} 
(S_j+x_j) ) {\bf 1}\{ t_i \geq 0\},
$$
and it suffices to prove that each summand is a random variable.
Choose $r_1$ such that $A \subset B^o(r_1)$, where $B^o(r_1)$ is the
interior of the ball $B(r_1)$.
Fix $i \in \N$.  By Lemma \ref{Boolem},
with probability 1 
only finitely many of the sets $S_j + x_j $ with $0 \leq t_j < t_i$
 have non-empty intersection with $B^o(r_1)$.
 Therefore 
%since the sets $S_j$ are all closed
the set
$$
U := B^o(r_1) 
\setminus \cup_{\{j: 0 \leq t_j < t_i\}} 
(S_j+x_j)
$$
is open. 

Given $n \in \N $, partition $\R^d$ into cubes of the form
$[0,2^{-n})^d  + 2^{-n} z$ with $z \in \Z^d$.
%half-open rectilinear 
%cubes of side $2^{-n}$
%with some of them having a corner at the origin. 
Let
 the cubes in the partition that are contained in $B^o(r_1)$ be denoted 
$Q_{n,1},\ldots,Q_{n,m_n}$.
Let $U_n $ be the union of cubes of the form $Q_{n,k}$
with $1 \leq k \leq m_n$ and $Q_{n,k} \subset U$. Since $U$ is
open we have $U_n \uparrow U$, and so  by monotone convergence,
\bean
(M_i +x_i) (A 
 \setminus \cup_{\{j: 0 \leq t_j < t_i\}} 
(S_j+x_j) )
 {\bf 1}\{ t_i \geq 0\}
= \lim_{n \to \infty} (M_i + x_i) 
( A \cap U_n)
 {\bf 1}\{ t_i \geq 0\}
 \\
 = \lim_{n \to \infty} \sum_{k=1}^{m_n} (M_i +x_i) (A \cap Q_{n,k}) 
{\bf 1}\{ Q_{n,k} \cap \cup_{\{j:0 \leq t_j < t_i\}} (S_j + x_j) = \emptyset \}
 {\bf 1}\{ t_i \geq 0\},
\eean
and we claim that this is
 a random variable. For example, if
% $d=2$ and
$Q_{n,k}= [0,1)^d$, then for each $j$ we have
$$
\{(S_j + x_j) \cap Q_{n,k} =\emptyset \} =
\cap_{n=2}^\infty \{ 
(S_j + x_j) \cap [0,1-1/n]^d \}  =\emptyset \} 
$$
which is an event by the definition of the Fell topology 
and the fact that $S_j+x_j$ is a random element
of $\K$ by \cite[Theorem 2.4.3]{SW}, for example.
\end{proof}
\begin{lemm}
\label{Elem}
For $(x,t) \in \R^d \times \R_+$ let
$E_{x,t}$ be the event that the site $x$ is exposed
(i.e., not already covered)
just before time $t$, in the time-reversed DLM. 
 That is, let
\bea
E_{x,t} := \{ x \notin \cup_{\{i: 0 \leq t_i <t \} } (S_i + x_i) \}.
\label{Edef}
\eea
Then
with $\lambda$ and $\lambda_{x}$ defined at (\ref{lambdadef}),
for all $x,y \in \R^d$, $t \geq 0$ and  $u \in [t,\infty)$ we have 
\bea
\Pr[E_{x,t} \cap E_{y,u} ] = \exp(-\lambda_{y-x} t- \lambda(u-t)).
\label{0711a}
\eea
 In particular $\Pr[E_{x,t} ] = \exp(-\lambda t)$.
\end{lemm}
\begin{proof}
%It suffices  to 
We first prove (\ref{0711a})
% the second assertion 
in the case with $u=t$.
%, since then the first is
%  obtained by taking $y=x$.
%As in the preceding proof,
%write $\Po= \{(x_i,t_i), i \geq 1 \}$.
%write $\Po= \sum_{i=1}^\infty \delta_{(x_i,t_i)}$.
The number of $i$ with $\{x,y\} \cap  (S_i +x_i) \neq \emptyset$ and
 $t_i \in [0,t)$ is Poisson distributed with parameter
\bean
t \int_{\cK} \int_{\R^d} {\bf 1} \{ \{x,y\} \cap (\sigma + z) \neq \emptyset 
\} dz \Q(d \sigma) 
= t \int_{\cK} \int_{\R^d}  {\bf 1}_{(\sigma +(-x)) \cup (\sigma + (-y)) }(-z) 
dz \Q(d\sigma)
\\
= 
 t \int_{\cK}  \H_d((\sigma +(-x)) \cup (\sigma + (-y)) ) \Q(d\sigma) 
%= \E[\H_2(S \cup (S+x-y) ) ]
 = t \lambda_{x-y}, 
\eean
and since $\lambda_{x-y}= \lambda_{y-x}$ this gives us (\ref{0711a})
for $u=t$.
%shows that $\Pr[E_{x,t} \cap E_{y,t}] =
%\exp(- \lambda_{y-x} t)$, as claimed.
 Taking $y=x$ gives us also
that $\Pr[E_{x,t}] = e^{-\lambda t}$. Finally, for $u >t$, by 
the independence property and time-homogeneity 
of the Poisson process $\sum_i \delta_{(x_i,t_i,S_i)}$ we have that 
$$
\Pr[E_{x,t} \cap E_{y,u}] = \Pr[E_{x,t} \cap E_{y,t} ] \Pr[E_{y,u-t}]
= \exp(-t\lambda_{y-x} ) \times \exp(-(u-t) \lambda),
$$
which gives us (\ref{0711a}) in  general.
\end{proof}
\begin{proof}[Proof of Theorem \ref{th:momgen}]
By Lemma \ref{lemmeas}, $\xi$ is a random measure.
It is easy to see that this random measure
  is  stationary.
 Recalling (\ref{xidef}),
and using the Mecke formula (see \cite{LP}), and the notation
from (\ref{Edef}),
 we have
that 
\bean
\alpha = \E[ \xi([0,1]^d)] \\ 
= \int_{\cK \times \bM}
\Q'(d (\sigma,\mu))
\int_0^\infty  dt
\int_{\R^d}  dx
\int_{\R^d}  
\mu(dy) {\bf 1}_{[0,1]^d}(x+y) \Pr[ E_{x+y,t} ]. 
\eean
Hence by Lemma \ref{Elem}
and Fubini's theorem,
%Given $(t,z) \in \R \times \R^d$,
%the number of points $(x_i,t_i)$ of $\Po$ with
%$0 < t_i <t$ and $z \in S_i + x_i$  
%is Poisson with parameter
%$
%t \int_{\cK} \int_{\R^d} {\bf 1}
%\{z \in \sigma + x \} dx \Q(d \sigma)  = t \lambda 
%$
%by (\ref{lambdadef}),
%so that $\Pr[E_{z,x}]= e^{-\lambda t}$. Hence,
\bean
\alpha = \int_{\cK \times \bM} \int_{\R^d}
\int_0^\infty  
\exp(- \lambda t) 
dt \mu(dy) 
\Q'(d (\sigma,\mu))
%\nonumber \\
= \lambda^{-1} \E    [|M|].
\eean
That is, we have (\ref{genmean}). 
%We note that Fubini's theorem
% is applicable in the above, because  the random measure $M$
%on  $\R^d$ is almost surely finite.
\end{proof}

\begin{proof}[Proof of Theorem \ref{thcovgen}]
Let $f \in \cR_0$ and $t \in \R$.
We first prove (\ref{0509a2}) in the special case
where  $g=f$ and $u=t$.  For each $n$ set $f_n := T_n(f)$
and  $Z_n:=\xi(f_n)$.  
Then
$\E [ Z_n^2 ] = a_n + b_n$, where we set
\bean
a_n :=
\E \sum_{\{i:t_i \geq 0\}} 
\left(
\int_{\R^d} f_n(x) ( 1- {\bf 1}_{\cup_{\{k: 0 \leq t_k < t_i\}} (S_k + x_k) }   
(x) )(M_i + x_i)(dx)
\right)^2,
\eean
and
\bean
b_n := \E \sum_{\{i: t_i \geq 0\}} \sum_{\{j: t_j \geq 0, j \neq i\}}
\int_{\R^d} f_n(x) ( 1 - {\bf 1}_{ \cup_{\{k: 0 \leq t_k < t_i\}}(S_k+x_k) }  
(x)) (M_i +x_i)( dx)   
 \\
\times
\int_{\R^d}  f_n(y) (1 - {\bf 1}_{\cup_{\{k: 0 \leq t_k < t_j\}} ( S_k + x_k) } 
(y) ) 
 (M_j + x_j) (dy). 
\eean
\textcolor{\blue}{
In what follows, we adopt the convention that any unspecified
domain of integration is taken to be $\R^d$. 
By the Mecke formula 
\bean
a_n = 
\int_{\cK \times \bM} \Q'(d(\sigma,\mu)) \int dx \int_0^\infty dt 
\\
\times
\E \left[ \left(
\int f_n(y) ( 1- {\bf 1}_{\cup_{\{k: 0 \leq t_k < t \}} } (S_k+x_k) (y) ) 
(\mu + x)(dy) \right)^2 \right], 
\eean
so using notation $E_{x,t}$  from Lemma \ref{Elem},
%\bean
%a_n = 
%\int_{\cK \times \bM} \Q'(d(\sigma,\mu)) \int dx \int dt 
%\E \left[ \left(
%\int f_n(y)   {\bf 1}_{E_{y,t}}  
%(\mu + x)(dy) \right)^2 \right] 
%\eean
%Hence,
\bean
a_n = 
\int_{\cK \times \bM} \Q'(d(\sigma,\mu)) \int dx \int_0^\infty dt 
\E \left[ 
\int
  {\bf 1}_{E_{y,t}  }  
 \times  f_n(y)
(\mu + x)(dy)
\right.
\\
\left.
 \times \int 
 {\bf 1}_{E_{z,t}} 
\times f_n(z)
(\mu + x)(dz)
  \right].
\eean
Changing variables to $\tily = y-x$ and $\tz= z-x$
 we obtain that
\bean
a_n = 
\int_{\cK \times \bM} \Q'(d(\sigma,\mu)) \int dx \int_0^\infty dt 
\E \left[ 
\int 
  {\bf 1}_{E_{x+\ty,t}  }  
f_n(x + \ty )
\mu (d \tily)
\right.
\\
\left.
 \times \int
 {\bf 1}_{E_{x+\tz,t}} 
 f_n(x+ \tz)
\mu (d\tz)
  \right].
\eean
%\bean
%a_n = 
%\int_{\cK \times \bM} \Q'(d(\sigma,\mu)) \int dx \int dt 
%\\
%\times
%\E \left[ \left(
%\int f_n(x+y) ( 1- {\bf 1}_{\cup_{\{k: 0 \leq t_k < t_j \}} } (S_k+x_k) (y) ) 
%(\mu + x)(dy) \right)^2 \right].
%\eean
Now writing $y $ for $\ty$ and $z$ for $\tz$, using Fubini's theorem 
%By the Mecke formula
 and Lemma \ref{Elem},
 followed by a change of variable
$\tx = x + y$,  we have
}

%the first expectation on the above right hand side equals
\bean 
a_n =
\int_{\cK \times \bM} \Q'(d(\sigma,\mu)) 
 \int dx \int_0^\infty dt 
\int_{\R^d} \mu(dy)
\int_{\R^d} \mu(dz) e^{-\lambda_{z-y} t } f_n(x+y) f_{n}(x+z)
\\
=
\int_{\cK \times \bM} \Q' (d(\sigma,\mu)) 
\int_{\R^d} \mu(dy)
\int_{\R^d} \mu(dz) \lambda_{z-y}^{-1 }
\int  d\tx f_n(\tx) f_{n}(\tx+z-y).
\eean
%Here and elsewhere, we use the convention that any unspecified
%domain of integration is taken to be $\R^d$. 
\textcolor{\blue}{Using the} further change of variables $x' := n^{-1/d} \tx$,
we have for almost all $x' \in \R^d$ and all $(z,y) $
that $f_n(n^{1/d}x' + z -y ) \to f(x')$ as $ n \to \infty$
(because we assume $f \in \cR_0$), 
so by the dominated convergence theorem
%, we find that 
$ n^{-1} a_n \to v_4 \|f\|_2^2, $
with $v_4$ given by (\ref{B5def}),
and $v_4$ is finite because we assume $\E[ |M|^2] < \infty$,
\textcolor{\blue}{ and
because $\lambda_z \geq \lambda >0$ for all $z \in \R^d$.}

By the multivariate Mecke formula
 (see 
\textcolor{\blue}{
e.g. \cite[page 30]{LP}), 
\bean
b_n = 2 \int_{\cK \times \bM} \Q' (d(\sigma,\mu))
 \int_{\cK \times \bM} \Q' (d(\sigma',\mu'))
\int_0^\infty ds \int_s^\infty dt \int du \int dv 
\\
\times \E \int  {\bf 1}_{E_{x,s}} f_n(x)(\mu+u)(dx)
 \int  {\bf 1}_{E_{y,t}} f_n(y)
%{\bf 1}_{(\sigma+ u)^c}(y) 
{\bf 1}\{y \notin \sigma +u\}
 (\mu'+v)(dy).
\eean
Then by Fubini's theorem and the changes of variables $\tx = x-u$ and $\ty = y- v$,
\bean
b_n = 2 \int_{\cK \times \bM} \Q' (d(\sigma,\mu))
 \int_{\cK \times \bM} \Q' (d(\sigma',\mu'))
\int_0^\infty ds \int_s^\infty dt \int du \int dv 
\\
\times \int \int f_n(u+ \tx) f_n(v + \ty) \Pr[E_{u+\tx,s} \cap E_{v + \ty,t} ] 
{\bf 1}\{ v+ \ty \notin \sigma + u\} \mu(d\tx) \mu'(d\ty).
\eean
By the law of total probability,
the two outer integrals  may
be written as expectations with respect to $(S,M)$ and $(S',M')$.
%a double  integral over $(\cK \times \bM)^2$. 
%We may
%re-write  this double integral as an expectation with respect to
%$(S,M,S'M')$.
 Using also 
 Lemma \ref{Elem}, and writing just $x$ for $\tx$ and $y$ for $\ty$,
 we obtain that
}
\bean
b_n = 2  \E \int du \int dv \int_{0}^\infty ds \int_s^\infty dt 
 \int_{\R^d} M  (dx) 
\int_{\R^d} M' (dy) 
f_n(u+x) 
f_n (v+y)
\\
\times {\bf 1} \{v+y \notin S  +u\} \exp( -\lambda_{v+y-u-x}s )
e^{-\lambda (t-s)} 
\\
= 2 \E \int du \int dv \int
M(dx)
 \int
M'(dy)
{\bf 1} \{v+y \notin S+u\} \lambda_{v+y-u-x}^{-1} \lambda^{-1} 
\\
\times
f_n(u+x) f_n(v+y).
\eean
On the other hand,
% $(\E Z_n )^2 = n^2 \alpha^2$, with
 $(\E Z_n )^2 =  \lambda^{-2}  (\E|M|)^2 (\int f_n(x)dx)^2$
by Theorem \ref{th:momgen} and Campbell's formula (see e.g.
\textcolor{\blue}{ \cite[page 128]{LP}}).
%$\alpha$ given at  (\ref{genmean}). 
Setting $w=v-u$ and $\tilde{u}= u+x$, we may deduce that 
\bean
b_n - (\E Z_n)^2  =
%\lambda^{-1} \E
%\int du \int dw \int M(dx) \int
%M'(dy) f_n(u+x) f_n(u+w+y) 
%\\
%\times \left( 2\lambda_{w+y -x} ^{-1} {\bf 1}\{ w+y \notin S \}
%- \lambda^{-1} \right).
%\eean
%Taking $\tilde{u}= u+x$ shows that this equals
%\bean
\lambda^{-1}
\E \int d\tu \int dw \int M(dx) \int M' (dy)
f_n(\tu) 
f_n(\tu + w + y  -x)
\\ \times
 (2\lambda_{w+ y-x}^{-1} {\bf 1}\{w+ y \notin S\} - \lambda^{-1} ). 
\eean
Now take $u'= n^{-1/d}\tu$. For almost every 
$u' \in \R^d$, and all $(w,x,y)$, we have that
% fixed $(u',w,x,y)$ that
 %${\bf 1}_{W_n}(n^{1/d}u' + w + y -x) = {\bf 1}_{W_1}(u')$ for
 $f_n(n^{1/d}u' + w + y -x) \to f(u')$ as $n \to \infty$.
Hence using dominated convergence, we find that 
\bean
n^{-1} ( b_n - (\E Z_n)^2) \to
\lambda^{-1} \|f \|_2^2 \E \int dw \int M(dx)
\int M'(dy) \left(
 \frac{2 {\bf 1}\{w+ y \notin S\} }{\lambda_{w+y-x}} - 
\frac{1}{\lambda} \right) 
\\
= (v'_5 - v'_6) \|f\|_2^2,
\eean
where we set
\bean
v'_5:= \lambda^{-1} \E \int dw \int M(dx)
\int M'(dy) \left(
 \frac{2  }{\lambda_{w+y-x}} - \frac{1}{\lambda} \right),
\eean
and 
$$
v'_6 := \lambda^{-1} \E \int dw \int M(dx)
\int M'(dy) \left(
 \frac{2 {\bf 1}\{w+ y \in S\} }{\lambda_{w+y-x}}  
 \right). 
$$
Then we obtain that $n^{-1}\Var [Z_n] \to ( v_4 + v'_5 - v'_6)
\|f\|_2^2$.

Now setting $w' := w+y$, we obtain that
\bean
v'_6 = \lambda^{-1} \E \int M(dx)
 \int M'(dy) \int_{S} d w' (2/\lambda_{w'-x}) 
= v_6,
\eean
with $v_6$ given by (\ref{B7def}).
Also, setting $v= w+y -x$ yields that
\bean
v'_5 = \lambda^{-1} \E  \int M(dx) \int
M'(dy) \int dv
 \left( \frac{2 \lambda - \lambda_{v} }{ \lambda \lambda_v} \right)
= v_5,
\eean
where $ v_5$ is given by (\ref{B6def}).
The integral in (\ref{B6def})  is finite because
$(2\lambda - \lambda_x)/\lambda_x$ is bounded
above by a constant times $\E[\H_d(S \cap (S +x))]$, and
with $R$ given by (\ref{Rdef}),
\bean
\E \int \H_d( S \cap (S+v) ) dv \leq \E\left[  \H_d(S) 
 \int_{B(2R)}  dv 
\right]
\leq 2^d \pi_d^2 \E[ R^{2d} ].
\eean
Thus we have the case $t=u$ and $f=g$
  of (\ref{0509a2}).
We can then deduce the case 
  of (\ref{0509a2}) with  $t=u$ but
with general $f,g  \in \cR_0$,  by 
polarisation (\textcolor{\blue}{see e.g. \cite[page 192]{LP}}).
%argument (apply the previous case
%to $f+g$ and $f-g$).

Finally we need to prove (\ref{0509a2}) in general.  Without loss of
generality, we assume $u >t$. 
Write $g_n $ for $T_n(g)$. Then
we may write $\xi_u (g_n) := X +Y$, where 
$$
X := 
\int g_n(x)
( 1 - {\bf 1}_{\cup_{i: t < t_i \leq u} (S_i+ x_i)} (x) ) \xi_t(dx ),   
%\xi_t( W_n \setminus 
$$
 and $Y$ denotes the sum, over those  $i$ for which $t < t_i \leq u$, 
 of the integral of $g_n$  with respect to
the measure
$(M_i + x_i)$
 restricted to regions which do not subsequently get covered
between times $t_i$ and $u$, i.e.
$$
Y := \sum_{i: t < t_i \leq u } 
\int g_n(x) (1 - 
{\bf 1}_{ \cup_{\{j: t_i < t_j \leq u\} } (S_j + x_j) } (x) )  
(M_i + x_i)(dx) .
$$
 Let $\F_t$ denote the
$\sigma$-algebra generated by all Poisson arrivals
 \textcolor{\blue}{ and associated marks}
 up to time $t$.   Then
by Lemma \ref{Elem},
\bean
\E[X| \F_t] = \int_{\R^d} g_n(x) 
e^{-\lambda(u-t)} 
\xi_t(dx) 
= e^{\lambda(t-u)} \xi_t( g_n) .
\eean
Also $Y$ is independent of $\F_t$ and by the Mecke formula and
Fubini's theorem,
\bean
\E[Y|\F_t] = \E[Y] = \E \int dx \int_{t}^u ds  \int M(dy)
g_n( x + y) \exp(- \lambda(u-s))
\\
=n \lambda^{-1}(1- e^{\lambda(t-u)} ) \E [|M|] \int g(x) dx.
\eean
Hence, 
\bea
\E[ \xi_u(g_n) |\F_t]  =   \E[X+Y|\F_t] 
~~~~~~~~~~~~~~~~
~~~~~~~~~~~~~~~~
~~~~~~~~~~~~~~~~
~~~~~~~~~~~~~~~~
\nonumber
\\
  =  \left( n\lambda^{-1} \E [|M| ] \int g(x)dx \right)  + 
e^{\lambda(t-u)} \left( \xi_t(g_n) - n \lambda^{-1} \E [ |M| ] \int g(x)dx 
\right).  
~~~~~
\label{0824a}
\eea
\textcolor{\blue}{
For all $h \in \cR_0$, 
 set $\txi_t(h) = \xi_t(h) - \E \xi_t(h) $.
By Theorem \ref{th:momgen} and Campbell's formula,
$\E[\xi_t(h)] =   \lambda^{-1} \E [|M|] \int h(x) dx$. 
Therefore by
 (\ref{0824a}),   $\E[\txi_u(g_n)|\F_t]= e^{\lambda(t-u) } \txi_t(g_n) $. Thus,
since $\xi_t(f_n)$ is $\F_t$-measurable,
}
%we obtain
\bean
n^{-1} \Cov(\xi_t(f_n),\xi_u(g_n) ) = n^{-1} \E[\txi_t(f_n) \txi_u(g_n) ] 
= n^{-1} \E[\txi_t(f_n) \E[\txi_u(g_n)|\F_t] ] 
\\
= n^{-1} \E[\txi_t(f_n)  e^{\lambda(t-u)} \txi_t(g_n) ],
%\\
%= e^{\lambda(t-u)} \sigma^2.
\eean
and by  the case of (\ref{0509a2}) already proved, this
tends to $ e^{\lambda(t-u)} \sigma_0^2 \langle f , g \rangle$ as
 $n \to \infty$.
Thus we obtain
  the general case of (\ref{0509a2}).
\end{proof}

We now work towards proving Theorem \ref{th:CLTgen}.
Recall the notation $W_1:= [0,1]^d$.

\begin{lemm}
\label{lem0207a2}
Suppose for some $q \geq 1, r_0 \in (0,\infty)$ that
 $\E[|M|^{q} ] < \infty$, and almost surely,
 $R \leq r_0$ and $M$ is supported by
the ball $B(r_0)$.
Then 
$
\E[\xi(W_1)^{q} ] < \infty. 
$
\end{lemm}
\begin{proof}
Let ${\cal I} :=  \{i \in \N: (x_i,t_i) \in B(r_0 +d) \times \R_+\}$,
\textcolor{\blue}{
and enumerate this set as ${\cal I} = \{j(1),j(2),j(3),\dots\}$
with $0 \leq t_{j(1)} < t_{j(2) } < t_{j(3)} < \cdots$.
Let}
\bea
N: = \min \{k: W_1 \subset 
\cup_{i=1}^k
(S_{j(i)} + x_{j(i)}) \}. 
%S_{j(2)} + x_{j(2)}, \ldots, 
%S_{j(k)} + x_{j(k)}$ cover $W_1$.  
\label{Ndef}
\eea
\textcolor{\blue}{
We prove first that all moments of $N$ are finite.
% $\E[N^m] < \infty$ for any $m \in \N$. 
For $x \in \R^d,$ $ r>0$ set $B(x,r):= B(r) +x$.
Recall from Section \ref{secDLRM} 
 that we assume $\Q'$ has first marginal
$\Q$,  and from Section \ref{subsecoverview} 
that we assume there exist $x^* \in \R^d$ and $r_2 >0$
such that $\Q(\{\sigma: B(x^*,2r_2 ) \subset \sigma \} ) >0$.
We may assume without loss of generality that $x^*$ is the origin.
}

\textcolor{\blue}{
Choose $m \in \N$ and $y_1,\ldots,y_m \in \R^d$ such that
$
W_1 \subset \cup_{i=1}^m  B(y_i,r_1).
$
For  $k=1,2,\ldots, m $ set 
$$
N_k := \min\{n \in \N: x_{j(n)} \in B(y_k,r_2) {\rm ~ and~ }
B(2r_2) \subset S_{j(n)}  \}.
$$
%Let 
%$$
%N_1 := \min\{n \in \N: x_{j(n)} \in B(y_1,r_2) {\rm ~ and~ }
%B(0,2r_2) \subset S_{j(n)}  \}.
%$$
%Then 
%set
%$$
%N_2 := \min\{n \in \N: x_{j(N_1+ n)} \in B(y_2,r_2) {\rm ~and~} B(0,2r_2)
%\subset S_{j(N_1+n)}   \} 
%$$
% successively for each $k=2,3,\ldots, m $ set $\nu_{k-1} = N_1+ \cdots
%+ N_{k-1}$, and set
%% up to $k=m$,
%$$
%N_k := \min\{n \in \N: x_{j(\nu_{k-1}+ n)} \in B(y_k,r_2)
%{\rm ~ and~} B(0,2r_2) \subset S_{j(\nu_{k-1})+n}    \}. 
%$$
%Then $B(y_k,r_2) \subset x_{j(N_1 + \cdots + N_k)} +S_{j(N_1+ \cdots + N_k)}$
%for each $k$, and hence $N \leq \sum_{k=1}^m N_k$. 
Then $B(y_k,r_2) \subset 
S_{j(N_k)} + x_{j(N_k )}$ 
for each $k$, and hence $N \leq \max(N_1,\ldots,N_m)$. 
Also each of $N_1,\ldots,N_m$ has
a geometric distribution with strictly positive parameter.
It follows that $N$ has finite moments of all orders as claimed,
since for each $r \in \N$ we have $N^r \leq \sum_{k=1}^m  N_k^r$ 
and $\E[N_k^r ] < \infty$  for each $k$.
}

\textcolor{\blue}{
For $k \in \N$ set $Z_k := |M_{j(k)}|$.
Then} $\xi(W_1) \leq \sum_{i=1}^N Z_i$, and
by Jensen's inequality,
\bean 
\E \left[ \left( \sum_{i=1}^N Z_i \right)^q \right]
\leq \E \left[ N^q \left( N^{-1}  \sum_{i=1}^N Z_i^q \right) \right]
= \E \left[ N^{q-1} \sum_{i=1}^N Z_i^q \right].
\eean
For each $i \in \N$, let 
$$
N_i: = \min \{k: W_1 \subset 
\cup_{\ell=i+1}^{i+k}
(S_{j(\ell)} + x_{j(\ell)}) \}, 
$$
which has the same distribution as $N$.
Then $N \leq i+ N_{i}$ for each $i$, so that
\bean
\E[\xi(W_1)^q] \leq \E \left[ \sum_{i=1}^\infty Z_i^q
 {\bf 1}_{\{N \geq i\}} N^{q-1}
\right]
 \leq \sum_{i=1}^\infty \E [ Z_i^q {\bf 1}_{\{N \geq i\}} (i+ N_{i})^{q-1}].
\eean
For each $i$ the three random variables $Z_i$, ${\bf 1}_{\{N \geq i\}}$ and
$N_i$ are mutually independent, so
\bean
\E[\xi(W_1)^q] \leq
\sum_{i=1}^\infty \E[Z_1^q] \Pr[N \geq i]\E [2^{q-1}(i^{q-1} + N^{q-1}) ],
\eean
which is finite because $\E[Z_1^q]< \infty$ by assumption,
while $\E[N^q] < \infty$ as explained earlier.
%since $N$ can be stochastically dominated by
%an appropriate negative binomial random variable, since we 
% assume $\Q$ assigns  strictly positive measure to the collection of
%sets in $\cK$ having non-empty interior.
\end{proof}

Given a finite graph $G$ with vertex set $V$, we say
$G$ is a {\em dependency graph} 
\textcolor{\blue}{
for a collection of random variables
(or random vectors)}
$\{X_i, i \in  V \}$ if for all pairs of disjoint subsets $V_1,V_2$
of $V$ such that there are no edges connecting $V_1$ to $V_2$,
the random vectors $(X_i,i \in V_1)$ and $(X_i,i \in V_2)$
are independent of each other.
Let $|V|$ denote the number of elements of $V$.
To derive a central limit theorem we are going to use the following result
from \cite[Theorem 2.7]{CS}:
\begin{lemm}
\label{lem0207b}
Let $2 < q \leq 3$. Let $X_i$, $i \in  V$, 
be random variables indexed by the vertices of a dependency
graph with maximum degree $D$. Let $W = \sum_{i \in V} X_i$. Assume
that $\E[W^2] =1, $ $\E[X_i]=0$, and $\E[|X_i|^q] \leq \theta^q$
for all $i \in  V$ and some $\theta >0$. Then
$$
\sup_{t \in \R} | \Pr[W\leq t] - \Pr[\cN(0,1) \leq t]| \leq
75 D^{5(q-1)} | V | \theta^q.
$$
\end{lemm}
\begin{proof}[Proof of Theorem \ref{th:CLTgen} (a)]
For $t \in \R$ and
 $f \in \cR_0$, write $\txi_t(f) $ for $ \xi_t(f) - \E [\xi_t(f)] $.
Let $k \in \N$ and $f_1,\ldots, f_k \in \cR_0$, and
$ t_1,\ldots,t_k  \in \R$. Write $f_{n,i}$ for $T_n(f_i)$, $1 \leq i \leq k$.
By the Cram\'er-Wold theorem 
%\cite{Kallenberg},
\textcolor{\blue}{\cite[page 49]{Bill}}, it suffices to prove that
\bea
n^{-1/2} \sum_{i=1}^k \txi_{t_i}(f_{n,i})
 \tod \cN \left( 0, \sum_{i=1}^k
\sum_{j=1}^k 
\kappa_0((f_i,t_i),(f_j,t_j))
%\sigma_0^2 \exp(- \lambda(|t_i-t_j|)) 
\right).
\label{0508a}
\eea
By Theorem \ref{thcovgen},
\bea
\lim_{n \to \infty} 
\Var \left[ n^{-1/2} \sum_{i=1}^k  \txi_{t_i}(f_{n,i})\right] = 
\sum_{i=1}^k \sum_{j=1}^k 
%\sigma^2 \exp(- \lambda|t_i-t_j|)
\kappa_0((f_i,t_i),(f_j,t_j)).
\label{limvar}
\eea
Partition $\R^d$ into half-open rectilinear unit cubes,
and denote those cubes in this partition which intersect
the support of at least one of $f_{n,1}, \ldots, f_{n,k} $ by
 $Q_{n,1},\ldots,Q_{n,m_n}$. Since  $f_1, \ldots, f_k$ are all in
$\cR_0$ and therefore have bounded support, $m_n =
O(n)$ as $n \to \infty$.  
%Write $Q'_{n,j}: = Q_{n,j} \cap W_n$, $1 \leq j \leq  m_n$.

Given $n $, 
for $ 1\leq m \leq m_n$  set 
\bea
X_{n,m} := \frac{ 
\sum_{i=1}^k  \txi_{t_i}(f_{n,i} {\bf 1}_{Q_{n,m}})  }{ 
\sqrt{\Var \left[\sum_{j=1}^k  \txi_{t_j}(f_{n,j})\right] } }.
\label{Xndef}
\eea 
Suppose the right hand side of (\ref{limvar}) is strictly positive.
Then the denominator in (\ref{Xndef})  is $\Theta(n^{1/2}) $, and therefore
by Lemma \ref{lem0207a2} and the assumption that 
$\E[|M|^{2+\eps}] < \infty$, there is a constant $C$ such that
$$
\E[ |X_{n,m}|^{2+ \eps/2}] \leq C n^{-1-\eps/4}, ~~~ 
1 \leq m \leq m_n, ~ n \in \N.
$$
By our assumption that the random set $S$ and the
support of the  random measure  $M$ are uniformly bounded, the
indices $1 \leq m \leq m_n$
of the random variables $X_{n,m}$, 
 have a dependency graph structure with all vertex degrees bounded by
a constant independent of $n$. Moreover
 $$
\frac{\sum_{i=1}^k  \txi_{t_i}(f_{n,i}) }{\sqrt{\Var(
\sum_{j=1}^k  \txi_{t_j}(f_{n,j}))}}   = \sum_{m=1}^{m_n}
X_{n,m}.
$$ 
Therefore we obtain from  Lemma  \ref{lem0207b} that
$$
\sup_{t \in \R} \left( \Pr \left[ \frac{ \sum_{i=1}^k  \txi_{t_i}(f_{n,i}) 
 }{ \sqrt{\Var(\sum_{j=1}^k  \txi_{t_j}(f_{n,j}))} }
\leq t  \right] - \Pr[\cN(0,1) \leq t ] \right) = O(n^{-\eps/4}).
$$
Using (\ref{limvar}) again, 
we thus have  (\ref{0508a}), if the right hand side
of (\ref{limvar}) is strictly positive. If in fact   this limit is zero,
we still have (\ref{0508a}) by Chebyshev's inequality.
\end{proof}
The proof of Theorem \ref{th:CLTgen}(b)
is based on the following lemma.
%Recall that $W_1 :=[0,1]^d$.

\begin{lemm}
\label{lem4D}
Let $-\infty < a < b < \infty$.
Suppose the assumptions of Theorem \ref{th:CLTgen} (b) hold.
Let $f \in \cR_0$, and  
\textcolor{\blue}{for $n \in \N \cup \{0\}$ set $f_n:= T_n(f)$.
Then there are constants $C >0,$ $ \eps' >0$ such that
for all $s,t,u$ with $ a \leq s < t < u \leq b$,}
\bea
\textcolor{\blue}{
\E[ n^{-2} (\xi_t(f_n) - \xi_s(f_n))^2 (\xi_u(f_n)- \xi_t(f_n))^2 ]
\leq C  (u-s)^{1 + \eps'} .  }
\label{0605a}
\eea
\end{lemm}
\begin{proof}
Assume initially that  $0 \leq f(x) \leq 1$
for all $x \in \R^d$.  Also let $g \in  \cR_0$ with $0 \leq g(x) \leq 1$
for all $x \in \R^d$, and set $g_n:= T_n(g)$.
% Assume for now that also $f(x) \geq 0$ for all $x \in \R^d$. 

Partition $\R^d$ into half-open rectilinear unit cubes,
and for $n >0$,
 denote those cubes in this partition which intersect
the support of $f_{n} $ or the support of $g_n$ by
 $Q_{n,1},\ldots,Q_{n,m_n}$,
with the centres of these cubes denoted
 $q_{n,1},\ldots,q_{n,m_n}$ respectively.
Then $m_n = O(n)$ as $n \to \infty$.  
%
%For $1 \leq u \keq m_n$, let $q_{n,i}$ denote the centre of $Q_{n,i}$.

Let $a \leq s < t <u \leq b$. For $1 \leq i \leq m_n$, set
\bean
R_i :=
R_i(s,t,u) :=
 \xi_t(f_n {\bf 1}_{Q_{n,i}} ) 
- \xi_s(f_n {\bf 1}_{Q_{n,i}});
\\
Y_i := 
Y_i(s,t,u) : =
 \xi_u(g_n {\bf 1}_{Q_{n,i}} ) 
- \xi_t(g_n {\bf 1}_{Q_{n,i}}),
\eean
and observe that $\E[R_i] = \E[Y_i] = 0$ by time-stationarity.
Let us introduce an  adjacency relationship $\sim$ on $\{1,2,\ldots,m_n\}$
whereby $i \sim j$ if and only if $\|q_{n,i} - q_{n,j}\| \leq 3(r_0+ d)$.
%
% the distance between $Q'_{n,i}$ and $Q'_{n,j}$ is at most $3(r_0+d)$.
Then
the degrees of the resulting graph
$(\{1,2,\ldots,m_n\},\sim)$ are bounded by a constant, independent of
$n$.  Also, 
% Moreover, if $i,j \in \{1,2,\ldots, m_n\}$ are not adjacent,
%the random vector $(R_i,Y_i) $ is independent
%of $(R_j,Y_j)$.
 this graph is a dependency graph for the 
random 2-vectors $(R_1,Y_1),\ldots,(R_{m_n},Y_{m_n})$.
%variables $Y_1(t,u),\ldots, Y_{m_n}(t,u)$.
%Therefore
Thus, for $i,j,k,\ell \in \{1,2,\ldots,m_n\}$, 
we have  $\E[R_iR_j Y_k Y_\ell] =0$ unless
the subgraph of $(\{1,\ldots,m_n\}, \sim)$
 induced by $\{i,j,k,\ell\}$
is either connected or has two connected components of order 2;
in the latter case the expectation factorises as the product of two
expectations of pairwise products.
% In fact, all products of the form $A_i A_j A_k A_\ell$, 
%taking $A_\alpha$ to be either $Y_\alpha$ or $R_\alpha$
%for $\alpha =i,j,k,\ell$, have mean zero unless
%the subgraph of $(\{1,\ldots,m_n\}, \sim)$
% induced by $\{i,j,k,\ell\}$
%is either connected or has two components of order 2. In fact,
%
Since  $\xi_t(f_n) -\xi_s(f_n) = \sum_{i=1}^{m_n} R_i$ and
  $\xi_u(g_n) -\xi_t(g_n) = \sum_{i=1}^{m_n} Y_i$, we thus have
\bea
%n^{-2} 
\E[
(\xi_t(f_n) - \xi_s(f_n))^2
(\xi_u(g_n) - \xi_t(g_n))^2
] = 
%n^{-2}
\E \sum_{i=1}^{m_n} \sum_{j=1}^{m_n} \sum_{k=1}^{m_n}
\sum_{\ell=1}^{m_n} R_i R_j Y_k Y_\ell  
\nonumber
\\
\leq C 
%n^{-2}
 m_n \sup_{i,j,k,\ell \in \{1,\ldots,m_n\}} \E[R_i R_j Y_k Y_\ell]
\nonumber
\\
+ C
% n^{-2}
 m_n^2  \left( \sup_{i,j \in \{1,\ldots,m_n\}}
(|\E[R_i R_j] | +  |\E[Y_i Y_j]| + |\E[R_i Y_j]| ) \right)^2,
\label{0530h}
\eea
where, throughout this proof, $C$ denotes a positive constant independent
of $s,$ $t,$ $u,$ $n, $ $i,$ $j,$ $k, $ and $\ell$ which may change from line to line (or even
within a line).  

Given $i $ and $j$, let
 $N_{ij}$ (respectively, $N'_{ij}$) be the number of arrivals 
of $\Po$ between times
$s$ and $t$ (respectively, between times $t$ and $u$) 
 within Euclidean distance $r_0$ of 
$Q_{n,i}  \cup Q_{n,j}$.
 Then $N_{ij}$ and $N'_{ij}$  are independent Poisson variables, each
 with parameter  bounded by $c(u-s)$, where we may take the
constant $c$ to be $2 (2r_0+1)^d$. 
Since we assume $a \leq s < u \leq b$, 
 by the law of the unconscious statistician,
%there is a constant $c$ such that
 for any constant $\beta >0$, and any $i,j$, 
 we have
\bea
\E[N_{ij}^\beta ] \leq \sum_{k=1}^\infty \frac{k^\beta (c(u-s))^k}{k!}
\leq (u-s) \sum_{k=1}^\infty \frac{k^\beta (b-a)^{k-1} c^k}{k!}.   
\label{0814a}
\eea
The last sum
 converges so
 $\E[N_{ij}^\beta ]= O(u-s)$,
uniformly in $n$, $i,$ $j$,  $s$, $t$ and $u$.
 Similarly
 $\E[(N'_{ij})^\beta ]= O(u-s)$,
uniformly in $n$, $i,$ $j$,  $s$, $t$ and $u$.

Given $i,j \in \{1,2,\ldots,m_n\}$, let 
\bean
{\cal I} :=  \{k \in \N: (x_k,t_k) \in ( 
(B(r_0 +d) + q_{n,i}) \cup (B(r_0 +d) + q_{n,j}) )
 \times (s,\infty)\};
\\
{\cal I'} :=  \{k \in \N: (x_k,t_k) \in ( 
(B(r_0 +d) + q_{n,i}) \cup (B(r_0 +d) + q_{n,j}) )
 \times (t,\infty)\}.
\eean
Enumerate ${\cal I} = \{k(1),k(2),k(3),\dots\}$
with $s < t_{k(1)} < t_{k(2) } < t_{k(3)} < \cdots$, and
enumerate ${\cal I'} = \{k'(1),k'(2),k'(3),\dots\}$
with $t < t_{k'(1)} < t_{k'(2) } < t_{k'(3)} < \cdots$.
For $h \in \N$ set $Z_{i,j,h} := |M_{k(h)}|$
 and $Z'_{i,j,h} := |M_{k'(h)}|$.

%Let $Z_1,Z_2,\ldots$ be  independent random variables
%with the distribution of $|M|$, \textcolor{\blue}{independent} of $N_{i}$.
%Suppose $1 \leq \beta \leq 4$.
By Jensen's inequality, followed by  (\ref{0814a}),
and our assumption that $\E[|M|^{4+\eps}] < \infty$,
for $ 1\leq \beta \leq 4$ we have
\bea
\E \left[ \left(\sum_{h=1}^{N_{ij}} Z_{i,j,h} \right)^\beta \right]
\leq \E \left[ \E \left[  {N_{ij}}^\beta \left( \frac{1}{{N_{ij}}} 
\sum_{h=1}^{N_{ij}} Z_{i,j,h}^\beta\right) 
| {N_{ij}} \right] \right]
= \E\left[ N_{ij}^\beta \right] \E[|M|^\beta] 
\nonumber \\
\leq C(u-s).
~~~~~~~~~~~~~~~~~
\label{0814d}
\eea

Given any random variable $X$, let $X^+:= \max(X,0)$,
 and $X^- := \max(-X,0)$ be its
positive and negative parts. 
Since we assume $0\leq f \leq 1$ and
$0 \leq g \leq 1$ pointwise we have 
the following estimates,
which we shall use repeatedly:
\bea
%\max( R_i^+,R_j^+)
R_i^+ \leq \sum_{h=1}^{N_{ij}} Z_{i,j,h}, ~~~~~~~ 
%\max( Y_i^+,Y_j^+) 
Y_i^+ \leq \sum_{h'=1}^{N'_{ij}} Z'_{i,j,h'},
\label{0815b}
\eea
%and
\bea
 R_i^- \leq \xi_s(Q_{n,i}) {\bf 1}\{N_{ij} \geq 1\}, ~~~~~~~ 
{\rm and} ~~~~~~~
Y_i^- \leq \xi_{t}(Q_{n,i}) {\bf 1}\{ N'_{ij} \geq 1\}.
\label{0815c}
\eea

Let $i,j \in \{1,\ldots,m_n\}$.
By (\ref{0815b}) and (\ref{0815c}),
%Let $(\xi_t- \xi_s)^+$, and $(\xi_t- \xi_s)^-$ denote
%the positive and negative parts, respectively, of  the signed measure
%$\xi_t - \xi_s$. Then $(\xi_t -\xi_s)^+ (Q_i) \leq 
%\sum_{h=1}^{N_{ij}} Z_{i,j,h} $
%and $(\xi_t -\xi_s)^- (Q_i)  \leq \xi_s(Q_{n,i}) {\bf 1}\{N_{ij} >0\}$.
%Therefore 
\bean
|R_i | \leq  
\left( \sum_{h=1}^{N_{ij}} Z_{i,j,h} \right) + \xi_s(Q_{n,i})
 {\bf 1}\{N_{ij} \geq 1\}
\leq
2 \max \left(  \sum_{h=1}^{N_{ij}} Z_{i,j,h} ,
 \xi_s(Q_{n,i}) {\bf 1}\{N_{ij} \geq 1 \} \right).
\eean  
Hence,
\bean
\E[|R_i|^2 ] \leq 2 \E \left[ \left( \sum_{h=1}^{N_{ij}} Z_{i,j,h} \right)^2
+ \xi_s(Q_{n,i})^2{\bf 1}\{N_{ij} \geq 1\} \right],
\eean
Thus using  (\ref{0814d}), and Lemma \ref{lem0207a2},
 and the fact that
 $N_{ij}$
is independent of $\xi_s(Q_{n,i})$, we have that
$
\E[R_i^2] \leq C (u-s).
$
By the same argument we may also deduce that $\E[Y_i^2] \leq C(u-s)$, 
and then by the Cauchy-Schwarz inequality we obtain that
\bea
\sup_{m \in \N, i,j \in\{1,\ldots,m_n\}}
 (|\E[R_iR_j]| + |\E[Y_iY_j]| + |\E[R_i Y_j]| )
\leq C (u-s).
\label{0815a}
\eea
%\\E[R_i R_j] $, $\E[Y_i,Y_j]$ and $\E[R_i Y_j]$ are also
%bounded by $C(u-s)$ uniformly over all $i,j$. Therefore using the fact
%that $m_n= O(n)$, we obtain that the last line of (\ref{0530h}) is bounded
%by $C(u-s)^2$.
%\bea
%\E[ R_i^+ Y_j^+ ] \leq \E\left[ \E \left[ \left(
%\sum_{h=1}^{N_{ij}} Z_h \right)^2 | N_{ij} \right] \right]
%\nonumber
%\\
%\leq \E[ N_{ij}^2] \E[Z_1]^2 + \E[N_{ij}] \E[Z_1^2] 
%\leq C (u-t),
%\label{0530a}
%\eea
%\textcolor{\blue}{
%by (\ref{0814a}) and the assumption that
% $\E[|M|^4]< \infty$.} Also 
%$$
%Y_i^- Y_j^- \leq \xi_t(Q_{n,i}) \xi_t(Q_{n,j}) {\bf 1}\{N_{ij} \geq 1\},
%$$
%and since the variable $N_{ij}$ is independent of 
%$ \xi_t(Q_{n,i}) \xi_t(Q_{n,j})$, therefore
%\bea
%\E[Y_i^- Y_j^-] \leq \E[ \xi_t(Q_{n,i}) \xi_t(Q_{n,j}) ]
%\Pr[N_{ij} \geq 1 ] \leq C (u-t),
%\label{0530b}
%\eea
%where we have used Lemma \ref{lem0207a2} and the Cauchy-Schwarz inequality.
% Combining (\ref{0530a})
%and (\ref{0530b}) we have
%\bea
%\E[Y_i Y_j] \leq C (u-t) .
%\label{0530f}
%\eea

Given now $i,j,k, \ell \in \{1,\ldots,m_n\}$,
by (\ref{0815b}) 
%let $N_{ijk\ell}$ or simply $N$
% denote the number of Poisson arrivals between times
%$s$ and $t$ within distance $r_0$ of $Q_{n,i} \cup Q_{n,j}
%\cup Q_{n,k} \cup Q_{n,\ell}$. 
%Let $N'_{ijk\ell}$ or simply $N'$
% denote the number of Poisson arrivals between times
%$t$ and $u$ within distance $r_0$ of $Q_{n,i} \cup Q_{n,j}
%\cup Q_{n,k} \cup Q_{n,\ell}$. 
%Assume now that
%$Z_1,Z'_1,Z_2,Z'_2,\ldots $ are a sequence of
%independent random variables with the distribution
%of $|M|$, independent of $(N_{ijk\ell},N'_{ijk\ell})$. 
%Then  
we have
\bean
%\E [
R_i^+ R_j^+ Y_k^+ Y_{\ell}^+ 
%]
 \leq
%\E \left[
 \left( \sum_{h=1}^{N_{ij}}  Z_{i,j,h} \right)^2
 \left( \sum_{h'=1}^{N'_{k \ell}} Z'_{k,\ell,h'} \right)^2,
% \right]
\eean
and the two factors on the right are independent of each other. Hence
 by (\ref{0814d}) and a similar estimate for
$\sum_{h'=1}^{N'_{k\ell}} Z'_{k,\ell,h'}$,
\bea
\E [ R_i^+ R_j^+ Y_k^+ Y_{\ell}^+ ]
\leq
\E \left[ \left( \sum_{h=1}^{N_{ij}} Z_{i,j,h} \right)^2 \right] 
\E \left[ \left( \sum_{h'=1}^{N'_{k\ell}} Z'_{k,\ell,h'} \right)^2
 \right]
%.
%\eean
%we may deduce that
%\\
%\leq
%\E \left[ N^4  \left( N^{-1} \sum_{h=1}^N Z_h^4 \right) \right], 
%\eean
%By Jensen's inequality,
%\bean
%\E \left[ \left( \sum_{h=1}^N Z_h \right)^2 \right] 
%\leq \E \left[ N^2 \left( N^{-1} \sum_{h=1}^N Z_h^2 \right) \right]
%\eean
%and by conditioning on $N$ we then have
%\bean
%\E \left[ \left( \sum_{h=1}^N Z_h \right)^2 \right] 
%\leq \E[N^2] \E[|M|^2] \leq C( u-s)
%\eean
%and  similar estimates apply for 
%$\E \left[ \left( \sum_{h'=1}^{N'} Z'_{h'} \right)^2 \right]$,
%so combining these we have 
%%and by conditioning on $N$ we then have
%\bea
%\E [R_i^+ R_j^+ Y_k^+ Y_{\ell}^+ ] 
\leq C(u-s)^2.
\label{0530c}
\eea

Also by  (\ref{0815c}),
$$
R_i^- R_j^- Y_k^- Y_\ell^-  \leq 
\xi_s(Q_{n,i}) \xi_s(Q_{n,j}) \xi_t(Q_{n,k}) \xi_t(Q_{n,\ell})
{\bf 1}\{N_{ij} \geq 1 \} 
{\bf 1}\{N'_{k\ell} \geq 1 \} ,
$$
and since  $N'_{k\ell}$ is independent of
$\xi_s(Q_{n,i})
\xi_s(Q_{n,j})
\xi_t(Q_{n,k})
\xi_t(Q_{n,\ell}) {\bf 1}\{N_{ij} \geq 1\}$,
we obtain that
\bean
\E[ R_i^- R_j^- Y_k^- Y_\ell^-  ] 
\leq \E[
\xi_s(Q_{n,i})
 \xi_s(Q_{n,j}) \xi_t(Q_{n,k}) \xi_t(Q_{n,\ell})
{\bf 1}\{N_{ij} \geq 1\} ] \Pr[ N'_{k\ell} \geq 1]. 
\eean
Since we assume  for some $\eps >0$ 
that $\E[|M|^{4+\eps}] < \infty $, 
we have by Lemma \ref{lem0207a2}
that $\E[\xi_s(Q_{n,i})^{4+\eps} ] $ and
 $\E[\xi_t(Q_{n,j})^{4+\eps} ] $ are bounded by a constant,
independent of $n$, $i$, $j$, $s$ and $t$. 
Hence by H\"older's inequality, taking $p=1+(\eps/4)$,
we have
\bea
\E[ R_i^- R_j^- Y_k^- Y_\ell^-  ] 
\leq \E[ 
|\xi_s(Q_{n,i}) \xi_s(Q_{n,j}) \xi_t(Q_{n,k}) \xi_t(Q_{n,\ell})|^p ]^{1/p} 
\nonumber \\
\times
\Pr[ N_{ij} \geq 1]^{1- 1/p} \Pr[N'_{k\ell} \geq 1]
\nonumber \\
\leq C (u-s)^{2-1/p} 
\leq C' (u-s)^{1+ \eps/4}. 
%~~~~~~~~~~~~~~~~~~
\label{0530d}
\eea

Also by (\ref{0815b}) and (\ref{0815c}), 
\bean
%\E[
R_i^+ R_j^+ Y_k^- Y_{\ell}^-
%]
 \leq 
%\E \left[
 \left( \sum_{h=1}^{N_{ij}} Z_{i,j,h} \right)^2 
 \xi_t(Q_{n,k}) \xi_t(Q_{n,\ell})
{\bf 1}\{N'_{k\ell} \geq 1\} 
%\right]
,
\eean
so by independence of $N'_{k\ell}$, and the Cauchy-Schwarz inequality,
\bean
\E[R_i^+ R_j^+ Y_k^- Y_{\ell}^-] \leq 
\Pr[N'_{k\ell}\geq 1]
 \E \left[ \left( \sum_{h=1}^{N_{ij}} Z_h \right)^4\right]^{1/2}
\E[
 \xi_t(Q_{n,k})^2 \xi_t(Q_{n,\ell})^2]^{1/2}.
% \right] , 
\eean
Hence by (\ref{0814d}),  
 Lemma \ref{lem0207a2}, our $(4+ \eps)$-th moment assumption on $|M|$,
and the Cauchy-Schwarz inequality,
\bea
\E[R_i^+ R_j^+ Y_k^- Y_{\ell}^-] \leq 
C (u-s)^{3/2}. 
%C \E \left[ \E \left[ 
%\left( \sum_{h=1}^N Z_h \right)^2 |N \right] 
%\right]
%\nonumber \\
%\leq C \left( \E[N^2] (\E[Z_1])^2 + \E[N] \E[Z_1^2] \right)
%\leq C (u-t).
\label{0530e}
\eea

Next, observe that by (\ref{0815b}) and (\ref{0815c}), 
\bean
%\E[
R_i^- R_j^- Y_k^+ Y_{\ell}^+
%]
 \leq 
%\E\left[
 \xi_s(Q_{n,i}) \xi_s(Q_{n,j}) 
{\bf 1}\{N_{ij} \geq 1\}
\left(\sum_{h=1}^{N'_{k\ell}} Z'_{k,\ell,h} \right)^2, 
%\right]
\eean
and since the last  factor on the right
is independent of the other factors,
 using the Cauchy-Schwarz inequality and Lemma \ref{lem0207a2} we obtain
\bea
\E[R_i^- R_j^- Y_k^+ Y_{\ell}^+] \leq 
\E\left[
 \xi_s(Q_{n,i})^2 \xi_s(Q_{n,j})^2 \right]^{1/2} 
\Pr[N_{ij} \geq 1 ]^{1/2}
\E \left[
\left(\sum_{h=1}^{N'_{k\ell}} Z'_{k,\ell,h} \right)^2 
\right]
\nonumber \\
\leq C (u-s)^{3/2}.
~~~~~~~~~~~~~~
\label{0814b}
\eea

Next, note from (\ref{0815b}) and (\ref{0815c}) that
\bean
R_i^+ R_j^- Y_k^+ Y_{\ell}^- \leq 
\xi_s(Q_{n,j} ) \left(\sum_{h=1}^{N_{ij}} Z_{i,j,h} \right)
\xi_t(Q_{n,\ell})
\left( \sum_{h'=1}^{N'_{k\ell}} Z'_{k,\ell,h'} \right),  
\eean
and since the last factor on the right is independent of the other factors, 
using the Cauchy-Schwarz inequality, (\ref{0814d})
and Lemma \ref{lem0207a2} 
 again yields
\bea
\E[ R_i^+ R_j^- Y_k^+ Y_{\ell}^- ] \leq 
\E[ \xi_s(Q_{n,j} )^2
\xi_t(Q_{n,\ell})^2]^{1/2} 
\E\left[
 \left(\sum_{h=1}^{N_{ij}} Z_{i,j,h} \right)^2 \right]^{1/2}
\E\left[
 \sum_{h'=1}^{N'_{k\ell}} Z'_{k,\ell,h'} 
\right]
\nonumber \\
\leq C (u-s)^{3/2}.
~~~~~~~~~
\label{0814c}
\eea

Combining (\ref{0530c}), (\ref{0530d}), (\ref{0530e}),
(\ref{0814b}) and (\ref{0814c})  gives us 
\bea
\E[R_i R_j Y_k Y_\ell] \leq C (u-s)^{1+\eps'},
\label{0530g}
\eea
where we take $\eps' := \min(\eps/4,1/2) $.
Using (\ref{0530g}), (\ref{0815a}),
and the fact that $m_n = O(n)$, gives us from
(\ref{0530h}) that 
%(\ref{0605a}) holds.
\bea
\E[ n^{-2}
(\xi_t(f_n)- \xi_s(f_n))^2 
 (\xi_u(g_n) - \xi_t(g_n))^2 
]
\leq C  (u-s)^{1 + \eps'} ,  
\label{0605a2}
\eea
for $0 \leq f \leq 1$ and $0 \leq g \leq 1$ pointwise (the case $g=f$
gives (\ref{0605a}) for the special case with $0\leq f \leq  1 $ pointwise).

Now we drop the assumption that $f \geq 0$ but still assume $|f| \leq 1$ 
pointwise.  Write $\xi_{s,t}(f)$ for $\xi_t(f) - \xi_s(f)$.  
Using the fact that for any real $A,B$ we have $(A+B)^2 \leq (2 \max( |A| , |B|) )^2 \leq
4 (A^2 +B^2)$, we obtain that
\bean
 (\xi_{s,t}(f_n) )^2
 (\xi_{t,u}(f_n) )^2
= ( \xi_{s,t}(f_n^+ ) - \xi_{s,t}(f_n^- ))^2 
 ( \xi_{t,u}(f_n^+ ) - \xi_{t,u}(f_n^- ))^2 
\\
\leq 
4 (\xi_{s,t}(f_n^+)^2 + \xi_{s,t}(f_n^-)^2 )
\times
4 (\xi_{t,u}(f_n^+)^2 + \xi_{t,u}(f_n^-)^2 ) 
\\
= 16[ 
\xi_{s,t}(f_n^+)^2 \xi_{t,u}(f_n^+)^2 
+
\xi_{s,t}(f_n^+)^2 \xi_{t,u}(f_n^-)^2 
\\
+
\xi_{s,t}(f_n^-)^2 \xi_{t,u}(f_n^+)^2 
+
\xi_{s,t}(f_n^-)^2 \xi_{t,u}(f_n^-)^2 ].
\eean
By applying (\ref{0605a2}) four times, we see that (\ref{0605a}) holds
for arbitrary $f \in \cR_0$ with $|f| \leq 1 $ pointwise. Then
the general case follows by linearity.
%
%\bean
%\E[ (\xi_{t,u}(f_n) )^4 ] \leq \E  
%[ (\xi_{t,u}(f_n^+) - \xi_{t,u}(f_n^-))^4 ] \leq
%16 \E  [
% |\xi_{t,u}(f_n^+)|^4 +  | \xi_{t,u}(f_n^-)|^4 ], 
%\eean
%and applying the case already proved to $f_n^+$ and $f_n^-$
%gives us the result in general.
\end{proof}

\begin{proof}[Proof of Theorem \ref{th:CLTgen}(b)] 
Assume the conditions of Theorem \ref{th:CLTgen}(b)
hold, and let $f \in \cR_0$.
 Suppose $-\infty < a < b < \infty$.
By Theorem \ref{th:CLTgen}(a), the finite-dimensional 
distributions of the processes 
 $n^{-1/2} (\xi_t(T_n(f)) - \E[ \xi_t(T_n(f))])_{t \in [a,b]}$  
converge to those of 
a centred Gaussian
process $(X_t)_{t \in [a,b]}$  with covariance function
 $\|f\|_2^2\kappa_0((f,t),(f,u))$ given by
(\ref{kappa0def})
% with continuous sample paths 
(a stationary Ornstein-Uhlenbeck process).

\textcolor{\blue}{
Given real numbers $a< b$, we
 can now apply \cite[Theorem 15.6]{Bill} to obtain convergence
in distribution in $D[a,b]$. Our
 Lemma \ref{lem4D} implies the condition \cite[eqn (15.21)]{Bill},
with $F(t)=t$.
The continuity of the limiting Ornstein-Uhlenbeck can be seen
for example from the Kolmogorov-\v{C}entsov theorem
\cite[Theorem 2.8]{KS}.
}
% we have
% (\ref{0605a}), which is similar to \cite[eqn (7.23)]{Surveys}.
%We can then follow the last part of the proof of
%\cite[Theorem 3.3]{Surveys}, to get the desired tightness
%and hence convergence in $D[a,b]$.

Once we have the convergence in $D[a,b]$ for all $a<b$
 we can obtain convergence
in $D(-\infty,\infty)$ using \cite[Theorem 2.8]{Whitt}.
\end{proof}

\begin{proof}[Proof of Proposition \ref{prop:bdy}]
Let $\Q'$ be such that
 $M(\cdot) := \H_{d-1}(\partial S \cap \cdot)$, and
 Condition  \ref{measassu} holds, and 
 $\H_{d-1}(\partial S) <\infty$ almost surely. 
%Write $X_i := S_i+ x_i$ for $i \in \N$.
We claim that  almost surely, 
for all $i,j \in \N$ with $i \neq j$, we have
$\H_{d-1}((S_i + x_i)  \cap ( S_j + x_j))=0$. Indeed,
%To see this, observe that
%
%Clearly, every point of $\Phi_t$ lies in the boundary of
%at least one set of the form $(S_i+x_i)$ with $(x_i,t_i) \in \Po$.
% Hence it suffices to show that for each $t$,
%  almost surely 
%the set of points of $\Phi_t$, that lie in the  boundary of two or more 
%of these sets, has $\H_{d-1}$-measure zero. 
%
by the Marking
theorem \textcolor{\blue}{(see \cite[Theorem 5.6]{LP})} the point process
% $\{(x_i,t_i,S_i,M_i)\}_{i \geq 1}$
 $\sum_{i =1}^\infty \delta_{(x_i,t_i,S_i,M_i)}$
is a Poisson process in $\R^d \times \R \times \cK \times \cM$,
with intensity $\H_d \otimes \H_1 \otimes \Q'$, and hence
by the bivariate Mecke formula,
\bean
\E \left[  \sum_{i \in \N } \sum_{j\in \N \setminus\{  i \}}
% over all pairs $(x_i,t_i,S_i,M_i)$ and
%$(x_j,t_j,S_j,M_j)$ of distinct 
%points of $\Po$, of  the $
\H_{d-1} ( (\partial S_i + x_i) \cap ( \partial S_j + x_j) ) \right]
\\
=
\int_\cK \Q(d \sigma) \int_{\cK} \Q(d \sigma') 
\int_{0}^\infty dt \int_{0}^\infty ds
\int_{\R^d} dx \int_{\R^d} dy 
\H_{d-1}((\partial \sigma + x) \cap (\partial \sigma'+y))
\\
= \int_\cK \Q(d\sigma) \int_{\cK} \Q(d\sigma') 
\int_{0}^\infty dt \int_{0}^\infty ds
\int_{\R^d} dx 
\int_{\partial \sigma} \H_{d-1}(dz)
\int_{\R^d} dy 
 {\bf 1}_{\partial \sigma' + y}(x+z),
\eean
which comes to zero because, almost surely, $\H_{d-1}(\partial S) < \infty$
and $\H_{d}(\partial S)=0$. The claim follows.

By (\ref{Phidef}), and the preceding claim, almost surely
\bean
\H_{d-1}(\Phi \cap \cdot) = 
 \sum_{i: t_i \geq 0 } \H_{d-1}(\cdot \cap (\partial S_i + x_i) \setminus
\cup_{j: 0 \leq t_j  < t_i} (S_j^o + x_j) )
\\
= \sum_{i: t_i \geq 0 } \H_{d-1}( \cdot \cap (\partial S_i + x_i) \setminus
\cup_{j: 0 \leq t_j  < t_i} (S_j + x_j) ), 
\eean
which is equal to $\xi$ by (\ref{xidef}), since
\textcolor{\blue}{
 $(M_i+x_i) = \H_{d-1}(( \partial S_i + x_i) \cap \cdot)$} 
 by our choice of $\Q'$.
\end{proof}

\section{Proof of results for the DLM in $d=1$}
\label{secpfoned}
\allco

We start this section with a measurability result that we shall use 
more than once. 
\begin{lemm}
\label{rcspplem}
Let $d \in \N$.
Suppose $X$ is a random closed set in $\R^d$ that is almost surely finite.
Then $\H_0(X \cap \cdot)$ is a point process in $\R^d$.
\end{lemm}
\begin{proof}
In the notation of \cite[page 51]{SW}, the set $X$
is a random element of $\cF_{\ell f}$. For bounded Borel $A \subset \R^d$,
and $k \in \N \cup \{0\}$, let $ F_{A,k}$  be the set of
 locally finite sets $\sigma \subset \R^d$
such that  $\H_0(\sigma \cap A)= k$. Then using notation
from \cite[Lemma 3.1.4]{SW}, we have 
$$ 
 F_{A,k}= i(\{\mu \in {\mathrm N}_s : \mu(A) =k\}) \in i_s({\cal N}_s),
$$
and therefore by
 \cite[Lemma 3.1.4]{SW}, $F_{A,k} \in {\cal B}({\cal F})_{\ell f}$
so that the event $\{\H_0(X \cap A) =k\} = \{X \in F_{A,k}\}$
is measurable (that is, it is indeed an event).
Hence $\H_0(X \cap \cdot)$ 
is a point process.
\end{proof}

Throughout  the rest of 
this section we take $d=1$. We prove the results stated in
Section \ref{seconedim}.

\begin{lemm}
\label{1dcondlem}
Suppose $\Q$ is such that $\partial S$ is almost surely finite.
Then Condition \ref{measassu} holds, that is, $\H_0(\partial S \cap \cdot)$
is a point process in $\R$.
\end{lemm}
\begin{proof}
We are now assuming $d=1$.
By \cite[Theorem 2.1.1]{SW},  the 
set $\partial S$ is a random closed set that
is almost surely finite by assumption. Therefore
$\H_0(\partial S \cap \cdot)$ is a point process in $\R$
by Lemma \ref{rcspplem}.
\end{proof}

\begin{proof}[Proof of Proposition \ref{th:main}]
By definition $\eta = \H_0(\Phi \cap \cdot)$, where $\Phi$ is
the set of boundary points of our time-reversed DLM tessellation.
We aim to apply Theorem \ref{th:momgen}. 
We are given the measure $\Q$, and let $\Q'$ be the
measure on $\cK \times \bM$ whereby a random
pair $(S,M)$  under $\Q'$ is such that $S$ has
distribution $\Q$ and $M = \H_0 ((\partial S) \cap \cdot)$.
Note that $M$ is a random element of $\bM$ by  Lemma \ref{1dcondlem}.
Then by Proposition \ref{prop:bdy},
 $\eta $ is the dead leaves random measure $ \xi$,  defined
 at  (\ref{xitdef}).
Hence by Theorem \ref{th:momgen}, $\eta$ is a point process and its
 intensity  is equal
to $\E[\H_0(\partial S)]/\lambda$.
\end{proof}

\begin{proof}[Proof of Theorem \ref{th:fmm}]
Assume without loss of generality that $\Q$ is concentrated on
intervals of the form \textcolor{\blue}{$[0,x]$ with $x \geq 0$}.
 Let $A,B \in \cB^1$ 
with $0 < \H_1(A) < \infty$
and $0 < \H_1(B) < \infty$, and with $x <y$ for all $x \in A,y \in B$.
  The product $\eta(A) \eta(B)$
equals the number of pairs of exposed endpoints of intervals
(i.e., leaves) in the time-reversed DLM,
one arriving in $A$ and the other in $B$. We can split this
into several contributions according to whether the endpoints
in question are left or right endpoints, whether they belong
to the same or different intervals, and (in the latter case)
which of the two endpoints arrives first.

Consider first the contribution from pairs consisting of  an exposed
right endpoint arriving
in $A$ before an exposed left endpoint arriving in $B$. Let $N_1$
denote the number of such pairs.
By the multivariate Mecke formula, 
%the expected number of such pairs, denoted $b_1$ is
%equal to
\bean
\E[N_1] = \int dx \int \nu(du) \int dy \int \nu(dv) \int_{0}^\infty ds
\int_s^\infty dt {\bf 1}_A(x+u) {\bf 1}_B(y) \Pr[E_{x+u,s} \cap E_{y,t}],
\eean
where $E_{x,t}$ is defined in Lemma \ref{Elem}, and the range of integration
when unspecified is $(-\infty,\infty)$.
By Lemma \ref{Elem}, for $0< s <t$  and $x,y \in \R$ 
we have
$\Pr[E_{x,s} \cap E_{y,t}]= \exp(-\lambda_{y-x}s - \lambda (t-s))$. 
Hence, using the change of variables $z= x+u$, we have
\bea
\E[N_1] =  \int \nu(du) \int dy \int \nu(dv) \int dz 
\int_{0}^\infty ds
\int_s^\infty dt {\bf 1}_A(z) {\bf 1}_B(y) e^{-s \lambda_{y-z}}
e^{-(t-s) \lambda}
\nonumber \\
= \int_A dz \int_B dy \lambda_{y-z}^{-1} \lambda^{-1}.
~~~~~~~~~~~~
\label{0705a}
\eea
We get the same contribution as $\E[N_1]$ for pairs consisting
of a right endpoint in $A$ arriving before a right endpoint
in $B$, 
and also from a left endpoint in $B$ arriving before a left  endpoint in $A$,
and also from a left endpoint in $B$ arriving before a right  endpoint in $A$.

Let $N_2$ denote the number of pairs that consist of  an 
exposed left endpoint arriving in $A$ before an 
exposed left endpoint arriving in $B$. 
In this case the first of these arrivals has to avoid covering the
second endpoint, for the pair to contribute.
 By the multivariate Mecke formula and Lemma \ref{Elem}, 
\bea
\E[N_2]= \int_A dx \int_B dy \int \nu(du) {\bf 1} \{x+ u < y\}
\int_0^\infty ds \int_s^\infty dt
\Pr[E_{x,s} \cap E_{y,t}]
\nonumber \\
= \int_A dx \int_B dy \nu([0,y-x)) \lambda_{y-x}^{-1} \lambda^{-1} 
\nonumber \\
= \int_A dx \int_B dy F(y-x) \lambda_{y-x}^{-1} \lambda^{-1} 
\label{0705b}
\eea
where we have used the fact that $\nu(\{z\}) = 0$ for all but
countably many $z \in \R$ so  $\int_B \nu(\{y-x\}) dy = 0$.
We get the same contribution as $\E[N_2]$ from
pairs consisting of a left endpoint arriving in $A $ before
 a right endpoint in $B$, and from pairs consisting of
a right endpoint in $B$ arriving before a left endpoint arriving in $A$, and
a right endpoint in $B$ arriving before a right endpoint arriving in $A$.

Let $N_3$ be the number of pairs consisting of an exposed
 left endpoint in $A$ and an exposed right endpoint in $B$, both being
endpoints of the same leaf.
Then
\bean
\E[N_3] = \int_A dx \int \nu(dy)  {\bf 1}_B (x+y) \int_0^\infty dt
\Pr[ E_{x,t} \cap  E_{x+y,t} ],
\eean
and using the change of variable $z=y+x$ along with  Lemma
\ref{Elem}, we obtain that
\bea
\E[N_3] = 
\int_A dx  \int_B \textcolor{\blue}{
\Pr[ H +x \in dz] } \int_0^\infty \exp(-\lambda_{z-x} t) dt   
\nonumber \\
= \int_A dx \int_B \lambda_{z-x}^{-1} \Pr[H+x \in dz].
\label{0705c}
\eea
Then $ \alpha_2(A \times B) = \E[\eta(A) \eta(B)] = 
4 \E[N_1] + 4 \E[N_2] + \E[N_3]$, 
and by (\ref{0705a}), (\ref{0705b}) and (\ref{0705c})
we obtain (\ref{0615a}).
%
 %By the multivariate Mecke formula, 
%
%
%(see for example \cite{LP}),
 %for $x <y$ we have
%\bean
%\alpha_2(d(x,y)) =  4 (1 + F(y-x)) \left( 
%\int_0^\infty \int_s^\infty 
%\exp(- s \lambda_{y-x})  \exp( - (t-s) \lambda)
%dt ds \right) dy dx
%\\
%+ \left(  \int_0^\infty e^{- s \lambda_{y-x} } ds \right) 
%\Pr[H + x \in dy] dx. 
%\eean
%In the first term of the above right hand side, the factor 4
%is the product of a factor of 2 from the fact that a similar
%integral over $t<s$ instead of $s<t$ makes the same contribution,
%along with a factor  of 2 from the fact that the 
%interval  endpoint at $y$ arriving at time $t$ could
%be either a left or right endpoint.
%The factor  of $1+ F(y - x)$ comes
% from the fact that the interval endpoint arriving at $(x,s)$  could
%be either a left or a right endpoint, but if it is a left 
%endpoint, this interval needs to avoid covering the point $y$. 
%The last term in the above right hand side accounts for
%the possibility of seeing a single
% leaf with
%both its left end (at $x$) and its right end (at $y$) visible.
%
% Evaluating the integrals gives us for $x < y$ that (\ref{0615a})
%holds. 
We then obtain (\ref{pcf}) from (\ref{0615a}), the
definition of pair correlation function, and Proposition \ref{th:main}.
Likewise, it is straightforward to deduce (\ref{delta2mom}) 
from (\ref{0615a}) when $\nu= \delta_\lambda$.
\end{proof}

\begin{proof}[Proof of Theorem \ref{th2mom1d}]
Suppose that $\E[(\H_0(\partial S))^2] < \infty$ and
$\E[(\H_1(S))^2] < \infty$.
The existence of the limit
 $\sigma_1^2 := \lim_{n \to \infty} (n^{-1}\Var [\eta([0,n])])$ 
follows from Theorem \ref{thcovgen}, taking $\Q'$ to be 
 as in the proof of Proposition \ref{th:main}.

Now suppose  also that $\Q$ is concentrated
on connected intervals and $F(0)=0$. Then
we can  derive  (\ref{0628a}) 
using either 
(\ref{0615a}) or
the formula for $\sigma_0^2$ given in
 Theorem \ref{thcovgen}. We take the first of these options, 
and leave it to the reader to check that the latter option gives the
same value for $\sigma_1^2$.
Since $\eta$ is a simple point process, we have
by (\ref{0615a}) and Proposition \ref{th:main} that
\bean
\E [(\eta([0,n]))^2] = \E[\eta([0,n])] + 2 \int_{0\leq x < y \leq n}
 \alpha_2 (d(x,y))
\\
= (2n/\lambda)  + 2 \int_0^n \int_{(x,n]} \lambda_{y-x}^{-1}
\Pr[ H+x \in dy ] dx + 8 \int_0^n \int_x^n \frac{1+ F(y-x)}{\la \la_{y-x} }
 dy dx,
\eean
while $\E[\eta([0,n])]^2 = 8 \int_0^n \int_x^n \la^{-2} dy dx$.
Taking $v= x/n$ and $u= y-x$, we thus have
\bean
n^{-1} \Var [\eta([0,n])] = \frac{2}{\lambda}  
+ 2 \int_0^1 \int_{(0,n-nv]} \lambda_u^{-1} \Pr[H \in du] dv
\\
+ 8 \int_0^1   \int_0^{n(1-v)} \left( \frac{1+F(u) }{\lambda \lambda_u }
- \frac{1}{\lambda^2} \right) du dv,
\eean
so that  (\ref{0628a}) holds by dominated convergence, provided
the right hand side of (\ref{0628a}) is finite.

\textcolor{\blue}{
Recall from just before Theorem  \ref{th:fmm}, 
and from (\ref{lambdadef}), that
  %$\lambda_u =\lambda+ \E[u \wedge H]$, 
  $\lambda_u =\lambda+ \E [\min(u,H)]$, 
for all $u \geq 0$.
Since 
%$\E[u \wedge H] 
$\E[\min(u, H)] 
= \int_0^u \oF(t)dt = \la - \int_u^\infty \oF(t)dt$,}
in the last integral of the right hand side of (\ref{0628a}) 
 the integrand can be re-written as
\bean
\frac{2 - \oF(u)}{\la(2\la - \int_u^\infty \oF(t)dt)} - \frac{1}{\la^2}
= \frac{ 2 \la - \la \oF(u) - (2 \la - \int_u^\infty \oF(t)dt ) }{\la^2
(2 \la - \int_u^\infty \oF(t)dt)} 
%\nonumber \\
%= \frac{\int_u^\infty \oF(t)dt - \la \oF(u)}{ 
%\la^2 (2 \la - \int_u^\infty \oF(t)dt)} 
\eean
which equals the expression in  (\ref{0628b}).

Since we assume $\E[H^2]< \infty$, we have
$\int_{0}^\infty \int_u^\infty \oF(t) dt du < \infty$, 
and therefore the expression in (\ref{0628b})
 is integrable and the right hand side of (\ref{0628a}) is indeed finite.

In the special case with $\nu= \delta_1$ (so that $\lambda=1$), the expression
in (\ref{0628b}) comes to $-u/(1+u)$ for
\textcolor{\blue}{ $u < 1$ (and zero for $u \geq 1$).} 
Therefore in this case the right hand side of (\ref{0628a}) comes to
\bean
2 +1 + 8 \left( \int_{0}^1 \left( \frac{1}{1+u} \right) du -1 \right) 
%\\
= 8 \log 2 -5,
%\approx 0.545,
\eean
as asserted.
\end{proof}

\begin{proof}[Proof of Theorem \ref{limcovthm}]
We use the last part of Theorem \ref{thcovgen}, 
 taking $\Q'$ to be as in the proof of Proposition \ref{th:main},
and taking $f = {\bf 1}_{[0,r]}$
and $g = {\bf 1}_{[0,s]}$.
\end{proof}

\begin{proof}[Proof of Theorems \ref{CLTb} and \ref{FCLT1}]
We use Theorem \ref{th:CLTgen},
 taking $\Q'$ to be as in the proof of Proposition \ref{th:main},
and applying it to functions of the form $f={\bf 1}_{[0,s]}$.
\end{proof}

\begin{proof}[Proof of Propositions \ref{thmXdist} and
  \ref{typintcoro}]
We use formulae from \cite{Matheron}, and also provide
some extra details compared to \cite{Matheron}.

\textcolor{\blue}{Given $h \geq 0$, 
let $K(h)$ and $P(h)$ be as defined in \cite[page 3]{Matheron};
that is, using our notation from Section \ref{seconedim},
let}
 $P(h) = \Pr[\eta([0,h]) =0]$ and 
$K(h) = \E[(H-h)^+]$. \textcolor{\blue}{Then }
 $ K(h) = \int_h^\infty (1-F(t))dt$, and
$K'(h)= -(1-F(h))$.
In particular $K'(0)= -1$ under our present assumptions.
Also $K(0) =   \lambda$. 
\textcolor{\blue}{
We assert 
%for all $h \geq 0$ 
that
\bea
P(h) = \frac{K(h)}{ \lambda + h}.
\label{0716a}
\eea 
Indeed, this is the last formula on \cite[page 3]{Matheron}, but
for completeness we sketch a proof here. In the time-reversed
DLM, the time $T$ to the first  arrival  (after time 0) of a leaf
 that intersects $[0,h]$ is
exponentially distributed  with mean $1/\mu_1$, where
$$
\mu_1 := \int_{-\infty}^\infty \Pr[[y,y+H] ] \cap [0,h] \neq \emptyset] dy
= h + \int_{-\infty}^0 \Pr[ H \geq -y] dy = h + \lambda.
$$ 
The time $T'$ to the first arrival (after time 0) 
of a leaf that covers the whole of $[0,h] $ is exponential with mean
 $1/\mu_2$, where 
$$
\mu_2 := \int _{-\infty}^0 \Pr[ [0,h] \subset [y,y+ H]] dy
= \int_{-\infty}^0 \Pr[H \geq -y +h] dy = K(h).
$$  
Moreover the time $T''$ to the first arrival of a leaf that intersects but
does not cover $[0,h]$ is also exponential with mean $1/(\mu_1-\mu_2)$,
and independent of $T'$. Then $P(h) = \Pr[ T' < T'']$, which
gives us  (\ref{0716a}) by a well-known result on the minimum
of independent exponentials.}

Let $X$ and $Y$ be as in the statement of 
Propositions \ref{thmXdist} and  
\ref{typintcoro}.  By stationarity, given $X =x$ the first point of
 $\eta$ to the right of $0$ is Unif$(0,x)$.
% distributed.
Hence
%Therefore  
%\bean
$
P(h) = \int_h^\infty 
%\frac{x-h}{x} 
((x-h)/x)
\Pr[X \in dx], 
$
%\eean
so by the discussion just before Proposition \ref{typintcoro},
\bean
P(h) = \int_h^\infty 
 ((x-h)/x)  
 (2x/\lambda)
\Pr[ Y \in dx]
%\\
= (2/\lambda) \int_h^\infty (1-F_Y(t)) dt,
\eean
where $F_Y := \Pr[Y \leq \cdot]$
 is the cumulative distribution function of $Y$,
and the  last equality comes from  Fubini's theorem.
Hence by 
%Therefore using
 (\ref{0716a}),
% we have
\bean
1- F_Y(y) = (-\lambda/2) P'(y) = \frac{\lambda}{2}
\left( \frac{1-F(y)}{\lambda+y}
+ \frac{K(y)}{(\lambda+y)^2} \right). 
\eean
This formula  appears on \cite[page 10]{Matheron}
(Matheron's $F_0$ is our $F$, and Matheron's $\nu$ is 
the intensity of $\eta$, which is $2/\lambda$ by Proposition \ref{th:main}).
By the product rule,
\bean
-dF_Y(y) = \frac{\lambda}{2} \left(  \frac{- dF(y)}{y+\lambda} -  
\frac{2(1-F(y))}{(y+\lambda)^2} dy - \frac{2 K(y)}{(\lambda+y)^3}
\right).
\eean
Since $K(y) = \E[(H-y)^+ ] = \int_y^\infty (u-y) dF(u)$,
 hence
\bean
dF_Y(y) = \frac{\lambda dF(y)}{2(y+\lambda)} + 
\frac{\lambda}{(y+\lambda)^3}
\left[
(\lambda +y)(1-F(y))  + \int_y^\infty (u-y) dF(u) 
\right] dy.
%\label{0716b}
\eean
The expression inside the square brackets in the above right hand side
%of (\ref{0716b}) 
is equal to $\int_y^\infty(\lambda +u) dF(u)$,
and so
% by (\ref{0716b}) 
we have Proposition \ref{typintcoro}.
The argument just before Proposition \ref{typintcoro}
shows that we can then deduce Proposition \ref{thmXdist}.
\end{proof}

\section{Proofs for the DLM in $d=2$}
\label{secpftwodim}
\allco
Throughout this section we take $d=2$. Also
 $S,S', S''$  denote independent random elements of
$\cK $ with common distribution $\Q$, and
$\Theta$ denotes a  random variable uniformly distributed 
over $(-\pi,\pi]$, independent of $(S,S')$.
% and $R$ is as in Section \ref{subsecnotation}.

\begin{proof}[Proof of Theorem \ref{meantheo}]
We obtain the result 
% (\ref{meanbdy}) 
by application of Theorem \ref{th:momgen}.
Here we are given $\Q$, and we  take $\Q'$ to be the
probability measure on $\cK \times \bM$ with first marginal $\Q$
such that if $(S,M)$ is $\Q'$-distributed then $M = \H_1(\partial S \cap
 \cdot)$.
\end{proof}
Our proof of \textcolor{\blue}{Theorem \ref{thbranch}   requires}
a series of lemmas.
The first is concerned with  random closed sets in $\R^2$ (or more generally,
in $\R^d$). 
%(probably already known).
\begin{lemm}
\label{RACSlem}
Any countable intersection of random closed sets in $\R^2$ is a random closed
set in $\R^2$.
\end{lemm}
\begin{proof}
Let $X_1,X_2,\ldots$ be random closed sets in $\R^2$. For $n \in \N$ set
$Y_n = \cap_{i=1}^n X_i$. Then $Y_n$ is a random closed set by 
\cite[Theorem 2.1.1]{SW}. Set $X= \cap_{n=1}^\infty X_n= \cap_{n=1}^\infty Y_n$.
 Then for any compact $K \subset \R^2$,
we have the event equalities
\bean
\{X \cap K = \emptyset\} = \{ K \subset \cup_{n=1}^\infty Y_n^c\}
= \cup_{n=1}^\infty  \{K \subset Y_n^c\} =
 \cup_{n=1}^\infty \{ Y_n \cap K = \emptyset\}, 
\eean
which is an event because each $Y_n$ is a random closed set
\textcolor{\blue}{
(see \cite[Definition 1.1.1]{Molchanov}).
}
 Therefore
$X$ is also a random closed set.
\end{proof}
\begin{lemm}
\label{intbdylem}
Suppose $\lambda < \infty$ and
 $\beta_3 <\infty$, where $\beta_3$ is given by (\ref{0628c}).
Then, almost surely:

 (a) for all distinct $i,j \in \N$ the set
$(\partial S_i + x_i) \cap (\partial S_j + x_j)$ is finite, and 

(b)  
%Moreover, almost surely,
 for all distinct $i,j,k \in \N$ the
set 
$(\partial S_i + x_i) \cap (\partial S_j + x_j)
\cap (\partial S_k + x_k)$ is empty.
\end{lemm}
\begin{proof}
Let $K >0$.
By the Mecke formula,
\bean
\E \sum_{i\in \N} \sum_{j\in \N \setminus \{i\}}
\H_0 ( (\partial S_i + x_i) \cap (\partial S_j + x_j))
{\bf 1}_{B(K) \times [-K,K]} ((x_i,t_i))
{\bf 1}_{B(K) \times [-K,K]} ((x_j,t_j))
\\
=  (2K)^2  \int_{B(K)} dx \int_{B(K)} dy 
\E [ \H_0( (\partial S +x) \cap (\partial S' +y)) ]
\\
\leq
4K^2 \int_{B(K)} dx \int_{\R^2}  dz \E [\H_0( (\partial S + x ) 
\cap (\partial S' +x + z)) ]
= 
 \textcolor{\blue}{
4K^2 (\pi K^2)
 \lambda^2 \beta_3},
%\\
%\leq
%(2K)^2 \int_{B(K)} dx \int_{\R^2}  dz \E [\H_0( \partial S  
%\cap (\partial S' + z)) ]
%= 4K^2 \beta_3
\eean
which is finite by assumption. Therefore, almost surely 
\bea
\H_0 ((\partial S_i + x_i )
\cap (\partial S_j + x_j )) < \infty
\label{0627a}
\eea
for all $(i,j)$ with $i \neq j$
with  $(x_i,t_i) \in B(K) \times [-K,K]$ 
and $(x_j,t_j) \in B(K) \times [-K,K]$. Therefore, letting $K \to \infty$
shows that  (\ref{0627a}) holds for all $(i,j)$ with $i \neq j$,
 almost surely, which gives
us (a).
%as required.

For (b), note that for $K >0$, by the multivariate Mecke formula,
writing $\sum_{i,j,k \in \N}^{\neq}$
for the sum over ordered triples $(i,j,k)$  of distinct
elements of $\N$, we have
\bean
\E \sum_{i,j,k \in \N}^{\neq} 
\H_0 ((\partial S_i + x_i )
\cap (\partial S_j + x_j )
\cap (\partial S_k + x_k )
) 
{\bf 1}_{[-K,K]}(t_i)
{\bf 1}_{[-K,K]}(t_j)
{\bf 1}_{[-K,K]}(t_k)
\\
= (2K)^3 \E \int dx \int dy \int dz
\H_0(( \partial S +x) \cap (\partial S' +y) \cap (\partial S'' + z)
),
\eean
where the range of integration is taken to be $\R^2$ whenever
it is not specified explicitly.
Taking $y'=y-x$ and $z'= z-x$, we find that
the last expression equals
\bean
 (2K)^3 \E \int dx \int dy' \int dz'
\H_0( \partial S  \cap (\partial S' +y') \cap (\partial S'' + z'))
\\
= (2K)^3 \E \int dx \int dy' \int dz' 
\int_{ \partial S  \cap (\partial S' +y') } \H_0(dw) 
{\bf 1}_{ \partial S'' + z'} (w).
\eean
In the last line we may interchange the innermost two integrals because
almost surely and for almost all $y'$ the innermost integral
$\int_{ \partial S  \cap (\partial S' +y') } \H_0(dw) $ is finite
because of the assumption that $\beta_3 < \infty$. 
Therefore the last expression is equal to
\bean
 (2K)^3 \E \int dx \int dy' 
\int_{ \partial S  \cap (\partial S' +y') } \H_0(dw) 
\int dz' 
{\bf 1}_{ \partial S'' } (w- z'), 
\eean
which is zero because, almost surely, $\partial S''$ is 
a rectifiable curve so that $\E [ \H_2(\partial S'') ] = 0$. Since
$K$ is arbitrary, this gives us part (b).
\end{proof}
\begin{lemm}
\label{notouchlem}
Assume either that $\Q$ has the piecewise $C^1$ Jordan property, or that 
$\Q$ has the rectifiable Jordan property  and is rotation invariant.
Then, almost surely,  there is no pair $\{i,j\}$ of distinct
elements of $\N$ such that $ \partial S_i + x_i$ touches 
$ \partial S_j + x_j$. 
\end{lemm}
\begin{proof}
Let $K \in (0,\infty)$. 
Let $N_K$  denote the number of ordered pairs $(i,j)$ of distinct
elements of $\N$ such that $ \partial S_i + x_i$ touches 
$ \partial S_j + x_j$, and $\{t_i,t_j\} \subset [-K,K]$. 

\textcolor{\blue}{Let us say,}
 for any two piecewise $C^1$ Jordan curves $\gamma$
and $\gamma'$, that $\gamma$ {\em grazes} $\gamma'$ if there
exists $z \in \gamma \cap \gamma'$ that is not
a corner of either $\gamma$ or $\gamma'$,
such that $\gamma$  grazes $\gamma$' at $z$.

Suppose $\Q$ has the piecewise $C^1$ Jordan
property. Then
% For $N_K$ to be non-zero,
% there would need to exist distinct $i,j \in \N$
%with $\{t_i,t_j\} \subset [-K,K]$,
%such that  $\partial S_i +x_i$ touches
%$\partial S_j + x_j$.
%%either the boundaries of the two 
%%sets $S_i+x_i$ and $S_j + x_j$ graze, or a corner 
%%of  $\partial S_i +x_i$  lies in $ \partial S_j + x_j$.
   by
the bivariate Mecke formula,
% the expected number of such pairs $(i,j)$ is given by 
\bean 
\E[N_K] = (2K)^2
\int_{\cK} \Q( d \sigma) \int_{\cK} \Q(d \sigma')
\int_{\R^2} dx \int_{\R^2} dy {\bf 1} \{ \partial \sigma+x \mbox{ touches } 
\partial \sigma' + y \}
%\\
%+ 
%\int_{\cK} \Q( d \sigma) \int_{\cK} \Q(d \sigma')
%\int_{\R^2} dx \int_{\R^2} dy {\bf 1} \{\partial \sigma' +y \mbox{ has a corner
  %in }  \partial
%\sigma +x \},
\eean
which 
%and the first of these integrals 
is zero by  Lemma \ref{touchlemC1}.
% while the second is zero because $ \H_2(\partial S) =0$ almost surely.

Suppose instead that $\Q$ has the rectifiable Jordan property
and is rotation invariant. Then
\bean
\E[N_K] = (2K)^2 \int dx \int dy \Pr[ (\partial S +x) \mbox{ touches }
(\partial S' +y) ]
\\
= (2K)^2 \int dx \int dy \Pr[ (\partial S + x) \mbox{ touches }
\rho_{\Theta} ( \partial S' +y) ]
\\ 
= (2K)^2 \int dx \int_{\cK} \Q(d\sigma) \int_{\cK} \Q(d \sigma')  
 \int dy \int_{-\pi}^{\pi} d \theta
{\bf 1} \{ (\partial \sigma +x) \mbox{ touches }
\rho_\theta ( \partial \sigma' + y ) \}
\eean
which equals zero by Lemma \ref{touchlemrect}. Thus in  both cases,
$N_K =0$ almost surely, for all $K$, and the result follows.
\end{proof}

Before proving
 Theorem \ref{thbranch}, we introduce some further notation.
%As in the proof of Lemma \ref{lemmeas},
%write $ \Po   = \sum_{i=1}^\infty \delta_{(x_i,t_i)}$ with 
%$(x_i,t_i)_{i \in \N}$ a sequence of random elements of 
%$\R^2 \times \R$.
%  Then $S_i +x_i$ is a random element 
%of $\cK$ for each $i \in \N$, by \cite[Theorem 2.4.3]{SW}, for example.
%Hence, for any distinct $i,j \in \N$, 
%also $(\partial S_i + x_i) \cap (\partial S_j + x_j)$
% is a random element
%of $\cK$ by \cite[Theorem 2.1.1]{SW}, and it is almost surely finite
%by Lemma \ref{intbdylem}.
%
  For any random closed set $X$ in $\R^2$
and event $A$, let $X_A$ be the random closed set that is $X$ if
$A$ occurs and is $\R^2$ if not. Let $X^A$ 
 be the random closed set that is $X$ if
$A$ occurs and is $\emptyset$ if not. 

Given $i \in \N$, write $X_i$ for the set $S_i+x_i$, and
$X_i^o$ for the interior of $X_i$.
  Then $X_i $ is a random element 
of $\cK$, by \cite[Theorem 2.4.3]{SW}, for example.
Hence, 
given also $j \in \N \setminus \{i\}$, the set
 $\partial X_i  \cap \partial X_j $
 is also a random element
of $\cK$ by \cite[Theorem 2.1.1]{SW}, and it is almost surely finite
by Lemma \ref{intbdylem}. Set
$$
Y_{ij} := (\partial X_i \cap \partial X_j)^{\{\min(t_i,t_j) >0\}} \cap 
  \cap_{k \in \N \setminus \{i,j\}} (X_k^o)^c_{\{0 < t_k < \max(t_i , t_j)\}}. 
$$
Recall from Section \ref{ctshidim} that we
 define $\Xi$ to be the set of points in $\R^2$ which lie in
three or more cells of the time-reversed DLM tessellation,
and $\chi$ to be the measure $\H_0(\Xi \cap \cdot)$.
%Later we shall view $\Phi$ as a planar graph with   the points
%of $\Xi$ as  the  nodes of degree 3, which we call {\em branch points},
%in this graph.
%We define the measure $\chi := \H_0(\Xi \cap \cdot) $.
\begin{lemm}
\label{Ylem}
Assume that $\Q$ either has the piecewise $C^1$ Jordan property, 
or is rotation invariant and has the rectifiable Jordan property.  
Assume also that $\beta_3 <\infty$,
where $\beta_3$ is given by (\ref{0628c}), and that 
%$\E[R^2]< \infty$.
$\E[\H_2(S \oplus B(1)) ] <\infty$.
Then,
almost surely, $\Xi = \cup_{i=1}^\infty (\cup_{j=i+1}^\infty Y_{ij})$
and $\chi = \sum_{i=1}^\infty (\sum_{j=i+1}^\infty \H_0( Y_{ij} \cap \cdot))$.
\end{lemm}
\begin{proof}
Assume the times $t_1,t_2,\ldots$ are distinct (this occurs a.s.).
 Given $x \in \Xi$, $x$ must lie on
the boundary of
the first two shapes $X_i$ to arrive after time zero and contain  $x$, and this gives us the inclusion
 $\Xi \subset \cup_{i=1}^\infty (\cup_{j=i+1}^\infty Y_{ij})$. 

For the reverse inclusion, let $E$ be the event that
there is no triple $(i,j,k)$ of distinct elements of $\N$ with
\textcolor{\blue}{
$\partial X_i \cap \partial X_j \cap \partial X_k \neq \emptyset$,
}
and let $E'$ be the event that there is no pair $(i,j)$ of
distinct elements of $\N$ such that $\partial X_i$ touches $\partial
X_j$. Then $E$ and $E'$ occur almost surely, by Lemmas
\ref{intbdylem} and \ref{notouchlem}. 
 Let $E''$ be the event that for all $K \geq 0$ the number of shapes
$X_j$ with $X_j \cap B_K \neq \emptyset$ and $-K \leq t_j \leq K$  
is finite. This event also occurs almost surely,
by Lemma \ref{Boolem}.
%an argument using  the assumption $\E[R^2] < \infty$,

Also the times $t_i$ are all distinct,
almost surely.
 %with probability 1 by  
 %note that almost surely there is
Assume from now on that events $E$, $E'$ and $E''$ all occur and all of the
times $t_i$ are distinct.

Suppose $x \in Y_{ij}$ for some
$i,j$ with $ 0 < t_i <t_j$.
Let $T_k:= \inf \{t_{\ell}: t_{\ell} >0, x \in X_{\ell}^o\}$.

Then since we assume $E$ occurs, $x \notin \partial X_{\ell}$ for all
$\ell \in \N \setminus \{i,j\}$. Since we assume $E''$ occurs, 
almost surely only finitely many of the shapes $X_\ell$ with 
$0 \leq t_\ell \leq t_k$
have non-empty intersection with $B(1)+x$. Hence,
 $x \in \partial X_i \cap \partial X_j \cap X_k^o$, and there exists
a constant $\eps >0$ such that 
$$
 B(\eps ) + x \subset \cap_{\ell: 0 \leq t_\ell \leq t_k, \ell \notin \{i,j,k\} }
X_\ell^c.  
$$
Now, $x \in X_i$ and since $X_i$ is a regular set, $x$ is
an accumulation point of the interior of $X_i$, which is connected
by the Jordan curve theorem. Thus $x$ is on the boundary of
a component of $\Xi^c$ which is contained in the interior of $X_i$. 

Since we assume that $E'$ occurs, $\partial X_j$ crosses $\partial X_i$ at $x$
rather than touching it. Hence there is an arc within $\partial X_j$,
 with an endpoint at $x$, that lies outside $X_i$ except for this endpoint.
On one side of this arc  is a part of the interior of $X_j$, and hence there
is a component of $X_j^o \setminus X_i$ with an accumulation point at $x$,
and hence a component of $\Phi^c$ that is contained in $X_j^o \setminus X_i$
with an accumulation point at $x$.

Moreover, on the other side of the arc just mentioned is a region of 
$X_j^c \cap X_i^c$ with an accumulation point at $x$. Hence there is a
component of $\Phi^c$ that is contained in $X_k^o \setminus (X_i \cup X_j)$
and has an accumulation point at $x$. Therefore $x \in \Xi$, so that
$\cup_{i=1}^\infty (\cup_{j=i+1}^\infty Y_{ij} ) \subset  \Xi$.
Therefore $\Xi = \cup_{i=1}^\infty ( \cup_{j=i+1}^\infty Y_{ij})$,
as claimed.
Then since $E$ is assumed to occur, we have $Y_{i'j'} \neq Y_{ij}$
for all $(i',j') \neq (i,j)$, so that
$\sum_{i=1}^\infty \sum_{j=i+1}^\infty \H_0(Y_{ij} \cap \cdot) =  \chi$.
\end{proof}

\begin{proof}[Proof of Theorem \ref{thbranch} (a)]
For each $k \in \N \setminus \{i,j\}$, the set $(X_k^o)^c$ is a random
closed set by \cite[Theorem 2.1.1]{SW}. Therefore 
  $(X_k^o)^c_{\{0 < t_k < \max(t_i , t_j)\}}$ is also a random closed set,
and hence by Lemma \ref{RACSlem} the set $Y_{ij}$ is also a random closed
set.  By Lemma \ref{intbdylem}, $Y_{ij}$ is almost surely finite.
By Lemma \ref{rcspplem},
$\H_0( Y_{ij} \cap \cdot) $ is a point process in $\R^2$. 
%
%In the notation of \cite[page 51]{SW}, the set $Y_{ij}$
%is a random element of $\cF_{\ell f}$. For Borel $A \subset \R^2$,
%and $k \in \N \cup \{0\}$, let
%$
%F_{A,k}$  be the set of locally finite sets $\sigma \subset \R^2$
%such that  $\H_0(\sigma \cap A)= k$. Then using notation
%from \cite[Lemma 3.1.4]{SW}, we have 
%$$ 
 %F_{A,k}= i(\{\eta \in {\mathrm N}_s : \eta(A) =k\}) \in i_s({\cal N}_s)
%$$
%and therefore by
 %\cite[Lemma 3.1.4]{SW}, $F_{A,k} \in {\cal B}({\cal F})_{\ell f}$
%so that the event $\{\H_0(Y_{ij} \cap A) =k\} = \{Y_{ij}\in F_{A,k}\}$
%is measurable (that is, it is indeed an event).
%Hence $\H_0(Y_{i,j} \cap \cdot)$ 
%is a point process, so that  $ \sum_{1 \leq i < j < \infty}
%\H_0( Y_{ij} \cap  \cdot)$
%is also a point process. 
%By Lemma \ref{Ylem}.
Since  $ \chi = \sum_{1 \leq i < j < \infty}
\H_0( Y_{ij} \cap  \cdot)$, also $\chi$  is a point process.
The stationarity of $\chi$ is clear.
%\end{proof}
%
%
%\begin{proof}[Proof of Theorem \ref{thbranch} (b)]

\textcolor{\blue}{ Denote the intensity}
 of the stationary point process $\chi$ by 
 $\tbeta_3$.  By Lemma \ref{Ylem} and the multivariate Mecke formula,
 using notation  $E_{x,t} $ from Lemma \ref{Elem}, we have
\bean
\tbeta_3 = \E[\chi (W_1 )]  = \E \sum_{i <j} \H_0(Y_{ij} \cap W_1)
\\
= \int_{\cK}  \Q(d \sigma) 
 \int_{\cK}  \Q(d \sigma') \int_{\R^2} dy \int_{\R^2} dx \int_0^\infty dt \int_{0}^t ds 
\sum_{z \in (\partial \sigma +y) \cap (\partial \sigma' +x) }
{\bf 1}_{W_1} (z) \Pr[ E_{z,t} ]. 
\eean
Taking $x'=x-y$ and $z'= z-y$, 
using Lemma \ref{Elem} 
we have
\bean
\tbeta_3 
= \int_{\cK}  \Q(d \sigma) 
 \int_{\cK}  \Q(d \sigma') \int_{\R^2} dy \int_{\R^2} dx'
\sum_{z' \in \partial \sigma \cap (\partial \sigma' +x') }
{\bf 1}_{W_1} (y+ z') 
\times  \int_0^\infty t e^{-\lambda t}  dt,  
\eean
and taking the $y$-integral inside the $x'$-integral and the sum,
we obtain that
%  $x'=x-y$ and $z'= z-y$ we have
\bean
\tbeta_3 
=
\lambda^{-2}
 \int_{\cK}  \int_{\cK}
 \int_{\R^2}
\H_0(\partial \sigma \cap (\partial \sigma'+x')) 
dx' 
 \Q(d\sigma) \Q(d\sigma'),  
\eean
and hence $\tbeta_3 = \beta_3$, as asserted. 
\end{proof}

\begin{proof}[Proof of \textcolor{\blue}{Theorem \ref{thbranch} (b)}]
Assume now that $\Q$ is rotation-invariant.
 Then $\rho_\Theta (S')   \eqd S'$,
so by (\ref{0628c}) we have 
\bean
 \lambda^{2}
\beta_3 
=
 \E  \int_{\R^2} \H_0(\partial S \cap  (\partial S'  + x) )  dx 
=
 \E  \int_{\R^2} \H_0(\partial S \cap  (\rho_\Theta(\partial S')  + x) )  dx 
\\
= 
 (2 \pi)^{-1} \E  
\int_{-\pi}^\pi 
\int_{\R^2}
\H_0( \partial S \cap (\rho_\theta(\partial S' ) +x )) 
 dx d \theta,
\eean
 and hence by the `two noodle' formula  (Lemma \ref{lemBuffon}),
\bean
\lambda^2
 \beta_3
 =
% \lambda^{-2} 
(2/\pi) \E[ \H_1(\partial S) \H_1(\partial S') ],  
\eean
which yields (\ref{0629a}).
\end{proof}

\textcolor{\blue}{
We now work towards proving Theorem \ref{thmconnect}.
Recall from (\ref{Phidef}) that $\Phi$ denotes  the boundary
of our time-reversed DLM tessellation. It is helpful to represent
$\Phi$ in terms of the following sets.
For each $i$ with $t_i \geq 0$,  define 
 the set
\bea
P_i := (S_i+ x_i) \setminus \cup_{\{j:0 \leq t_j < t_i\}} (S_j+ x_j).  
\label{patcheq}
\eea
(Here the $P$ stands for `patch'.})  Set $P_i=\emptyset$ if $t_i <0$. 
\begin{lemm}
\label{lempatrep}
\textcolor{\blue}{
Under the assumptions of Theorem \ref{thmconnect},
 almost surely
$\Phi = \cup_{i=1}^\infty \partial P_i$.
}
\end{lemm}
\begin{proof}
Let $y \in \Phi$. Then by (\ref{Phidef}), 
$y$ lies on the boundary of some leaf that arrives at a non-negative
time; let $i$ be the index of
the earliest-arriving  leaf  (at or after time 0)
 that contains $y$ in its boundary.
Then  by (\ref{Phidef}) again, $y$ does not lie in (either the
interior or the boundary of)
any leaf arriving between times 0 and $t_i$, so $y \in P_i$,
and since $y \in \partial S_i + x_i$, moreover $y \in \partial P_i$.
Thus $\Phi \subset \cup_{i=1}^\infty \partial P_i$.

Conversely, let $j \in \N$ be such that $P_j \neq \emptyset$, and let
 $z \in \partial P_j$. Then
$z \in S_j + x_j$ (since $S_j$ is closed). Also
$z \notin S_k^o + x_k$ for all $k \in \N$ with $0 \leq t_k < t_j$ (else
some neighbourhood of $z$ is disjoint from $P_j$).
 If $z \in \partial S_j + x_j$, then $z \in \Phi$ by (\ref{Phidef}).
If $z \in S_j^o + x_j$, then (since $z \in \partial P_j$) there exists
some $k$ with $0 \leq t_k < t_j$ and $z \in \partial S_k+ x_k$. Hence, again
$z \in \Phi$ by (\ref{Phidef}).
Thus $\cup_{i=1}^\infty \partial P_i \subset \Phi$.
\end{proof}

\begin{lemm}
\label{lempatch}
\textcolor{\blue}{
Let $A \subset \R^2$ be bounded.
Under the assumptions of Theorem \ref{thmconnect},
it is almost surely the case that
 (a) each component of $\R^2 \setminus \Phi$ is  contained in
the interior of
 one of the
 patches $P_i$, and (b)
the union of all components
of $\R^2 \setminus \Phi$ that intersect $A$ is a bounded set.
}
\end{lemm}
\begin{proof}
Suppose $y \in \R^2 \setminus \Phi$. Then $y$ must lie in
the interior of the first-arriving leaf (after time 0) to 
contain $y$. Suppose this leaf has index $i$. Then $y \in P_i^o$,
and since $\partial P_i \subset \Phi$ by Lemma 
\ref{lempatrep}, the component of 
 $\R^2 \setminus \Phi$ containing $y$ is contained in
$P_i^o$. This gives us part (a).
Moreover the patches $P_i$ are almost surely bounded
sets.
Therefore, for part (b)
  it suffices to prove
that the number of patches $P_i$  which intersect 
$[0,1]^2$ is almost surely
finite. 

The number  of  leaves $S_i+x_i$  
having non-empty intersection with $[0,1]^2$ and arrival time $t_i \in [0,1]$
is Poisson with intensity $\lambda_0$ given by
% the integral
\bean
\lambda_0 := \int_{\R^2} \Pr[ (S+ x) \cap [0,1]^2 \neq \emptyset ] 
\leq \int_{\R^2} \Pr[ R \geq  \|x\| - 2 ] dx,
\eean
with $R$ is given by (\ref{Rdef}). Thus $\lambda_0 < \infty$
 since we assume $\E[R^2] < \infty$.  Set
$$
\cI := \{i \in \N: (S_i + x_i) \cap  [0,1]^2 \neq \emptyset, t_i \geq 0\};
$$
then $\sum_{i \in \cI} \delta_{t_i} $ is a
 1-dimensional Poisson process of  intensity $\lambda_0$. 

Define $N$ as at (\ref{Ndef}); by a similar argument to the one given
in the proof of Lemma \ref{lem0207a2}, $N$ is almost surely finite.
That is, almost surely the square $[0,1]^2$ is completely covered within
a finite (random) time, denoted $T$ say. The number of $i$ for which 
$(S_i + x_i) \cap [0,1]^2 \neq \emptyset$ and
$0 \leq t_i \leq T$ is almost surely finite, and provides
an upper bound for the number of patches $P_i$ which intersect $[0,1]^2$,
so this number is also almost surely finite, as required.
\end{proof}
\begin{lemm}
\label{lemclosed}
\textcolor{\blue}{
Under the assumptions of Theorem \ref{thmconnect},
 it is almost surely the case that
% $\Phi$ is closed, and 
for every connected
component $Z$ of $\Phi$, the sets $\Phi$,  $Z$ and
 $\Phi \setminus Z$ are all closed.}
\end{lemm}
\begin{proof}
Let $Q \subset \R^2$ be compact. Then by Lemma \ref{lempatrep},
$$
\Phi \cap Q = \cup_{i=1}^\infty (\partial P_i \cap Q),
$$
and by the proof of Lemma \ref{lempatch}, the number
of patches $P_i$ which intersect $Q$ is almost surely finite. Therefore
$\Phi \cap Q$ is a finite union of closed sets, so is closed.
This holds for any compact $Q$, and hence $\Phi$ is also closed, almost surely. 

Now suppose $Z$ is a connected component of $\Phi$. It is easy
to see from the definition of a connected component
(see e.g. \cite{Rudin}) that any limit
point of $Z$ must be in $Z$, and therefore $Z$ is closed.

It remains to prove that $\Phi \setminus Z$ is closed.
Suppose this were not the case.
% Suppose $\Phi \setminus Z$ were not closed.
Then there would exist
 $z \in Z$, 
 and a sequence
 $(z_n)_{n \geq 1}$ 
of points in
$\Phi \setminus Z$, such that  
$z_n \to z$ as $n \to \infty$. Then $z$ lies on the boundary
of some leaf arriving after time 0; let $i$ be the index
of the earliest-arriving such leaf. Also let $k$ be the index
of the earliest-arriving leaf with $z$ in its interior.
Then without loss of generality we may assume that for all $n$ 
we have $z_n \in (S_k^o + x_k)$.

For each $n \in \N$, choose $j(n) $ such that $t_{j(n)} < t_{k}$ and
$z_n \in \partial S_{j(n)} + x_{j(n)}$. Since there are almost surely
only finitely many leaves arriving that intersect any given
 compact region between times 0 and
 $t_{k}$, the numbers $j(n)$ run through a finite set of
indices, and hence by taking a subsequence if necessary
we can assume there is a single index $j$ such that $j(n)=j$ for
all $n$. Since $\partial S_{j} + x_j$ is closed, we also have
$z \in \partial S_j + x_j$. By our choice of $i$ we then have
$t_i \leq t_j$. 

Suppose $t_i = t_j$;
% and 
 then $i=j$ and
% for all large enough $n$ we have
$z_n \in \partial S_i + x_i$ for all $n$. 
Hence for all large enough $n$,
there is a path from $z_n$ to $z$ along $\partial S_i + x_i$ that does 
not meet any leaf
arriving before time $t_i$, so this path lies in $\Phi$, and hence
% we have and in fact
 $z_n \in Z$, which is a contradiction.

Therefore we may assume that $t_i < t_j$, and
%  Then
% for all large enough $n$ we have
 $z_n \in \partial S_j + x_j$ for all $n$.
Then $z \in (\partial S_i + x_i ) \cap (\partial S_j + x_j)$.
 By Lemma \ref{intbdylem} (b),
 $z \notin S_{\ell} + x_\ell$ for all $\ell \in \N \setminus \{i,j\}$ with
$t_\ell < t_k$. Hence $z$ has a neighbourhood that is disjoint from
$\cup_{\{\ell \in \N \setminus \{i,j\}:t_\ell < t_k\}}
(S_\ell + x_\ell)$.
%we have $z \notin S_\ell$  for all large enough $n$ we have that
% $z_n \notin S_{\ell} + x_\ell$ for all $\ell \in \N \setminus \{i,j\}$ with
%$t_\ell < t_k$.
 Hence for all large enough $n$, there is a path
in $\partial S_j +x_j$ from $z_n$ to $z$ that does not intersect any
leaf arriving before $t_k$ other than possibly leaf $i$; then
by taking this path from $z_n$ as far as the first intersection with leaf $i$,
and then (if this intersection is not at $z$) concatenating it with
 a path from there
along $\partial S_i + x_i$ to $z$ we have a path in $\Phi$ from 
$z_n$ to $ z$, and therefore also  $z_n \in Z$, which is a contradiction.
Thus $\Phi \setminus Z$ must be closed, as claimed.
\end{proof}
\textcolor{\blue}{
In the next two proofs, we shall
use the fact that
$\R^2$ is {\em unicoherent.} The unicoherence property says that
 for any 
two closed connected sets in $\R^2$ having union $\R^2$,
the intersection of these two sets is connected. See e.g.
\cite[page 177]{RGG}, or
\cite{Dugundji}.}

\begin{lemm}
\label{lembounded}
\textcolor{\blue}{
Under the assumptions of Theorem \ref{thmconnect},
all connected components of the set $\Phi$ are unbounded, almost surely.
}
\end{lemm}
\begin{proof}
  Suppose that $\Phi$ 
has at least one bounded component, pick one of these
bounded components, and denote this component by
$Z$. Given $\eps >0$, let $Z_\eps$ denote the
closed $\eps$-neighbourhood of $\Phi$, that is, the set of $x \in \R^d$
such that $\|x-y\| \leq \eps$ for some $y \in Z$. 
By Lemma \ref{lemclosed}, $Z$ is compact and $\Phi \setminus Z$
is closed. Hence we can and do choose 
 $\eps >0$ such that 
\bea
Z_{\eps} \cap ( \Phi \setminus Z) = \emptyset.
\label{phi2epseq}
\eea 
%This justifies the claim at (\ref{phi2epseq}).
%
% 
%Choose $\eps >0$ so that (\ref{phi2epseq}) holds.
%
Denote by $V_{\eps}$ 
the unique unbounded connected component
of $\R^2 \setminus Z_{\eps}$.
 The set $\R^2 \setminus V_\eps$ is connected; we can think of it
as `$Z_\eps$ with the holes filled in'.
 Then $\R^2 \setminus V_{ \eps}$, 
and $\overline{V_{\eps}}$, 
are connected closed sets
with union $\R^2$, so by unicoherence their intersection, which is simply
$\partial V_\eps$,
is a connected  set (a kind of loop surrounding $Z$).
%To make this strip into an open set, let $D$ be the open $\eps$-neighbourhood 
%of $(\R^2 \setminus \Psi_{3 \eps}) \cap \overline{\Psi}_{2 \eps}$.
%
Moreover $\partial V_\eps \cap \Phi = \emptyset$ by (\ref{phi2epseq}),
since every element of $\partial V_\eps$ is distant $ \eps$ 
from $Z$.  Hence $\partial V_\eps$ is contained
in a single component of $\R^2 \setminus \Phi$, so
by Lemma \ref{lempatch} (a),
there exists $j_0 \in \N$ such that $\partial V_\eps \subset P^o_{j_0}$.
% $\partial V_\eps$  is contained in a single set of the
%the form $P_j^o$,  as given by (\ref{patcheq}), say for $j=j_0$.
 Then by the assumed Jordan property, and the Jordan curve theorem,
$\partial V_\eps$ is surrounded by the boundary of $S_{j_0} + x_{j_0}$.  
Also $Z$ is contained in a bounded component of
 $\R^2 \setminus \partial V_\eps$.

%Given $\eps >0$, let $\Psi_\eps$ denote 
% the intersection of the $\eps$-neighbourhood of $\partial_{\rm ext}
% \Phi_0$ with $\Psi$.
%By a compactness argument, we can and do choose $\eps >0$ so that
%$\Psi_\eps$
% is connected, and is moreover contained in
%$\Phi^c$.

Pick $x \in Z$. Then $x$ must lie on the boundary of
some leaf $i$ arriving before leaf $j_0$ (in the
time-reversed DLM), since otherwise $x$ would be in the interior
of $P_{j_0}$ and not in $\Phi$ at all. Therefore there is some $i$
such that $0 \leq t_i < t_{j_0}$ and the leaf boundary
 $ \partial S_i + x_i$ includes a point
in $Z$. Let $i_0$ be the index of the first-arriving such leaf.

By the Jordan property
 the leaf $S_{i_0} + x_{i_0}$ is connected, 
%and bounded, 
and it does not intersect $\partial V_\eps$,
since $\partial V_\eps \subset P_{j_0}$ and $t_{i_0} < t_{j_0}$.
Therefore it is contained in a bounded component
of $\R^2 \setminus \partial V_\eps$.  
 Hence, this leaf is entirely  surrounded by
the boundary $ \partial S_{j_0} + x_{j_0}$.  
%
%Then $x_j + S_j$ is the first-arriving set of the time-reversed
%DLM, at all points of $\Psi_\eps$.
%There must be a set of the form $x_i +S_i$ inside $\R^2 \setminus
%\Psi$, with $t_i < t_j$ (else $\Phi_0$ would be empty).  Then
%$x_i + S_i$ does not intersect the set $\Phi_\eps$.
 Thus in this case there would exist distinct  $i_0,j_0 \in \N$
such that $S_{i_0} +x_{i_0} \subset S_{j_0} + x_{j_0}$, and the expected
number of such pairs is zero by 
our non-containment assumption, and a similar argument
using the Mecke formula to the proof of Lemma
\ref{notouchlem}. Thus  $\Phi$ almost surely has no bounded component.
\end{proof}

%\bgein
%\begin{lemm}
%\label{lemphicon}[
%\textcolor{\blue}{
%Under the assumptions of Theorem \ref{meantheo2},
%the set $\Phi$ is connected, almost surely.
%}
%\end{lemm}
\begin{proof}[Proof of Theorem \ref{thmconnect}]
%The set $\Phi$ is almost surely unbounded, since all connected
%components of its complement are almost surely bounded by Lemma \ref{lempatch}.
%If $\Phi$ is  not  connected,  then
 %either it has a bounded
%connected  component, or it has at least two unbounded components.
%The first of these possibilities is ruled out by Lemma \ref{lembounded}.
%
Suppose that $\Phi$ has at least two  unbounded components.
Pick two unbounded components of $\Phi$, and denote them by
$Z_0$ and $Z_1$. 
By Lemma \ref{lemclosed}, both $Z_0$ and $\Phi \setminus Z_0$ are
closed. By Urysohn's lemma, we can (and do) pick a continuous function
$g: \R^2 \to [0,1]$ 
 taking the value $0$ for all $x \in Z_0$ and $1$ for
all $x \in \Phi \setminus Z_0$. For example, we could take
%$
%g(x) := \dist(x,Z_0)/(\dist(x,Z_0) + \dist(x,\Phi \setminus Z_0) ).
%$
%
% the denominator is always non-zero so
%$g(x)$ is well-defined.
%Define the function
%for $0 < \eps < 1$,
 %define the set 
%$ \Phi_\eps := \{x \in \R^2 : g(x) \leq \eps\}$, where we set
$$
g(x) := \frac{\dist(x,Z_0)}{\dist(x,Z_0) +
\dist(x,\Phi \setminus Z_0) }, 
~~~~~ x \in \R^2.
$$
%where for each $x \in \R^2 $ and non-empty $A \subset \R^2$,
%we set $\dist(x,A):= \inf_{y \in A} \|x-y\|$.
%Also $g(x)$ is a continuous function of $x$,
%Moreover, by Lemma \ref{lemclosed}, we have
%$0 < g(x) <1$ for all $x \in \R^2 \setminus \Phi$.

Define the set $F := \{x \in \R^2 : g(x) \leq 1/2\}$.
Then $F \cap ( \Phi \setminus Z_0) = \emptyset$, and $F$ 
is a closed subset of $\R^2$. 
Let $V$ denote the 
component of $\R^2 \setminus F$ containing $Z_1$.
Then $ \R^2 \setminus V$, and $\overline{V}$, are closed
connected sets with union $\R^2$, so 
%Set 
%%$\partial_{\rm ext}F := 
%$ \partial
%  \overline{V} \setminus V $. 
by the unicoherence of $\R^2$, their intersection, which is
$\partial V$, 
is connected;
moreover, 
%$\partial_{\rm ext} F$  
$\partial V \subset \partial F$,
%so $g(x) = 1/2$ for all $x \in \partial_{\rm ext} F$
so $g(x) = 1/2$ for all $x \in \partial V$,
and hence
%$ \partial_{\rm ext} F \cap \Phi = \emptyset$.
$ (\partial V) \cap \Phi = \emptyset$.

By Lemma \ref{lempatch} (b),  all components of
$\R^2 \setminus \Phi$ are bounded, almost surely.  Since $\partial V  $
 is connected, and disjoint from $\Phi$, it is contained
in a single component of $\R^2 \setminus \Phi$, and
therefore $\partial V$ is bounded.

We now show, however that the set $\partial V$ is 
 {\em unbounded}, which is a contradiction. 
Let $r>0 $ and recall that $B(r): = \{y \in \R^2:\|y\| \leq r\}$.
Since $Z_0$ and $Z_1$ are unbounded, we can  pick points
$z_0 \in Z_0 \setminus B(r)$ and
$z_1 \in Z_1 \setminus B(r)$.
We may then take a polygonal path in $\R^2 \setminus B(r)$ 
 from $z_0$ to $z_1$.
The last point $z$ in this path 
for which $g(z) = 1/2$, lies in $\partial V$.
Hence, $\partial V  \setminus B(r)$
is non-empty. Since $r$ is arbitrary, $\partial V$ is
unbounded.

%
%Then the set $\partial_{\rm ext} \Phi_0$ is a curve of infinite length.
%For arbitrarily $K>0$, we can take an arc in $\partial_{\rm ext} \Phi_0$
%of diameter greater than $K$.
%Then setting $\Psi_{\eps,K}$ to be the intersection of the $\eps$-neighbourhood
%of this arc with $\Psi$, we have for small enough $\eps >0$ that
%$\Psi_\eps$  is connected of diameter $K$, and contained in the complement
%of $\Phi$. This, however, contradicts the fact that all components
%of $\R^2 \setminus \Phi$ are bounded. 

We have proved by contradiction that
 $\Phi$ almost surely has at most one unbounded component. Combined
with Lemma \ref{lembounded},
%the previous argument this justifies our claim that
this shows that  it is
almost surely connected.
\end{proof}

\begin{proof}[Proof of Theorem \ref{meantheo2}]
We now view the random set $\Phi$ (the boundaries of the
DLM tessellation) as a planar graph. By Lemmas 
\ref{intbdylem} and  \ref{notouchlem},
 there are 
no vertices of degree 4 or more in this graph.

The planar graph $\Phi$ has no vertices
of degree 1, by the Jordan assumption.
Thus, 
we may view
% the random set 
$\Phi$
% (the boundaries of the DLM tessellation) 
as a planar graph with
 all of its vertices having degree 3,  
 and
$\chi$ is the point process of these  vertices.

\textcolor{\blue}{
Moreover, by Theorem \ref{thmconnect}
this planar graph is almost surely connected.}
Let $\tau$  denote the intensity 
of the point process of midpoints of edges in this planar graph.
By the handshaking lemma,
 $\tau = 3 \beta_3/2$.
Also, by an argument based on Euler's formula
(see \cite[eqn. (10.3.1)]{SKM}),
$\beta_1= \tau - \beta_3 = \beta_3/2$, as asserted. 
In the rotation-invariant
case, by (\ref{0629a}) we have (\ref{0320a}).
\end{proof}

\begin{proof}[Proof of Theorem \ref{thlengthcov}]
Apply Theorem \ref{thcovgen}, 
using the same 
choice of $\Q'$ as in the first part of the proof
of Theorem \ref{meantheo}.
\end{proof}

\begin{proof}[Proof of Theorem \ref{CLTa}]
We apply Theorem \ref{th:CLTgen},
using the same 
choice of $\Q'$ as in the first part of the proof
of Theorem \ref{meantheo}.
\end{proof}

{\bf Acknowledgements.}
We thank Pieter Trapman, G\"unter Last and Michael Klatt for
conversations in 2016 which stimulated our interest in this
model.
\textcolor{\blue}{
  We also thank Daniel Hug, Dominique Jeulin and
(especially) an anonymous referee for numerous helpful
comments on  earlier versions of this paper.}

\end{document}